\pdfoutput=1
\documentclass[11pt]{article}
\usepackage{lipsum}

\newcommand\blfootnotea[1]{%
  \begingroup
  \renewcommand\thefootnote{}\footnote{#1}%
  \endgroup
}

\def\colorful{0}

\ifnum\colorful=1
\newcommand{\new}[1]{{\blue #1}}
\newcommand{\anote}[1]{\footnote{{\bf [Ankit: {#1} ]}}}
\newcommand{\pnote}[1]{\footnote{{\bf [Po-Ling: {#1} ]}}}
\newcommand{\vnote}[1]{\footnote{{\bf [Varun: {#1} ]}}}
\else
\newcommand{\new}[1]{{#1}}
\newcommand{\anote}[1]{}
\newcommand{\pnote}[1]{}
\newcommand{\vnote}[1]{}
\fi

\usepackage[utf8]{inputenc} %
\usepackage[T1]{fontenc}    %
\usepackage{url}            %
\usepackage{booktabs}       %
\usepackage{nicefrac}       %
\usepackage{microtype}      %

\usepackage{amsthm,amsfonts,amsmath,amssymb,epsfig,color,float,graphicx,verbatim, enumitem}

\usepackage{algpseudocode,algorithm,algorithmicx}  

\usepackage{bbm}
\usepackage{caption}

\usepackage[dvipsnames,table,xcdraw]{xcolor}

\usepackage[colorlinks,citecolor=RoyalBlue,linkcolor=RubineRed,bookmarks=true,pagebackref]{hyperref}

\usepackage[nameinlink,capitalise]{cleveref}

\usepackage{thm-restate}

\usepackage{mleftright}

\def\E{\mathbb E}
\def\P{\mathbb P}
\def\R{\mathbb R}
\def\I{\mathbb I}

\def\N{\mathbb N}

\let\vec\mathbf

\newcommand{\bG}{\vec{G}}
\newcommand{\bH}{\vec{H}}

\newcommand{\bJ}{\vec{J}}

\newcommand{\bT}{\vec{T}}

\newcommand{\poly}{\mathrm{poly}}

\newcommand{\cN}{\mathcal{N}}

\newcommand{\cT}{\mathcal{T}}
\newcommand{\cE}{\mathcal{E}}
\newcommand{\cZ}{\mathcal{Z}}

\newcommand{\cP}{\mathcal{P}}
\newcommand{\cY}{\mathcal{Y}}
\newcommand{\cQ}{\mathcal{Q}}
\newcommand{\cH}{\mathcal{H}}

\newcommand{\cR}{\mathcal{R}}
\newcommand{\cS}{\mathcal{S}}
\newcommand{\cA}{\mathcal{A}}
\newcommand{\cC}{\mathcal{C}}
\newcommand{\cX}{\mathcal{X}}

\newcommand{\cB}{\mathcal{B}}
\newcommand{\dtv}{d_{\mathrm{TV}}}
\newcommand{\hel}{d_{\mathrm{h}}}
\newcommand{\haff}{\beta_{\mathrm{h}}}

\newcommand{\cTT}{\mathcal{T}^\textup{\texttt{thresh}}}

\newcommand{\SDA}{{\mathrm{SDA}}}
\newcommand{\nstar}{n^{\!\ast}}
\newcommand{\nide}{{n^{\!\ast}_{\textup{\texttt{identical}}}}}
\newcommand{\robust}{\textup{\texttt{robust}}}

\newcommand{\nnid}{{n^{\!\ast}_{\textup{\texttt{non-identical}}}}}
\newcommand{\nada}{{n^{\!\ast}_{\textup{\texttt{adaptive}}}}}

\newcommand{\tstar}{{\bT_{*}}}

\newcommand{\opt}{{\mathrm{OPT}}}
\newcommand{\VSTAT}{{\mbox{VSTAT}}}
\newcommand{\STAT}{{\mbox{STAT}}}
\newcommand{\MSAMPLE}{{\mbox{1-MSTAT}}}

\newcommand{\ic}{\textrm{IC}}
\newcommand{\mic}{\textup{min-IC}}

\newcommand{\Var}{\operatorname{Var}}

\newcommand{\real}{\ensuremath{\mathbb{R}}}

\crefname{equation}{Equation}{Equations}
\crefname{lemma}{Lemma}{Lemmata}
\crefname{claim}{Claim}{Claims}
\crefname{fact}{Fact}{Facts}
\crefname{theorem}{Theorem}{Theorems}
\crefname{proposition}{Proposition}{Propositions}
\crefname{corollary}{Corollary}{Corollaries}
\crefname{remark}{Remark}{Remarks}
\crefname{definition}{Definition}{Definitions}
\crefname{question}{Question}{Questions}

\newtheorem{theorem}{Theorem}[section]
\newtheorem{lemma}[theorem]{Lemma}

\newtheorem{claim}[theorem]{Claim}
\newtheorem{proposition}[theorem]{Proposition}
\newtheorem{corollary}[theorem]{Corollary}
\newtheorem{fact}[theorem]{Fact}

\theoremstyle{definition}
\newtheorem{definition}[theorem]{Definition}
\newtheorem{example}[theorem]{Example}

\newtheorem{problem}[theorem]{Problem}

\newtheorem{remark}[theorem]{Remark}

\allowdisplaybreaks

\theoremstyle{definition}

\usepackage{color}
\definecolor{Red}{rgb}{1,0,0}
\definecolor{Blue}{rgb}{0,0,1}
\definecolor{DGreen}{rgb}{0,0.55,0}
\definecolor{Purple}{rgb}{.75,0,.25}

\renewcommand{\left}{\mleft}
\renewcommand{\right}{\mright}

\title{Communication-constrained hypothesis testing:\\
 Optimality, robustness, and reverse data processing inequalities\blfootnotea{This paper was presented in part at ISIT 2022.}}
\author{
Ankit Pensia\thanks{Supported by NSF grant CCF-1841190, NSF grant CCF-2011255, and the Department of Pure Mathematics and Mathematical Statistics at the University of Cambridge.}
\\
IBM Research\\
{\tt ankitp@ibm.com}\\
\and
Varun Jog\\
University of Cambridge\\
{\tt vj270@cam.ac.uk}
\and
Po-Ling Loh\\ University of Cambridge\\
{\tt pll28@cam.ac.uk}
}

\begin{document}

\maketitle
\begin{abstract}
We study hypothesis testing under communication constraints, where each sample is quantized before being revealed to a statistician. Without communication constraints, it is well known that the sample complexity of simple binary hypothesis testing is characterized by the Hellinger distance between the distributions.
We show that the sample complexity of simple binary hypothesis testing under communication constraints is at most a logarithmic factor larger than in the unconstrained setting and this bound is tight. We develop a polynomial-time algorithm that achieves the aforementioned sample complexity.
Our framework extends to robust hypothesis testing, where the distributions are corrupted in total variation distance. Our proofs rely on a new reverse data processing inequality and a reverse Markov inequality, which may be of independent interest. For simple $M$-ary hypothesis testing, the sample complexity in the absence of communication constraints has a logarithmic dependence on $M$. 
We show that communication constraints can cause an exponential blow-up, leading to $\Omega(M)$ sample complexity even for adaptive algorithms.

\end{abstract}

\section{Introduction} %
\label{sec:introduction}

Statistical inference has been extensively studied under constraints such as memory~\cite{Cov69, HelCov73, HelCov73a,GarRT18,BerOS20,DiaKPP22-streaming}, privacy~\cite{DworkRoth13,KaiOV16,DucJW18,CanKMSU19,GopKKNWZ20}, communication~\cite{Tsitsiklis93,BravGMNW16,HanOW21-t.inf.theory,AchCT20-II}, or a combination thereof~\cite{SteVW16,Feldman17,DaganShamir18,DiaGKR19,DucRog19,AchCT20-I}, typically designed to model physical or economic constraints. Our work focuses on communication constraints, where the statistician does not have access to the original samples---but only their quantized versions---generated through a communication-constrained channel. 
For example, instead of observing a sample $x \in \cX$, the statistician might observe a single bit $f(x) \in \{0,1\}$, for some function $f: \cX \to \{0,1\}$. The choice of the channel (here, the function $f$) crucially affects the quality of statistical inference and is the topic of study in our paper.

Under communication constraints, a recent line of work has established minimax optimal rates for a variety of problems, including distribution estimation and identity testing~\cite{Sha14,AchCT20-I,AchCT20-II,HanOW21-t.inf.theory,CheKO21,Canonne22}. However, under the same constraints, the problem of simple hypothesis testing has received scant attention.

Recall the simple hypothesis testing framework: Let $\cP$ be a finite set of distributions over the domain $\cX$. Given i.i.d.\ samples $X_1,\dots,X_n$ from an unknown distribution $p \in \cP$, 
the goal is to correctly identify $p$ with high probability (say, with probability at least $0.9$), with $n$ as small as possible. We denote this problem as $\cB(\cP)$ and use $\nstar(\cP)$ to denote its sample complexity; i.e., the minimum number of samples required to solve $\cB(\cP)$.

When $\cP = \{p, q\}$, the problem is referred to as the simple binary hypothesis testing problem and has a rich history in statistics~\cite{NeyPea33,Wald45,HubStr73,Cam86}.
Given its historical and practical significance, we have a deep understanding of this problem (cf.\ \Cref{sec:prelim} for details).
In particular, it is known that  $\nstar(\cP) = \Theta(1/ \hel^2(p,q))$, where $\hel(p,q)$
denotes the Hellinger distance between $p$ and $q$.

Hypothesis testing under communication constraints was studied in detail in the 1980s and 1990s under the name ``decentralized detection''~\cite{Tsitsiklis93}.
Briefly, the setup involves $n$ users and a central server. Each user $i$ observes an i.i.d.\ sample $X_i$ from an unknown distribution $p \in \cP$,
 generates a message $Y_i \in \{0,1,\dots,D-1\}$ using a channel $\bT_i$ (chosen by the statistician), and transmits $Y_i$ to the central server.  The central server observes $(Y_1,\dots,Y_n)$ and produces an estimate $\widehat{p} \in \cP$. 
The goal is to choose $(\bT_1,\dots,\bT_n)$ so that the central server can identify $p$ correctly with high probability, while keeping $n$ as small as possible.  
We call this problem ``simple hypothesis testing under communication constraints'' and denote it by $\cB(\cP,D)$. We denote the corresponding sample complexity by $\nstar(\cP,D)$.

\paragraph{Simple binary hypothesis testing} 
 \label{par:binary_hypothesis_testing}
We begin our discussion with the fundamental setting of simple binary hypothesis testing under communication constraints, i.e., $\cP = \{p,q\}$.
It is known that the central server should perform a likelihood ratio test~\cite{Tsitsiklis93}. 
Furthermore, an optimal choice of channels can be achieved using deterministic threshold tests; i.e., $Y_i = f_i(X_i)$ for some $f_i: \cX \to \{0,1,\dots,D-1\}$, such that $f_i$ is characterized by $D$ intervals that partition $\R_+$, and  $f_i(x) = j$ if and only if $p(x)/q(x)$ lies in the $j^{\text{th}}$ interval. 
The optimality of threshold tests crucially relies on the $f_i$'s being possibly non-identical across users~\cite{Tsitsiklis88}.

Nonetheless, several fundamental statistical and computational questions have remained unanswered.
We begin with the following statistical question:  
\begin{quotation}
\begin{center}
    \noindent \it  For $\cP = \{p,q\}$, what is the sample complexity of $\cB(\cP,D)$, and what is $\frac{\nstar(\cP,D)}{\nstar(\cP)}$?
\end{center}
\end{quotation}
Let $\nstar = \nstar(\cP)$ and $\nstar_{\text{bin}} = \nstar(\cP, 2)$ for notational convenience.
A folklore result using Scheffe's test (\Cref{def:scheffe}) implies that $\nstar_{\text{bin}} / \nstar \lesssim \nstar$ (cf. \cref{prop:Scheffe}).
One of our main results is an exponential improvement on this guarantee, showing that $\nstar_{\text{bin}} / \nstar \lesssim \log(\nstar)$, i.e.,  communication constraints only lead to at most a logarithmic increase in sample complexity.
More precisely, we show the following sample complexity bound:
\begin{align}
\label{eq:IntroBound}
\nstar(\cP,D) \lesssim \nstar(\cP) \max\left\{ 1, \frac{\log(\nstar(\cP))}{D} \right\}.
\end{align}
Furthermore, there exist cases where the bound~\eqref{eq:IntroBound} is tight (cf.\ \cref{thm:BinHypTestLowerBound}).  
The bound can further be improved when the support sizes of $p$ and $q$ are smaller than $\log(\nstar(\cP))$.

Turning to computational considerations, let $p$ and $q$ be distributions over $k$ elements. Although the optimality of threshold tests implies that each user can search over $k^{\Omega(D)}$ possible such channels, this is prohibitive for large $D$. Such an exponential-time barrier has been highlighted as a major computational bottleneck in decentralized detection~\cite{Tsitsiklis93}, leading to the following question:  
\begin{quotation}
\begin{center}
    \noindent \it  Is there a $\poly(k,D)$-time algorithm to compute channels $(\bT_1,\dots,\bT_n)$ that achieve the sample complexity bound~\eqref{eq:IntroBound}?  
\end{center}
\end{quotation}
We answer this question affirmatively by showing that it suffices to consider threshold tests parametrized by a single quantity (cf.\ equation~\eqref{eq:ThreshTestGeom}).
In fact, we show that it suffices to use an identical channel across the users (cf.\ \Cref{lem:IdentVsNonIdent}).

\paragraph{Robustness to model misspecification}
In many scenarios, it may be unreasonable to assume that the true distribution is either exactly $p$ or $q$, but rather that it is close to one of them in total variation distance.
 Let $\epsilon$ be the amount of corruption, so that the underlying distribution $p'$ belongs to $\cP_1 \cup \cP_2$, where $\cP_1 := \{ \tilde{p}: \dtv(p,\tilde{p}) \leq \epsilon \}$ and $\cP_2 := \{ \tilde{p}: \dtv(p,\tilde{q}) \leq \epsilon \}$, and $\dtv$ denotes the total variation distance.
Our goal is to design channels and a test such that, given samples from any distribution in $\cP_1$ (respectively $\cP_2$), we output $p$ (respectively $q$) with high probability.
Under the communication constraint of $D$ messages, we denote this problem by $\cB_\robust(\{p,q\},\epsilon,D)$.

As long as $\epsilon \lesssim \dtv(p,q) $, we can use Scheffe's test to solve $\cB_\robust(\{p,q\},\epsilon,2)$ with at most $O(1/ \dtv^2(p,q))$ samples. We may hope to improve upon this by using the optimal channel $\bT'$ for the uncontaminated hypothesis testing problem, $\cB(\{p,q\},D)$.
It is, however, unclear if $\bT'$ satisfies any robustness properties like the channel in Scheffe's test.  We show in \Cref{prop:robustness-for-free} that in the moderate contamination regime, when $\epsilon \lesssim \dtv^2(p,q)$ (up to logarithmic factors),  $\bT'$ solves $\cB_\robust(\{p,q\},\epsilon,D)$  with the same sample complexity as $\cB(\{p,q\},D)$. As a converse, we present cases where $\bT'$ is not robust to larger $\epsilon$. 

For the high contamination setting, we combine our technical results with the framework of ``least favorable distributions'' pioneered by Huber \cite{Huber65} and extended to the communication-constrained setting by Veeravalli, Basar, and Poor~\cite{VeeBP94}.
We show that the robust sample complexity under communication constraints, $\nstar_\robust(\{p,q\},\epsilon,D)$, increases by at most a logarithmic factor. Specifically, letting $\nstar_\robust:= \nstar_\robust(\{p,q\},\epsilon)$ be the robust sample complexity without any communication constraints, we obtain the following result in \Cref{thm:ub-Simple-D-robust}:
\begin{align}
\label{eq:IntroBoundRobust}
\nstar_\robust(\{p,q\},\epsilon, D) \lesssim \nstar_\robust \max\left\{ 1, \frac{\log\left(\nstar_\robust\right)}{D} \right\}.
\end{align}
This rate can be much tighter than the one obtained using Scheffe's test. Moreover, the rate above is achieved by a computationally efficient algorithm.

\paragraph{$M$-ary hypothesis testing} 
Finally, we consider the setting where $\cP$ contains $M > 2$ distributions and allow the choice of channels to be adaptive, i.e., the channel $\bT_i$ may depend on $Y_1,\dots,Y_{i-1}$. 
For simplicity, we consider the setting where $D$ and $\cP$ are fixed and focus on the dependence on $M$.
Using a standard tournament procedure, we show that there is an adaptive algorithm with sample complexity $O(M \log M)$.
On the other hand, in the absence of communication constraints, it is known that the sample complexity is $O(\log M)$.
We show that this exponential blow-up is necessary using the techniques from Braverman, Garg, Ma, Nguyen, and Woodruff~\cite{BravGMNW16}, i.e., the sample complexity under communication constraints is $\Omega(M)$. We also show $\Omega(\sqrt{M})$ lower bounds using two other techniques: (i) statistical query lower bounds~\cite{SteVW16,FelGRVX17}, and (ii) the impossibility of $\ell_1$-embedding~\cite{ChaSah02,LeeMN05}. Although these bounds are weaker than the $\Omega(M)$ lower bound, they have other favorable properties: the support size of the distributions in the hard instance is much smaller ($k$ is linear in $M$ as opposed to exponential in $M$), and the technical arguments that rely on the impossibility of $\ell_1$-embeddings are elementary.
Lastly, we consider the setting where all of the channels are restricted to be identical across users, which may be desirable in some applications. We provide specialized upper and lower bounds in this setting.

\paragraph{Our contributions}
We summarize our main contributions as follows:
\begin{enumerate}
  \item (Simple binary hypothesis testing.) We establish the minimax optimal sample complexity (cf.\ inequality~\eqref{eq:IntroBound}) of binary simple hypothesis testing under communication constraints (\Cref{thm:ub-Simple-D,thm:BinHypTestLowerBound}). Moreover, we provide an efficient algorithm, running in $\poly(k,D)$ time, to find a channel that achieves the minimax optimal sample complexity.
 
\item (Robust version of simple binary hypothesis testing.) \Cref{thm:ub-Simple-D-robust} focuses on the robust hypothesis testing problem and shows that the robust sample complexity increases by at most a logarithmic factor, which is achievable using a computationally-efficient algorithm.
  
  \item ($M$-ary hypothesis testing.) Generalizing to the setting of $M$-ary distributions, we show that for some cases, communication constraints can lead to an exponential increase in sample complexity, even for adaptive channels. 
  We also derive results, both upper and lower bounds, specialized to settings where the channels are restricted (cf.\ \Cref{sec:m-ary}).
 
  \item (Technical results.) Along the way, we prove the following two technical results which may be of independent interest: (i) a reverse data processing inequality for general $f$-divergences and communication-constrained channels (\cref{thm:quant-scheme}), and (ii) a reverse Markov inequality for bounded random variables (\cref{lem:revMarkovBinary}).
\end{enumerate}
The remainder of the paper is organized as follows: \Cref{sec:prelim} defines notation, states the problem, and recalls useful facts.
\Cref{sec:RevDataProc} contains a reverse data processing inequality for $f$-divergences.
\Cref{sec:BinHypTest} uses these inequalities to derive our statistical and computational guarantees for binary hypothesis testing.
Finally, \Cref{sec:m-ary} presents results for $M$-ary hypothesis testing.
More technical proofs are deferred to the supplementary appendices.

\section{Preliminaries} 
\label{sec:prelim}

\paragraph{Notation} 
\label{par:notation}
Throughout this paper, we will focus on discrete distributions. For $n \in \N$, we use $[n]$ to denote $\{1,\dots,n\}$ and $[0:n]$ to denote $\{0,1,\dots,n\}$.
 We use $\Delta_k$ to denote the set of distributions over $k$ elements.
 For a distribution $p \in \Delta_k$ and an index $i \in [k]$, we use both $p_i$ and $p(i)$ to denote the probability of element $i$ under $p$.
Given two distributions $p$ and $q$, let $\dtv(p,q)$ and $\hel(p,q) := \sqrt{\sum_{i}(\sqrt{p_i} - \sqrt{q_i})^2}$ denote the total variation and Hellinger distances between $p$ and $q$, respectively.
Let $\haff(p,q)$ denote the Hellinger affinity, i.e., $\haff(p,q) := 1 - 0.5 \hel^2(p,q)$.
Given $n$ distributions $p^{(1)},\dots,p^{(n)}$, we use $\prod_{i=1}^n p^{(i)}$ to denote their product distribution.
When each $p^{(i)} = p$, we use $p^{\otimes n}$ to denote the $n$-fold product distribution.
For a set $A \subseteq \cX$, we use $\I_A: \cX \to \{0,1\}$ to denote the indicator function of $A$.
We consider $[a,b)$ to be an empty set when $b \leq a$. For a channel $\bT : \cX \to \cY$ and a distribution $p$ over $\cX$, we use $\bT p$ to denote the distribution over $\cY$ when $X \sim p$ passes through the channel $\bT$.
As the channels between discrete distributions can be represented by column-stochastic matrices, we also use bold capital letters, such as $\bT$, to denote the corresponding matrices.
In particular, when $p$ is a distribution over $[k]$, represented as a vector in $\R^k$, and $\bT$ is a channel from $[k] \to [d]$, represented as a matrix $\bT \in \R^{d \times k}$, the output distribution $\bT p$ corresponds to the usual matrix-vector product.
We use $c,C,c',C'$, etc., to denote absolute positive constants, whose values might change from line to line, but with values which can be inferred by careful bookkeeping, while $c_1,C_1,c_2,C_2$, etc., are used to denote absolute positive constants that remain the same throughout the proof.

Finally, we use the following notations for simplicity: (i) $\lesssim$ and $\gtrsim$ to hide positive constants, (ii) the standard asymptotic notation $O(\cdot)$, $\Omega(\cdot)$, and  $\Theta (\cdot)$, and (iii) $\poly(\cdot)$ to denote a quantity that is polynomial in its arguments.
 
\subsection{Definitions and basic facts}
\label{par:relation_between_divergences}

\begin{definition}[$f$-divergence] For a convex function $f:\R_+ \to \R$ with $f(1)= 0$, we use $I_f(p,q)$ to denote the $f$-divergence between 
$p$ and $q$, defined as $I_f(p,q):= \sum_{i}q_i f\left( p_i/q_i \right)$.\footnote{We use the following conventions~\cite{Sason18}: $f(0) = \lim_{t \to 0^+}f(t)$, $0 f(0/0) = 0$, and for $a>0$, $0f(a/0) = a \lim_{u \to \infty} f(u)/u$.}

\end{definition}

We use the following facts:
\begin{fact}[Properties of divergences~\cite{Tsybakov09,ZivZakai73}]
\label{fact:div}
For any distributions $p,p^{(1)},\dots,p^{(n)}$ and $q,q^{(1)},\dots, q^{(n)}$ in $\Delta_{k}$:
\begin{enumerate}
  \item (Total variation and Hellinger distance.) $\dtv^2(p,q) \leq \hel^2(p,q) \leq 2\dtv(p,q)$.
  \item (Sub-additivity of total variation.) $\dtv\left(\prod_{i=1}^np^{(i)},\prod_{i=1}^nq^{(i)}\right) \leq \sum_{i=1}^n \dtv\left(p^{(i)},q^{(i)}\right)$.
  \item (Hellinger tensorization.)
  $\haff\left(\prod_{i=1}^n p^{(i)},\prod_{i=1}^n q ^{(i)}\right) = \prod_{i=1}^n \haff\left(p^{(i)},q^{(i)}\right)$.
  \item (Data processing.) For any channel $\bT$, $f$-divergence $I_f$, and a pair of distributions $(p,q)$, we have $I_f(\bT p, \bT q) \leq I_f(p,q)$.
\end{enumerate}
\end{fact}

We now define the simple hypothesis testing problem:
\begin{problem} [Simple $M$-ary hypothesis testing]
Given $\cP$, a set of $M$ distributions over $\cX$, we say a function (test) $\phi: \cup_{n=1}^\infty \cX^n \to \cP$ solves the simple $M$-ary hypothesis testing problem with sample complexity $n$ if 
\begin{align*}
\sum_{ p \in \cP} \P_{x \sim p^{\otimes n}} \left\{   \phi(x) \neq p\right\} \leq 0.1.
\end{align*}
We define the sample complexity of hypothesis testing to be the smallest $n$ such that there exists a test $\phi$ which solves the hypothesis testing problem with sample complexity $n$.
We use $\cB(\cP)$ to denote the simple $M$-ary hypothesis testing problem and $\nstar(\cP)$ to denote the sample complexity of $\cB(\cP)$.
\end{problem}

\begin{fact}[Hypothesis testing and divergences~\cite{Yatracos85,DevLug01,CanKMSU19,Wainwright19}]
\label{fact:testing}
We have the following:
\begin{enumerate}
\item (Total variation and binary hypothesis testing.)
For any random variable $Z$ over $\cZ$ and test $\phi: \cZ \to \left\{ P, Q \right\}$, define the probability of error to be $\P_{P}(\phi(Z) = Q) + \P_{Q}(\phi(Z) = P)$. The minimum probability of error over all tests is  $1 - \dtv(P,Q)$ and is achieved by the following test:
  let $A^* \subseteq \cZ$ be any set that maximizes $P(A) - Q(A)$ over $A \subseteq \cZ$, and define $\phi(z) = P$ when $z \in A^*$ and $\phi(z) = Q$ otherwise.

\item (Hellinger distance and $\cB(\{p,q\})$.)  The sample complexity for the simple binary hypothesis test between $p$ and $q$ is $\Theta\left(\frac{1}{\hel^2(p,q)}\right)$, i.e., $\nstar(\{p,q\}) = \Theta\left(\frac{1}{\hel^2(p,q)}\right)$.

\item (Sample complexity of $M$-ary hypothesis testing.) Let $\cP$ be a set of $M$ distributions such that $\min_{p,q \in \cP: p \neq q }\hel(p,q) = \rho$. Then $\frac{1}{\rho^2} \lesssim \nstar(\cP) \lesssim \frac{\log M}{\rho^2}$. 
\end{enumerate}
\end{fact}

We will use the following additional fact that states the dependence on failure probability for simple binary hypothesis testing:
\begin{fact}[Failure probability and sample complexity]
\label{fact:fail-prob-samp-complexity}
Let $\psi:\cup_{n=1}^\infty \cX^n \to \{p,q\}$ be the optimal likelihood ratio test for two distributions $p$ and $q$ with uniform prior.
If $n \gtrsim \frac{\log(1/\delta)}{\hel^2(p,q)}$ for $\delta \leq 0.1$, the failure probability of the test $\psi$ with $n$ samples is less than $\delta$.
\end{fact}
The fact above follows by the optimality of the likelihood ratio test and a boosting argument using the median.

We now define Scheffe's test, which is commonly used for simple binary hypothesis testing.
\begin{definition}[Scheffe's test]
\label{def:scheffe}
For two distributions $p$ and $q$, consider the set $A = \{x: p(x) \geq q(x)\}$. 
Let $p'$ and $q'$ denote the distributions of $\I_A(X)$ when $X$ is distributed as $p$ and $q$, respectively. 
Given $(x_1,\dots,x_n) \in \cX^n$, 
Scheffe's test transforms each individual point $x_i$ to $\I_A(x_i)$  and then applies the optimal test between $p'$ and $q'$ to the transformed points.\footnote{Note that $p'$ and $q'$ are Bernoulli distributions with probabilities of observing $1$ equal to $p(A)$ and $q(A)$, respectively. The optimal test between $p'$ and $q'$ corresponds to a threshold on  $\sum_i\I_A(x_i)$.}
\end{definition}
It is easy to see that $\dtv(p',q') = \dtv(p,q)$, which implies that $\hel(p',q') \geq 0.5 \hel^2(p,q)$ (using \cref{fact:div}), leading to an $O\left(\frac{1}{\hel^4(p,q)}\right)$ sample complexity of Scheffe's test. This dependence is tight~\cite{Suresh21}, and shown for completeness in \Cref{prop:Scheffe} below.

\begin{proposition}[Sample complexity of Scheffe's test (folklore)]
\label{prop:Scheffe}
The sample complexity of Scheffe's test is at most $O\left( 1/\hel^4(p,q) \right)$. Furthermore, this is tight in the following sense: For any $\rho \in (0,1)$, there exist $p$ and $q$ such that $\nstar(\{p,q\}) = O(1/\rho)$, whereas the sample complexity of Scheffe's test is $\Omega\left( 1/\rho^2\right)$.
\end{proposition}

\begin{proof}
We begin by showing the upper bound on sample complexity. Let $p$ and $q$ be the two given distributions and let $\rho = \hel^2(p,q)$.
Let $\bT$ be the channel corresponding to Scheffe's test.
Since Scheffe's test preserves the total variation distance, we have $\dtv(p,q) = \dtv(\bT p, \bT q)$.
By \cref{fact:div}, we have
\begin{equation*}
\hel(\bT p, \bT q) \geq \dtv(\bT p, \bT q) = \dtv(p,q) \geq 0.5 \hel^2(p,q) \geq 0.5 \rho.
\end{equation*}
Thus, by \cref{fact:testing}, the sample complexity is at most $O(1/\hel^2(\bT p, \bT q)) = O(1/\rho^2) = O(1/d_h^4(p,q))$.

We now turn our attention to the tightness of the upper bound. Without loss of generality, we consider the setting when $\rho \leq 0.01$. Consider the following two distributions on $\Delta_3$: $p = (\rho, 1/2 - 2\rho, 1/2+ \rho )$ and $q = (0,1/2,1/2)$.
Let $\bT$ be the channel corresponding to Scheffe's test. 
Then we have $\bT p = (1/2 + 2 \rho, 1/2 - 2 \rho)$ and $\bT q = (1/2,1/2)$.
An elementary calculation shows that $\hel^2(p,q) = \Theta(\rho)$ and $\hel^2(\bT p, \bT q) = \Theta(\rho^2)$.
Applying \cref{fact:testing}, we obtain the desired conclusion.
\end{proof}

\subsection{Simple hypothesis testing under communication constraints}
\label{sec:prelim-commConstraint}

Let $\cX$ be the domain, $\cP$ a family of distributions over $\cX$, and $\cT$ a family of channels from $\cX$ to $\cY$. Let $\cT_D$ denote the set of all channels from $\cX$ to $[0:D-1$]. We first formally define the problem of simple hypothesis testing under communication constraints.

\begin{definition} [Simple hypothesis testing under communication constraints]
\label{def:simpleHypothesisTestComm}
Let  $\{U_i\}_{i=1}^n$ denote a set of $n$ users who choose channels $\{\bT_i\}_{i=1}^n \subseteq \cT$ according to a rule $\cR: [n] \to \cT$.
Each user $U_i$ then observes a random variable $X_i$ i.i.d.\ from an (unknown) $p \in \cP$, and  generates $Y_i = \bT_i(X_i) \in \cY$.
The central server $U_0$ observes $(Y_1,\dots,Y_n)$ and constructs an estimate $\widehat{p} = \phi(Y_1,\dots,Y_n)$. We refer to this problem as simple hypothesis testing under communication constraints of $\cT$ and denote it by $\cB(\cP,\cT)$. When $\cY = [0:D-1]$ and $\cT= \cT_D$ for $D \geq 2$, we call $\cB(\cP,\cT_D)$ the simple hypothesis testing problem under communication constraints of $D$-messages.
When $\cP = \{p,q\}$, we also use the notation $\cB(\{p,q\},\cT_D)$.
\end{definition}

\begin{remark}
In the definition above, one could allow the rule to be stochastic based on either private randomness or public randomness, and this choice may affect the resulting sample complexity. We refer the reader to Acharya, Canonne, Liu, Sun, and Tyagi~\cite{AchCLST22-interactive} for differences between various protocols.
In the (standard version of the) public-coin protocol, a public random variable is observed by all users and the central server, and may be used to coordinate the choice of channels.
When public randomness is available to the users but \emph{not} the central server (this protocol subsumes the  private-coin protocol), it follows from the work of Tsitsiklis~\cite[Proposition 2.1]{Tsitsiklis93} that the sample complexity  is not lower than that of the deterministic protocol considered in \Cref{def:simpleHypothesisTestComm}. 
 When the public randomness is observed by the central server as well, the sample complexity of simple \emph{binary} hypothesis testing is the same (up to constants) to that of private-coin protocols (and thus, to deterministic protocols). For more complicated problems such as $M$-ary hypothesis testing or composite hypothesis testing, public-coin protocols may lead to smaller sample complexities. The effect of public randomness is discussed in detail in \Cref{app: public}.
\end{remark}

\begin{definition}[Sample complexity of $\cB(\cP,\cT_D)$]
\label{def:sampleComplexityComm}
For a given test-rule pair $(\phi,\cR)$ with $\phi:\cup_{j=1}^ \infty \cY^j \to \cP$, we say that $(\phi,\cR)$ solves $\cB(\cP,\cT_D)$ with sample complexity $n$ 
 if 
\begin{align}
\label{eq:DefErrorProbCommCons}
  \sum_{p \in \cP}\P_{(x_1,\dots,x_n) \sim p^{\otimes n}}( \phi(y_1,\dots,y_n) \ne p) \leq 0.1.
\end{align}
We use $\nstar(\cP,\cT_D)$ to denote the sample complexity of this task, i.e., the smallest $n$ so that there exists a $(\phi,\cR)$-pair that solves $\cB(\cP,\cT_D)$.
We use $\nide(\cP,\cT_D)$ to denote the setting where each channel is identical, i.e., $\cR : [n] \to \cup_{\bT \in \cT_D} \{\bT\}^n$.
In order to emphasize the setting where the channels need not be identical, we sometimes use $\nnid(\cP,\cT_D)$ to denote $\nstar(\cP,\cT_D)$.
When $\cP = \{p,q\}$, we will use the notation $\nstar(\{p,q\},\cT_D)$, $\nide(\{p,q\},\cT_D) $, and $\nnid(\{p,q\},\cT_D)$.

\end{definition}
We shall discuss the setting of adaptive channels in \Cref{sec:m-ary}.

\paragraph{Special case: Binary hypothesis testing.} %
\label{par:special_case_binary_hypothesis_testing}
In the rest of this section, we will focus on the special case when $\cP = \{p,q\}$.
For a fixed rule $\cR$, an optimal $\phi$ corresponds to the likelihood ratio test.
 Thus, our focus will be on designing the rule $\cR$, while choosing the test $\phi$ implicitly\footnote{We will mention the test explicitly wherever required, e.g., in robust hypothesis testing.}, such that the test-rule pair $(\phi,\cR)$ has minimal sample complexity.

A subset of channels called \emph{threshold channels} plays a key role in our theory: Consider a set $\Gamma =\left\{ \gamma_1 , \dots, \gamma_{D-1} \right\}$ such that $0 < \gamma_1 \leq \dots \leq \gamma_{D-1} < \infty$.
Let $\gamma_0 := 0$ and $\gamma_{D} := \infty$.
Define the function $w_\Gamma:[k] \to [0:D-1]$ as follows\footnote{When $q(x) = 0$ for some $x$ and $p(x) \neq 0$, we take $p(x)/q(x) = \infty$. Without loss of generality, we can assume that for each $x \in [k]$, at least one of $p(x)$ or $q(x)$ is non-zero.}: if $q(x) = 0$, then $w_\Gamma(x) = D-1$; otherwise,
  \begin{equation}
\label{eq:ThreshFunction}
 w_\Gamma(x) = j \text{ if and only if }  p(x) / q(x) \in [\gamma_{j}, \gamma_{j+1}).
  \end{equation}
 We are now ready to define a threshold test:
 \begin{definition}[Threshold test]
\label{def:thresh_test}
We say that a channel $\bT \in \cT_D$ corresponds to a threshold test for two distributions $p$ and $q$ over $[k]$  if there exists $\Gamma =\left\{ \gamma_1 , \dots, \gamma_{D-1} \right\}$ such that $0 < \gamma_1 \leq \dots \leq \gamma_{D-1} < \infty$, and $w_\Gamma(X) \sim \bT \tilde{p}$ whenever $X \sim \tilde{p}$ for any $\tilde{p}$ (cf.\ equation~\eqref{eq:ThreshFunction}). Any such $\Gamma$ is called the set of thresholds of the test $\bT$.
We use $\cTT_D$ to denote the set of all channels $\bT \in \cT_D$ that correspond to threshold tests. 
\end{definition}

Note that a priori, searching for an optimal channel over $\cTT_D$ seems to require $k^{\Omega(D)}$ time, as it requires searching over all possible values of $\Gamma$. By restricting our attention to a special class of thresholds parametrized by a single quantity, we will obtain a $\poly(k,D)$-time algorithm. 
In particular, we will focus on channels with thresholds in the following set:
\begin{align}
\cC := \left\{ \Gamma =  (\gamma_1,\ldots, \gamma_{D-1}): \forall j \in [D-2], \, \gamma_{j+1}/\gamma_{j}= 2 \right\}.
\label{eq:ThreshTestGeom}
\end{align}

A classical result states that threshold tests (cf.\ \cref{def:thresh_test}) are optimal tests under communication constraints:
\begin{theorem}[{\cite[Proposition 2.4]{Tsitsiklis93}}]
\label{thm:Tsitlikis}
$\nnid(\{p,q\},\cTT_D) = \nnid(\{p,q\},\cT_D).$ 
\end{theorem}
Our lower bounds on the sample complexity of hypothesis testing under communication constraints crucially rely on the optimality of threshold tests.

\section{Reverse data processing inequality for quantized channels} %
\label{sec:reverse_data_processing_inequality}
\label{sec:RevDataProc}

In this section, we state and prove a reverse data processing inequality for a class of $f$-divergences for communication-constrained channels.
We begin by defining a suitable family of $f$-divergences:
\begin{definition}[Well-behaved $f$-divergences]
\label{def:well-behaved}
We say $I_f(\cdot,\cdot)$ is a well-behaved $f$-divergence if it satisfies the following:
\begin{enumerate}[label=I.\arabic*]
  \item \label{item:pos} $f$ is a convex nonnegative function with $f(1)=0$.
  \item \label{item:symmetry} $xf(y/x) = y f(x/y)$.\footnote{This implies $I_f(p,q) = I_f(q,p)$.}
  \item \label{item:quad} There exist $ \alpha > 0, \kappa > 0, C_1 > 0$, and $C_2 > 0 $ such that for all $x \in [0, \kappa]$, we have
  
  \begin{equation*}
  C_1 x^\alpha \leq f(1 + x) \leq C_2x^\alpha.
  \end{equation*}
\end{enumerate}
\end{definition}

Some examples of well-behaved $f$-divergences include the total variation distance, squared Hellinger distance, symmetrized $\chi^2$-divergence,  symmetrized KL-divergence, and triangular discrimination (see \Cref{claim:wellBehExamples} for more details).
If $f$ is differentiable at $1$, $f'(1) = 0$, and the corresponding $f$-divergence is symmetric, then $f$ satisfies \ref{item:symmetry}~\cite{Gilardoni06,Sason15}.
Given an $f$-divergence that does not satisfy \ref{item:symmetry}, we can construct a new $f$-divergence with $\tilde{f}(x):= f(x) + x f(1/x)$, which is also a convex  function---this can be checked by noting that $\tilde{f}''(x) = f''(x) + \frac{1}{x^3} f''(x)$, which is non-negative, as $f$ is convex---satisfying $\tilde{f}(1) = 0$ and \ref{item:symmetry}. Furthermore, the convexity and non-negativity of $f$ means $\alpha$ must be at least $1$.

\subsection{Main result}

The main result of this section is as follows:

\begin{restatable}[Reverse data processing inequality]{theorem}{fDivRevDataProc}
\label{thm:quant-scheme}
Let $I_f$ be a well-behaved $f$-divergence with $(\alpha, \kappa, C_1, C_2)$ as defined in Definition~\ref{def:well-behaved}. Let $p$ and $q$ be two fixed distributions over $[k]$ such that for all $i \in [k]$, we have $q_i \geq \nu p_i$ and $p_i \geq \nu q_i$,  for some $\nu \in [0,1]$.
Then for any $D \geq 2$, there exists a channel $\bT^* \in \cTT_D$ (and thus in $\cT_D$) such that 
\begin{align}
\label{eq:GenFDivergence}
1  \leq \frac{I_f(p,q)}{I_f(\bT^* p,\bT^* q)}  \leq  4 \frac{f(\nu)}{f(1/(1+\kappa))} + \frac{52C_2}{C_1} \max\Big\{1, \frac{R}{D} \Big\},
\end{align}
where $R = \min \{k,k'\}$ and $k' = 1 + \log\left( \frac{4C_2 \kappa^\alpha}{I_f(p,q)}\right)$.
Furthermore, given $f$, $p$, and $q$, there is a $\poly(k,D)$-time algorithm that finds a $\bT^*$ achieving the rate in inequality~\eqref{eq:GenFDivergence}.
\end{restatable}
\begin{remark}
In the usual data processing inequality, the $f$-divergence $I_f(\bT p, \bT q)$ is upper-bounded by $I_f(p,q)$. Since the direction of the inequality is reversed in the second inequality in \Cref{thm:quant-scheme}, we interpret it as a \emph{reverse} data processing inequality. Another natural way to interpret this result is from the lens of quantization: \Cref{thm:quant-scheme} asserts that for any $p$, $q$, and any well-behaved $f$-divergence, there exist good quantization schemes to preserve the $f$-divergence.
\end{remark}
We provide a brief proof sketch for the special case of the Hellinger distance and $D = 2$ in \Cref{sec:proof_sketch_of_quant_scheme}, and defer the full proof to \Cref{app:RevDatProcProof}.
As our main focus will be on the Hellinger distance, we state the following corollary, which will be used later:
\begin{corollary}[Preservation of Hellinger distance]
\label{cor:hellQuant}
For any $p, q \in \Delta_k$ and $D \geq 2$, 
there exists a $\bT^* \in \cTT_D$ such that the following holds:
\begin{align}
\label{eq:HellGuarantee}
1  \leq \frac{\hel^2(p,q)}{\hel^2(\bT^* p,\bT^*q)}  \leq 1800\max\left\{1,   \frac{ \min\{k, k'\}}{D} \right\},
\end{align}
where $k'= \log(4/\hel^2(p,q))$.
Given $p$ and $q$, there is a $\poly(k,D)$-time algorithm that finds $\bT^*$ achieving the rate in inequality~\eqref{eq:HellGuarantee}.
\end{corollary}
\begin{proof}
 The desired bound follows by noting that $f(x) = (\sqrt{x} -1 )^2 $ for the Hellinger distance and taking $\nu= 0$. As shown in \Cref{app:technical_details} (\Cref{claim:wellBehExamples}), we can take $\kappa = 1$, $C_1 = 2^{-3.5}$, $C_2 = 1$, and $\alpha= 2$. 
Note that $f(0) = 1$ and $f(1/(1+\kappa)) = (\sqrt{2} - 1)^2/2 \geq 0.04$. 
This suffices to give a guarantee of $\frac{\hel^2(p,q)}{\hel^2(\bT^* p,\bT^*q)}  \leq 100 +  \frac{900}{D} \left( \min\{k , k'\}\right)$ from \Cref{thm:quant-scheme}.
\end{proof}

We note that \Cref{cor:hellQuant} also follows from the work of Bhatt, Nazer, Ordentlich, and Polyanskiy~\cite[Theorem 1]{BhaNOP21} by noting that when $X$ is a binary random variable uniform over $\{0,1\}$ and $Y$ is distributed as $p$ conditioned on $X=0$ and $Y$ is distributed as $q$ conditioned on $X=1$, then $\mathrm{I}(X;Y)$, the mutual information between $X$ and $Y$, satisfies  $\mathrm{I}(X;Y) \asymp \hel^2(p,q)$.

\begin{remark}[Dimensionality reduction using channels]\label{rem:HellEffSupportSize}
\cref{cor:hellQuant} can be interpreted as saying that the effective support size of $p$ and $q$ for the Hellinger distance is at most $k':=\log(4/\hel^2(p,q))$, because the distributions could be mapped to a $k'$-sized alphabet using a channel in a manner that preserves the pairwise Hellinger distance up to constant terms.
We also remark that our notion of dimensionality reduction requires the transformation to be performed using a channel, which is fundamentally different from the setting in Abdullah, Kumar, McGregor, Vassilvitskii, and Venkatasubramanian~\cite{AbdKMVV16}.
\end{remark}

The following result states that the bound in \Cref{cor:hellQuant} is tight:
\begin{restatable}[Reverse data processing is tight]{lemma}{RevMarkIneqHell}
\label{lem:HellTight} 
 There exist positive constants $c_1,c_2,c_3,c_4,c_5$, and $c_6$ such that for every $\rho \in (0, c_1)$ and $D \geq 2$, there exist $k \in [c_2 \log(1/\rho), c_3 \log(1/ \rho)]$ and two distributions $p$ and $q$ on $[k]$ such that $\hel^2(p,q) \in [c_4 \rho, c_5 \rho]$ and
\begin{align}
\label{eq:HellTight}
\inf_{\bT \in \cTT_D }\frac{\hel^2(p,q)}{\hel^2(\bT p,\bT q)}  \geq    c_6 \cdot \frac{R'}{D},
\end{align}
where $R' = \max\{k, k'\} $ and $k' = \log\left( 1/ \rho \right)$. Thus, $R' = \Theta(k) = \Theta(\log(1/\rho))$.
\end{restatable}
The proof of \Cref{lem:HellTight} is given in \Cref{app:tightRevDatProc}.

\subsection{Proof sketch of \Cref{thm:quant-scheme} and \Cref{cor:hellQuant}}
\label{sec:proof_sketch_of_quant_scheme}

We will focus on the case of the Hellinger distance and $D=2$.
The first step is to establish the following result:

\begin{lemma}[Reverse Markov inequality]
\label{lem:revMarkovBinary}
Let $X$ be a random variable over $[0,1)$, supported on at most $k$ points, with $\E[X] > 0$. Let $k' = 1 + \log(1/\E [X])$. Then
\begin{align}
\label{eq:revMarkovBinart}
\sup_{\delta \in [0,1)} \delta \P\left( X \geq \delta \right) \geq   \frac{\E[X]}{13R},
\end{align}
where $R = \min\{k,k'\}$.
\end{lemma}

The canonical version of Markov's inequality states that $\delta \P(X \geq \delta) \leq \E[X]$ for any non-negative random variable $X$ and any $\delta$.
Since the direction of the inequality is reversed in \Cref{lem:revMarkovBinary} (up to a shrinkage factor of roughly $\log(1/\E[X])$), we call it a \emph{reverse} Markov inequality.
A generalized version of \cref{lem:revMarkovBinary}
  for the case $D> 2$, along with its proof, is given in \cref{lem:revMarkovD}.
\begin{remark}
\label{rem:RevMarkovTight} 
Note that \cref{lem:revMarkovBinary} is tight, as shown in \cref{claim:ReverseMarkovTight}, which is crucially used in the proof of \cref{lem:HellTight}. 
\end{remark}
\begin{remark}
\label{rem:revMarkovVsExisting}
It is instructive to compare the guarantee of \cref{lem:revMarkovBinary} with existing results in the literature:
\begin{enumerate}
  \item The Paley-Zygmund inequality~\cite[Corollary 3.3.2]{delaGine99} states that for any $ \delta \in (0, \E[X])$, we have 
$$\P(X \geq \delta) \geq \left( 1 - \frac{\delta}{\E[X]} \right)^2 \frac{(\E[X])^2}{\E[X^2]}.$$
Multiplying both sides by $\delta$ and optimizing the lower bound over $\delta$ (achieved at $\delta =  \E[X]/3$) yields
\begin{align*}
\sup_{\delta \geq 0} \delta \P(X \geq \delta)  \gtrsim \E[X] \cdot \frac{1}{\E[X^2]/(\E[X])^2}  .
\end{align*}
Note that the shrinkage factor is $\E[X^2]/(\E[X])^2$, which is at most $1 / \E[X]$, but could be exponentially larger than the factor $\log(1/\E[X])$ provided in \cref{lem:revMarkovBinary}. (For example, consider a random variable with $\P(X = 0) = 1- p$ and $\P(X = 1/2) = p$: We have $\E[X^2]/(\E[X])^2 = 1 / p$, whereas $\log(1/\E[X]) = \log(2/p)$.)

\item A standard version of the reverse Markov inequality~\cite[Lemma B.1]{ShaBen14} for a random variable bounded in $[0,1]$ states that $\P(X \geq \delta) \geq \frac{\E[X] - \delta}{1 - \delta}$, for $\delta \in (0,\E[X])$.
Multiplying both sides by $\delta$ and optimizing the bound over $\delta \in (0, \E[X])$, under the condition that $\E[X] \leq 0.1$, gives us the following: 
\begin{align*}
\sup_{\delta \geq 0} \delta \P(X \geq \delta) \gtrsim (\E[X])^2 = \E[X] \cdot \frac{1}{1/\E[X]} ,
\end{align*}
i.e., the shrinkage factor is  $1/\E[X]$, which is again exponentially larger than $\log(1/\E[X])$. 
\end{enumerate}
\end{remark}

Using \cref{lem:revMarkovBinary}, we now sketch the proof that there exists a channel $\bT \in \cTT_2$ achieving $\hel^2(\bT p, \bT q) \gtrsim \hel^2(p,q)/R$. For simplicity of notation, we assume that for all $i \in [k]$, we have $p_i >0$ and $q_i > 0$. We first define the sets
\begin{align}
\label{eq:AdefinitionMainBody}
A_{l,u} & = \Big\{ i \in [k]: \frac{p_i}{q_i} \in [l,u) \Big\}, \notag \\
A_{l,\infty} & = \Big\{ i \in [k]: \frac{p_i}{q_i} \in [l,\infty] \Big\}.
\end{align}

Then $\hel^2(p,q)$ can be decomposed as follows:
\begin{align*}
\hel^2(p,q) &= \sum_{i \in A_{0,1/2}} \left( \sqrt{p_i} - \sqrt{q_i} \right)^2 + \sum_{i \in A_{1/2,1}} \left( \sqrt{p_i} - \sqrt{q_i} \right)^2 \\
&\quad\quad+ \sum_{i \in A_{1,2}} \left( \sqrt{p_i} - \sqrt{q_i} \right)^2 + \sum_{i \in A_{2,\infty}} \left( \sqrt{p_i} - \sqrt{q_i} \right)^2.
\end{align*}
We note that at least one of these terms must be at least $\hel^2(p,q)/4$.
By symmetry, it suffices to consider the cases where the sum over $ A_{2,\infty}$ is at least $\hel^2(p,q)/4$, or the sum over $ A_{1,2}$ is at least $\hel^2(p,q)/4$.

\paragraph{Case $1$: $\sum_{i \in A_{2,\infty}} \left( \sqrt{p_i} - \sqrt{q_i} \right)^2 \geq \hel^2(p,q)/4$} %
\label{par:case_1}

Let $\bT \in \cTT_2$ be a threshold test with threshold $\Gamma = \{2\}$, i.e., $\bT$ is a deterministic channel that corresponds to the function $i \mapsto \I_{p_i/q_i \geq 2}$. 
We note that $\bT p$ and $\bT q$ are binary distributions, characterized by $p'= \sum_{i\in A_{2,\infty}} p_i$ and $q' =  \sum_{i\in A_{2,\infty}} q_i$, respectively.
Then
\begin{align*}
 \hel^2(p,q) \leq 4 \sum_{i \in A_{2, \infty}} \left(\sqrt{p_i}- \sqrt{q_i}\right)^2 \leq 4 \sum_{i \in A_{2, \infty}} p_i = 4 p',
 \end{align*}
where the first inequality uses the assumption.
Using the fact that $p'\geq 2q'$, we also have
\begin{align*}
\hel^2(\bT p,\bT q) 
\geq \left(\sqrt{p'}- \sqrt{q'}\right)^2
\geq \left(\sqrt{p'}- \sqrt{p'/2}\right)^2 =(1- 1/\sqrt{2})^2 p' \geq 0.01p'.
\end{align*}
Combining the two displayed equations, we obtain $\hel^2(\bT p,\bT q) \geq \hel^2(p,q)/400$. This completes the proof.

\paragraph{Case $2$: $\sum_{i \in A_{1,2}} \left( \sqrt{p_i} - \sqrt{q_i} \right)^2 \geq \hel^2(p,q)/4$} %
\label{par:case_2}

For $i \in A_{1,2}$, let $\delta_i:= (p_i - q_i)/q_i$, which lies in $[0,1)$.
Consider the random variable $X$ over $[0,1)$ such that for $i \in A_{1,2}$, we define $\P(X= \delta_i)= q_i$ and $\P(X = 0) = 1 - \sum_{i \in A_{1,2}} q_i$.
Let $\delta \in [0,1)$ be arbitrary (to be decided later). Consider the channel $\bT$ corresponding to the threshold $1+\delta$.
Suppose for now that the following inequalities hold:
\begin{align}
\label{eq:TaylorMainBody}
\hel^2(p,q) \lesssim \E X^2 \quad \text{and} \quad  \hel^2(\bT p, \bT q) \gtrsim \delta^2 \P(X \geq \delta), 
\end{align}
which we will establish shortly using a Taylor approximation.
Letting $Y = X^2$ and $\delta' = \delta^2$, we obtain the following inequality using the bounds~\eqref{eq:TaylorMainBody}:
\begin{align}
\label{eq:TaylorMainBody2}
\frac{\hel^2(\bT p, \bT q)}{\hel^2(p,q)} \gtrsim \frac{\delta^2 \P(X \geq \delta)}{\E[X^2]} = \frac{\delta' \P(Y \geq \delta')}{\E[Y]}.
\end{align}
Fix
\begin{equation*}
R = \log(1/\E[Y]) = \log(1/\E[X^2]) = \log( O(1/ \hel^2(p,q)).
\end{equation*}
By \cref{lem:revMarkovBinary}, note that there exists $\delta'$ (and therefore also $\delta$) such that $\delta' \P(Y \geq \delta') \gtrsim \E[Y]/R$, which yields the desired lower bound $\frac{\hel^2(\bT p, \bT q)}{\hel^2(p,q)} \gtrsim \frac{1}{R} $, using inequality~\eqref{eq:TaylorMainBody2}.

We now provide a brief proof sketch of the bounds~\eqref{eq:TaylorMainBody}.  We derive the first bound using the following arguments:
 \begin{align*}
 \hel^2(p,q) &\leq 4 \sum_{i \in A_{1,2}} \left( \sqrt{p_i} - \sqrt{q_i} \right)^2 
      = 4 \sum_{i \in A_{1,2}} q_i( \sqrt{1 + \delta_i} - 1 )^2 
      \leq 4 \sum_{i \in A_{1,2}} q_i \delta_i^2
      = 4 \E[X^2],
 \end{align*}
 where the first inequality uses the assumption and the second inequality uses the fact that $\sqrt{1 + x} \leq 1 +x$ for $x \geq 0$.

We now turn our attention to the second bound~\eqref{eq:TaylorMainBody}.
Recall that $\bT$ is a channel corresponding to the threshold $1 + \delta$. 
Let $p' = \sum_{i: \delta_i \in  [\delta,1)} p_i$ and $q' = \sum_{i: \delta_i \in [\delta,1) }q_i$.
Note that $q' = \P(X \geq \delta)$ and $p'-q' = \sum_{i: \delta_i \in [\delta,1)} \delta_i q_i = \E[X \I_{X \geq \delta}]$. Thus, we have $(p'-q')/q' = \E[X | X \geq \delta] \geq \delta$.

It can be shown that $\hel^2(\bT p, \bT q) \geq (\sqrt{p'} - \sqrt{q'})^2$ (cf.\ inequality \eqref{eq:f-DivRestricted} in \Cref{app:reverse_data_processing}), which leads to the following inequalities:
\begin{align*}
\hel^2(\bT p, \bT q) &\geq  (\sqrt{p'} - \sqrt{q'})^2 
          = q' \left( \sqrt{1 + \frac{p'-q'}{q'}} - 1 \right)^2 
\gtrsim  q' \left(  \frac{p'-q'}{q'}\right)^2  
\geq \delta^2 \P(X \geq \delta),
\end{align*}
based on properties of the function $x \mapsto (\sqrt{1+x} - 1)/x$ on $(0,1]$, extended by continuity to $[0,1]$.

\subsection{Reverse Markov inequality: Proof and tightness}

In this section, we state the generalized version of the reverse Markov inequality (\Cref{lem:revMarkovBinary}) as \Cref{lem:revMarkovD}, and show that it is tight in general (\Cref{claim:ReverseMarkovTight}). 

\begin{restatable}[Generalized reverse Markov inequality]{lemma}{RevMarkovD}
\label{lem:revMarkovD}
Let $Y$ be a random variable over $[0, \beta)$ with expectation $\E[Y] > 0$. Let $k' = 1 +  \log(\beta/\E [Y])$. Then
\begin{align}
\label{eq:revMarkovD}
\sup_{0 \le \nu_1 \le  \cdots \le \nu_D = \beta} \sum_{j=1}^{D-1} \nu_{j} \P\left( Y \in [\nu_{j}, \nu_{j+1}) \right) \geq \frac{1}{13} \E[Y] \min \left\{1, \frac{D}{R}  \right\},
\end{align}
where $R = k' := 1 +  \log(\beta/ \E[Y])$. 
Furthermore, the bound~\eqref{eq:revMarkovD} can be achieved by $\nu_j$'s such that $\nu_j = \min\{\beta,x2^{j}\}$ for some $x\in[0,\beta]$.

For the special case where $Y$ is supported on $k$ points, we may set $R = \min\{k,k'\}$, and there is a $\poly(k,D)$ algorithm to find $\nu_j$'s that achieve the bound~\eqref{eq:revMarkovD}.
\end{restatable}
\begin{proof}
 We can safely assume that $D \leq R$. Under this assumption on $D$, we will show that the desired expression is lower-bounded by both of the following quantities (up to constants):
$\frac{\E[Y] D}{k}$ and $\frac{\E[Y] D}{k'}$.
We will also assume that $\beta=1$; otherwise, it suffices to apply the following argument to $\frac{Y}{\beta}$.

\paragraph{Dependence on $k$}
Suppose $Y$ has support size $k$.\footnote{It is easy to see that if the support size is strictly smaller than $k$, we have a tighter bound.} Let the support elements be $\{\delta_i'\}_{i=1}^{k}$, such that $\delta'_1 < \delta'_2 < \dots < \delta'_k < 1$.  Let $\{p_i\}_{i=1}^{k}$ be such that $\P[Y= \delta_i'] = p_i$ and $\sum_{i=1}^{k} p_i = 1$.

It suffices to prove that there exists a labeling  $\pi: [D-1] \to [k]$ such that $\pi(1) < \pi(2) < \dots < \pi(D-1)$ and the following bound holds:
\begin{align}
\label{Eq:RevWeakerClaim}
\sum_{j=1}^{D-1} \delta'_{\pi(j)} p_{\pi(j)} \geq \E[Y] \left(\frac{D-1}{k}  \right).
\end{align}
This is true because for $j \in [D-1]$, we have  $p_{\pi(j)} = \P\left\{ Y \in [\delta_{\pi(j)}, \delta_{\pi(j) + 1})  \right\}   \leq \P\left\{ Y \in [\delta_{\pi(j)}, \delta_{\pi(j+1)}) \right\}$, where we define $\delta_{k+1} := 1$ and $\pi(D):= 1$, and the desired conclusion follows by setting $\nu_j = \delta_{\pi(j)}$.
In the rest of the proof, we will show that such a $\pi$ exists.

Let $\sigma: [k] \to [k]$ be a permutation such that $p_{\sigma(i)} \delta_{\sigma(i)} \geq  p_{\sigma(i+1)} \delta_{\sigma(i+1)}$. 
Then we have
\begin{align*}
\frac{\E [Y]}{k} &= \frac{\sum_{i=1}^k p_i \delta_i}{k} =  \frac{\sum_{i=1}^k p_{\sigma(i)} \delta_{\sigma(i)}}{k} \leq \frac{\sum_{i=1}^{D-1} p_{\sigma(i)} \delta_{\sigma(i)}}{D-1} = \frac{\sum_{i=1}^{D-1} p_{\pi'(i)} \delta_{\pi'(i)}}{D-1},
\end{align*}
for some $\pi': [D-1] \to [k]$ such that $ \pi'(1) < \pi'(2) < \dots < \pi'(D-1)$. Thus, we have established inequality \eqref{Eq:RevWeakerClaim}.
Note that the desired bound is achieved by choosing the $\nu_j$'s as follows: Let $S$ be the set of top $D-1$ elements among the support of $Y$ that maximize $y\P(Y=y)$, and let the $\nu_j$'s have values in $S$ such that they are increasing and distinct. It is clear that this assignment can be implemented in $\poly(k,D)$ time.

\paragraph{Dependence on $k'$}

We begin by noting that the desired expression can also be written as
\begin{align*}
\sum_{j=1}^{D-1}\left( \nu_{j} - \nu_{j-1} \right) \P\left\{ Y \geq \nu_j\right\},
\end{align*}
where $\nu_0 := 0$.
We need to obtain a lower bound on the supremum of this expression over the $\nu_j$'s.
In fact, we will show a stronger claim, where we fix the $\nu_j$'s in a particular way: We will take $\nu_j$ to be of the form $x 2^{j-1}$, for $j \in [D-1]$, and optimize over $x \in (0,1)$. Note that we can allow $\nu_j \geq 1$ without loss of generality, because their contribution to the desired expression would then be $0$.  
In the rest of the proof, we will show the following claim for $c_1 = \frac{1}{13}$:
\begin{align}
\label{eq:RevCauchyExpSequence}
\sup_{x \in (0,1)} x \P\left\{ Y \geq x\right\} + \sum_{j=2}^{D-1} \left( x 2^{j-1} - x 2^{j-2} \right)\P\left\{ Y \geq x2^{j-1}\right\} \geq c_1 \E[Y]\frac{D}{k'}.
\end{align}

Suppose that the desired conclusion does not hold. We will now derive a contradiction.
Under the assumption that inequality \eqref{eq:RevCauchyExpSequence} is false, we  have the following, for each $x\in (0,1)$:
\begin{align*}
c_1 \E[Y]\frac{D}{k'}  &>
 x \P\left\{ Y \geq x \right\}  +\sum_{j=2}^{D-1} x \left( 2^{j-1} - 2^{j-2} \right) \P\left\{ Y \geq x 2^{j-1}\right\}  \\
&= x \left( \P\left\{ Y \geq x \right\} + \sum_{j=1}^{D-2} 2^{j-1} \P\left\{ 2^{-j} Y \geq x\right\} \right).
\end{align*}
We thus obtain the following, for all $x\in(0,1)$:
\begin{align}
 \label{eq:DaryMsgFirstUpp}
 \P\left\{ Y \geq x \right\} + \sum_{j=1}^{D-2} 2^{j-1} \P\left\{ 2^{-j} Y \geq x\right\}   < c_1 \E[Y] \cdot \frac{D}{k'} \cdot \frac{1}{x}.
\end{align}
Using the fact that the probabilities are bounded by $1$, we also have the following bound on the expression on the left side of inequality \eqref{eq:DaryMsgFirstUpp}:
\begin{align}
 \label{eq:DaryMsgSecUpp}
\P\left\{ Y \geq x \right\} + \sum_{j=1}^{D-2} 2^{j-1}\P\left\{ 2^{-j} Y \geq x\right\}  &\leq 1 + \sum_{j=1}^{D-2} 2^{j-1}  = 2^{D-2}.
\end{align}
Combining the two bounds in inequalities \eqref{eq:DaryMsgFirstUpp} and \eqref{eq:DaryMsgSecUpp}, the following holds for every $x \in (0,1)$:
\begin{align}
 \label{eq:DaryMsgJointUpp}
\P\left\{ Y \geq x \right\} + \sum_{j=1}^{D-2} 2^{j-1} \P\left\{ 2^{-j} Y \geq x\right\}  \leq \min\left\{2^{D-1}, \frac{c_1 D\E[Y]}{k' x} \right\}.
\end{align}
Using the fact that $Y \in [0,1]$, we also have the following for every $a > 1$: 
\begin{align*}
\E\left[\frac{Y}{a}\right] = \int_{0}^{1/a} \P\left(\frac{Y}{a} \geq t \right) dt  = \int_{0}^{1} \P\left(\frac{Y}{a} \geq t \right)dt,
\end{align*}
implying that
\begin{align}
\label{eq:DaryMsgExpEquality}
\int_{0}^{1} \left( \P\left\{ Y \geq t \right\} + \sum_{j=1}^{D-2}  2^{j-1}\P\left\{ Y2^{-j} \geq t \right\}  \right) dt 
 = \E[Y] + \sum_{j=1}^{D-2}2^{j-1}\E[2^{-j}Y] =  \frac{D\E[Y]}{2}.
\end{align}
Combining inequalities \eqref{eq:DaryMsgJointUpp} and \eqref{eq:DaryMsgExpEquality}, we obtain the following, for an arbitrary $x_* \in (0,1)$:
\begin{align}
\nonumber\E[Y] &= \frac{2}{D} \int_{0}^{1} \left( \P\left\{ Y \geq t \right\} + \sum_{i=1}^{D-2}  2^{i-1}\P\left\{ Y2^{-i} \geq t \right\}  \right) dt \\
\nonumber&\leq \frac{2}{D} \int_{0}^{1} \min\left\{2^{D-1}, \frac{c_1D\E[Y]}{k' x} \right\}   dt \\
\label{eq:DaryMsgContradFinal}&\leq \frac{2^{D}x_*}{D} + \frac{2c_1\E[Y]}{k'} \log\left( \frac{1}{x_*} \right).
 \end{align}
Let $x_* = 2^{-D} \left( \E[Y] / k' \right)$. Then
$$\log(1/x_*) = D + \log(1/\E[Y]) + \log(k') \leq 3k',$$
where we use the fact that $\max\{D,\log(1/\E[Y])\} \leq k'$.
Using $k' \geq 1$ and $D \geq 2$, the expression on the right-hand side of inequality~\eqref{eq:DaryMsgContradFinal} can be further upper-bounded, to obtain the following inequality:
\begin{align*}
\E[Y] \leq \left( \frac{\E[Y]}{k'D} \right) + \frac{2c_1\E[Y]}{k'} \left( 3k' \right) \leq \frac{\E[Y]}{2} + 6c_1 \E[Y],
\end{align*}
which is a contradiction, since $\E[Y] > 0$ and $c_1 < \frac{1}{12}$.
Thus, we conclude that inequality \eqref{eq:RevCauchyExpSequence} is true.
\end{proof}

\begin{restatable}[Tightness of reverse Markov inequality]{claim}{ClaimRevMarkovTIght}
\label{claim:ReverseMarkovTight}
There exist constants $c_1,c_2,c_3, c_4,c_5$, and $c_6$ such that for every $\rho \in (0,c_5)$,
there exists an integer $k \in [c_3 \log(1/\rho), c_4 \log(1/\rho)]$ and a probability distribution $p$,  supported over $k$ points in $(0,0.5]$, such that the following hold: 
\begin{enumerate}
  \item  $\E[X^2] \in [c_1\rho,c_2 \rho] $, and for every $D \leq 0.1k$,
\begin{align}
\label{eq:RevMarkovTight1}
\sup_{0 < \delta_1 < \cdots < \delta_D = 1} \sum_{j=1}^{D-1} \P\left\{ X \geq \delta_j \right\} \left( \E\left[ X | X \geq \delta_j \right] \right)^2 
 \leq c_6 \cdot\E [X^2]  \frac{ D }{R'},
\end{align}
where $R' = \max \{k, k'\} $ and $ k' = \log(3/\E[X^2])$.
\item $\E[Y] =  [c_1\rho,c_2 \rho]$, and 
 \begin{align}
 \label{eq:RevMarkovTight2}
\sup_{0 < \delta_1' < \cdots < \delta_{D}' = 1} \sum_{j=1}^{D-1} \delta_{j}' \P\left( Y \in [\delta_{j}', \delta_{j+1}') \right) 
\leq c_6 \cdot \E[Y] \frac{D}{R'},
\end{align}
where $R' = \max \{k, k'\}$ and $ k' = \log(3/\E[Y])$. Moreover, $R' = \Theta(\log(1/\rho))$.
 \end{enumerate}  
\end{restatable}

\begin{proof}
For now, let $k \in \N$ be arbitrary; we will choose $k$ so that $\E[X^2] \in [c_1\rho, c_2 \rho]$.
Consider the following discrete random variable $Y$ supported on $\{2^{-i}: i \in [k]\}$:
\begin{align*}
\P\left\{ Y = 2^{-i} \right\} = r 2^i,
\end{align*}
where $r$ is chosen so that it is a valid distribution, i.e., $r$ satisfies $1 = \sum_{i =1}^k r2^{i} = 2r (2^{k}-1)$.
Let $X = \sqrt{Y}$.
We then have
\begin{align}
\E [X^2] =\E [Y] = \sum_{i=1}^ k 2^{-i} \left( r 2^{i} \right) =rk. 
\label{eq:RevMarkovTight3}
\end{align}

Consider a $\delta_i' \in [2^{-j}, 2^{-(j-1)})$, for some $j \in [k]$, and let $\delta_i = \sqrt{\delta'_i}$. For any such choice, we obtain the following:
\begin{align}
\label{eq:revMarkovTight4}
\P(X \geq \delta_i) &= \P(Y \geq \delta_i') = \P(Y \geq 2^{-j}) = \sum_{i \in [j]} \P\left\{Y = 2^{-i} \right\} = \sum_{i \in [j]} r 2^i = 2r\left( 2^{j} - 1 \right) \leq 2r2^j.
\end{align}
Thus, for any $\delta'$, we have $\delta'\P\left\{ Y \geq \delta' \right\} \leq 2 r$, showing that the expression in inequality \eqref{eq:RevMarkovTight2} is upper-bounded by $2(D-1)r$, which is equal to $\frac{2\E[Y] (D-1)}{k}$, by equation \eqref{eq:RevMarkovTight3}. It remains to show that $R \leq c_6k/2$.

We first calculate bounds on $k$ so that $\E[Y] = \Theta(\rho)$.
Note that by construction, we have $r = 1/\left( 2(2^k - 1) \right)$, implying that $r \in [2^{-k+1}, 2^{-k-1}]$.
Since $\E[Y] = r k$, it suffices to choose $k$ such that $f(k) \in [2c_1 \rho, 0.5c_2 \rho]$, where $f(j):=2^{-j}j$.
As $\frac{f(j+1)}{f(j)} \in (1/2,1)$ for $j > 1$, we have $f(\lfloor \ln(1/\rho)\rfloor) \geq \rho \lceil\ln(1/\rho)\rceil$ and $f(\lceil 2\ln(1/\rho)\rceil) \leq \rho^2 \lceil\log(1/\rho)\rceil$, so we know that such a $k$ exists in $ [ \ln(1/\rho), 2 \ln(1/\rho)] $ when $c_1 = 0.5$, $c_2 = 10$, and $c_5 = 2^{-20}$.

We now calculate the quantity $R$. By the definition of $k'$,  we have 
\begin{align*}
 k' &= \log\left( \frac{3}{\E[X^2]} \right) = \log\left( \frac{3}{rk} \right) \geq \log\left( \frac{3 \cdot 2^{k-1}}{k} \right) = k \log 2  - \log k + \log(3/2). 
 \end{align*}
As $k$ is large enough, we have $k' \in [0.5k,2k]$.
Since $R = \max\{k, k'\}$, we have $R \in [0.5k, 2k]$.  This completes the proof of the claim in inequality \eqref{eq:RevMarkovTight2}, with $c_6=4$.

We now prove the claim in inequality \eqref{eq:RevMarkovTight1}.
We begin with the following:
\begin{align*}
\E[X \I_{X \geq 2^{-j/2}}] &= \sum_{i \in [j]} 2^{-0.5 i}\left( r 2^i \right) = \sum_{i \in [j]} r 2^{0.5 i} = (r \sqrt{2})\left(  \frac{2^{0.5j} -1 }{\sqrt{2}-1} \right) \leq 10 r 2^{0.5j}.
\end{align*}
For $\delta_i \in [2^{-j/2}, 2^{-(j-1)/2})$, for some $j$, we have the following:
\begin{align*}
\P\left\{ X \geq \delta_i \right\} \left( \E\left[ X | X \geq \delta_i \right] \right)^2 &= \P(X \geq 2^{-j/2}) \left( \E[X | X \geq 2^{-j/2}] \right)^2\\
 &= \frac{\left( \E[X \I_{X \geq 2^{-j/2}}] \right)^2}{\P(X \geq 2^{-j/2})}\\
  &\leq \frac{100 r^22^j }{2r(2^j - 1)} \leq 100r.
\end{align*}
Thus, the supremum over any arbitrary $\delta_i$ is also upper-bounded by $100r$.
Hence, we can upper-bound the expression on the left-hand side of inequality \eqref{eq:RevMarkovTight1} by $100(D-1)r$, which is equal to $100 \cdot \frac{\E[X^2] (D-1)}{k}$.
Using the same calculations as in the first part of the proof, we prove inequality \eqref{eq:RevMarkovTight1} with $c_6 = 200$.
\end{proof}

\section{Simple binary hypothesis testing} %
\label{sec:BinHypTest}

We will now apply the results from the previous sections to simple binary hypothesis testing under communication constraints.
Let $\bT$ be a fixed channel, and suppose all users use the same channel $\bT$.
In this setting, \Cref{fact:testing} implies that the sample complexity is $\Theta(1/(\hel^2(\bT p, \bT q)))$. 
Without any communication constraints, the sample complexity of the best test is known to be $\Theta(1/(\hel^2(p,q)))$.
Thus, the additional (multiplicative) penalty of using the channel $\bT$ is $ \frac{\hel^2(p,q)}{\hel^2(\bT p, \bT q)}$, which is at least $1$, by the data processing inequality.
As we are allowed to choose any channel $\bT \in \cT_D$, we would like to choose the channel that minimizes this ratio, which was precisely studied in \cref{sec:RevDataProc}.

\subsection{Upper bound}
\label{sec:upper_bound_for_simple_binary_hypothesis_testing}

We begin with an upper bound, which follows directly from \cref{cor:hellQuant}.

\begin{theorem}
\label{thm:ub-Simple-D} 
There exists a positive constant $c$ satisfying the following: For any $k \in \N$, let $p$ and $q$ be two distributions on $\Delta_k$ and  define $\nstar := \nstar(\{p,q\})$. 
For any $D \geq 2$, the sample complexity of simple binary hypothesis testing with identical channels satisfies
\begin{align}
\label{eq:ub-Simple-D}
 \nide(\{p,q\}, \cT_D) \leq c \cdot \nstar \cdot \max \left\{ 1, \frac{\min\{k, \log \nstar\} }{D} \right\}.
 \end{align}
Furthermore, there is an algorithm which, given $p$, $q$, and $D$, finds a channel $\bT^* \in \cTT_{D}$ in $\poly(k,D)$ time that achieves the rate in inequality~\eqref{eq:ub-Simple-D}. 
\end{theorem}

\begin{proof}
As noted earlier, for a fixed $\bT$, the sample complexity is
\begin{align*}
\Theta\left( \frac{1}{\hel^2(\bT p, \bT q)} \right) = \Theta\left( \frac{1}{\hel^2(p,q)} \cdot \frac{\hel^2(p,q)}{\hel^2(\bT p, \bT q)} \right) = \Theta\left( \nstar \cdot g(\bT)  \right),
\end{align*}
 where $g(\bT) := \frac{\hel^2(p,q)}{\hel^2(\bT p, \bT q)}$.
 Our proof strategy will be to upper-bound the quantity $\inf_{\bT \in \cT_D} g(\bT)$.
By \cref{cor:hellQuant}, there exists a $\bT^*$ such that $g(\bT^*) \lesssim \max\{1, \min\{k,\log(\nstar)\}/D\}$, since $\nstar= \Theta(1/\hel^2(p,q))$ by \cref{fact:testing}.
Thus, the proof of \cref{thm:ub-Simple-D} follows from \cref{cor:hellQuant} by choosing the optimal $\bT^*$ achieving the bound in \cref{cor:hellQuant}.
As mentioned in \cref{cor:hellQuant}, the channel $\bT^*$ can be found efficiently. 
\end{proof}

\subsection{Lower bound}

We now prove a lower bound, showing that there exist distributions $p$ and $q$ such that the upper bound in \Cref{thm:ub-Simple-D} is tight.
As discussed in \cref{sec:prelim-commConstraint}, an optimal test that minimizes the probability of error under communication constraints is a threshold test based on $\frac{p(i)}{q(i)}$ \cite{Tsitsiklis93}.
However, this notion of optimality is conditioned on the fact that the channels are potentially non-identical; examples exist where such a condition is necessary even for $D=2, M=2$, and $n=2$ \cite{Tsitsiklis88}.

We will show that, up to constants in the sample complexity, it suffices to consider identical channels for simple hypothesis testing.
In fact, we prove a much more general result below that does not rely on restricting the function class to threshold tests.
\begin{restatable}[Equivalence between identical and non-identical channels for simple hypothesis testing]{lemma}{LemIdentVsNonIdent}
\label{lem:IdentVsNonIdent}
Let $\cT$ be a collection of channels from $\cX \to \cY$. Let $p$ and $q$ be two distributions on $\cX$. Then
\begin{align*}
\nnid\left(\{p,q\},\cT \right) = \Theta\left( \nide(\{p,q\}, \cT)\right).
\end{align*}
\end{restatable}
\begin{proof}
Recall that we use $\haff(p,q)$ to denote the Hellinger affinity between $p$ and $q$.
It suffices to consider the case where $\nide(\{p,q\}, \bT)$ is larger than a fixed constant.
Define the following: 
\begin{align*}
 h_* = \sup_{\bT \in \cT}\hel(\bT p,\bT q), \quad \text{and} \quad \beta_* := \inf_{\bT \in \cT} \haff(\bT p,\bT q),
 \end{align*}
and note that $ \beta_* = 1 - 0.5h_*^2$.
Let $n_* := \nide(\{p,q\}, \bT)$.
Let $\tstar$ be any channel such that $\haff (\tstar p,\tstar q) \leq \beta_* + \epsilon \beta_*$, for some $\epsilon > 0$ satisfying $(1 + \epsilon)^{n_*} \leq 2$. 
Let $p_* = \tstar p$ and $q_* = \tstar q$.

\paragraph{Identical channels: Optimal $\tstar$}
If each channel is identically $\tstar$, the joint distributions of $n_*$ samples will either be $p_*^{\otimes n_*}$ or $q_*^{\otimes n_*}$.

Let $f(n) = \dtv\left(p_*^{\otimes n}, q_*^{\otimes n}\right)$.
Note that the probability of error for $ p_*^{\otimes n}$ and  $q_*^{\otimes n}$  is equal to  $1 - f(n)$ (cf.\ \cref{fact:testing}).
Since the sample complexity of $\cB(p_*, q_*)$ is at least $n_*$, the probability of error with $n_*-1$ samples must be greater than $0.1$, i.e., $f(n_*-1) < 0.9$.
Using \cref{fact:div}, we have 
$\hel^2\left(p_*^{\otimes (n_* - 1)},q_*^{\otimes (n_* - 1)}\right) \leq 1.8$, and consequently, $ \haff\left(p_*^{\otimes (n_* - 1)},q_*^{\otimes (n_* - 1)}\right) \geq 0.1$. Using the tensorization of Hellinger affinity (cf.\ \cref{fact:div}) and the relation between $\beta_*$ and $\haff(p_*,q_*)$, we have
\begin{align}
\label{eq:helAffinity}
(\beta_*)^{n_* - 1 } \geq \left(\frac{\haff(p^*,q^*)}{1 + \epsilon}\right)^{n_* - 1 } \geq \frac{1}{2} \haff\left(p_*^{\otimes (n_* - 1)},q_*^{\otimes (n_* - 1)}\right) \geq 0.05.\end{align}

\paragraph{Non-identical channels}

We now show that even if $n$ non-identical channels are allowed but $n \leq 0.01 n_*$, the probability of error is at least $0.2$.
For a choice of $\bT_1,\dots,\bT_n$, let $P'_n := \prod_{i=1}^n \bT_i p$ and $Q'_n := \prod_{i=1}^n \bT_i q$ be the resulting joint probability distributions under $p$ and $q$, respectively.
As the probability of error of the best test is $1 - \dtv\left( P'_n,  Q'_n \right)$ (cf.\ \cref{fact:testing}), it suffices to show that if $n \leq 0.01 n_*$, then $\dtv(P'_n , Q'_n) \leq 0.8$.

Using \cref{fact:div}, it suffices to show that $\hel^2(P'_n,Q'_n) \leq 0.64$.
Equivalently, it suffices to show that $\haff(P'_n,Q'_n) \geq 0.68$.
Using the tensorization of Hellinger affinity and optimality of $\beta_*$, we have
\begin{align*}
\haff(P'_n,Q'_n) &= \prod_{i=1}^n \haff(\bT_ip,\bT_iq) \geq \beta_*^n = (\beta_*^{n_* - 1})^{\frac{n}{n_* - 1}} \geq \left( (0.05)^{\frac{1}{50}} \right)^{\frac{50n}{n_*-1}} \geq  0.9^{\frac{50n}{n_*-1}},
\end{align*}
where we use inequality~\eqref{eq:helAffinity}.
Thus, if $n \leq 0.01n_*$, the Hellinger affinity is larger than $0.68$, implying that the total variation is small and the probability of error is large.
\end{proof}

We now derive the following lower bound on $\nnid(\{p,q\},\cT)$:
\begin{theorem}
\label{thm:BinHypTestLowerBound}
  There exist positive constants $c_1$ and $c_2$ such that for every $n_0 \in \N$ and $D \geq 2$, there exist (i) $k = \Theta(\log n_0)$ and (ii) two distributions $p$ and $q$ on $[k]$, such that the following hold:
  \begin{enumerate}
    \item $c_1 n_0 \leq \nstar(\{p,q\}) \leq c_2n_0$, and
    \item $\nnid(\{p,q\},\cT_D) \geq \frac{n_0\log(n_0)}{D}$.
   \end{enumerate}
\end{theorem}

\begin{proof}
We first provide a proof sketch. Using \Cref{lem:IdentVsNonIdent,thm:Tsitlikis}, it suffices to consider the setting with identical threshold channels.
With identical channels, say $\bT$, the problem reduces to that of $\cB(\bT p, \bT q)$, and thus to bounding $\hel(\bT p, \bT q)$, using \Cref{fact:testing}.  
Tightness of \Cref{lem:HellTight} then gives the desired result.

Turning to the details, note that it suffices to consider $D \leq \log n_0$. Moreover, we can consider the setting where $n_0$ is sufficiently large; the result for general $n_0$ then follows by changing the constants $c_1$ and $c_2$ appropriately.

Now define $\rho = 1/ n_0$. Since $n_0$ is large enough, this $\rho$ satisfies the condition of \cref{lem:HellTight}.
Let $p$, $q$, and $k = \Theta(\log(1/\rho))$ be from \cref{lem:HellTight}, such that (i) $\hel^2(p,q)= \Theta(\rho)$ and (ii) inequality~\eqref{eq:HellTight} holds.
Using \Cref{fact:testing}, we have the following:
\begin{align}
\label{eq:testinglowerbd}
\nstar(\{p,q\}) = \Theta \left(  \frac{1}{\hel^2(p,q)} \right)  = \Theta \left(\frac{1}{\rho}  \right)  = \Theta\left( n_0  \right).
\end{align}
Thus, $p$ and $q$ satisfy the first condition of the theorem.
Furthermore, we have the following (where $c'$ represents a positive constant which may change from line to line):
\begin{align*}
\nnid(\{p,q\},\cT_D) &= \nnid(\{p,q\},\cTT_D) && \left(  \text{{using \cref{thm:Tsitlikis}}}\right) \\
&\geq c' \nide(\{p,q\},\cTT_D) && \left( \text{using \cref{lem:IdentVsNonIdent}} \right)  \\
&\geq c' \inf_{\bT \in \cTT_D} \frac{1}{\hel^2(\bT p,\bT q)}  && \left( \text{using \cref{fact:testing}} \right) \\
&= \frac{c'}{\hel^2(p,q)} \inf_{\bT \in \cTT_D} \frac{\hel^2(p,q)}{\hel^2(\bT p, \bT q)} \\
&\geq \frac{c'}{\hel^2(p,q)} \frac{\log (1/\hel^2(p,q))}{D} &&\left( \text{using \cref{lem:HellTight}} \right) \\
&\geq c' \nstar(\{p,q\}) \frac{\log (c'' \nstar(\{p,q\}))}{D} &&\left( \text{using equation \eqref{eq:testinglowerbd}} \right) \\
&\geq n_0 \frac{\log (n_0)}{D} &&\left( \text{bounding $\nstar(\{p,q\})$} \right).
\end{align*}
This completes the proof.
\end{proof}

\subsection{Robust tests}
\label{sec:robust_tests_binary}

In this section, we study the robust version of $\cB(\{p,q\},\cT_D)$. Here, the data-generating distribution may not belong to $\cP:=\{p,q\}$, but is only guaranteed to lie within a certain radius of an element of $\cP$. 
Our main result (\Cref{thm:ub-Simple-D-robust}) shows that communication constraints increase the sample complexity of the robust version of hypothesis testing by at most logarithmic factors.

We begin by formally defining the robust version of simple hypothesis testing under communication constraints:
\begin{definition} [Robust version of $\cB(\{p,q\},\cT_D)$]
\label{def:robustHyp}
Let $\cP_1$ and $\cP_2$ be defined as $\cP_1 := \{\tilde{p}: \dtv(p,\tilde{p}) \leq \epsilon\}$ and  $\cP_2 := \{\tilde{q}: \dtv(q,\tilde{q}) \leq \epsilon\}$.
The robust version of $\cB(\{p,q\},\cT)$, denoted by $\cB_\robust(\{p,q\}, \epsilon, \cT)$, is defined as in \cref{def:simpleHypothesisTestComm}, but with $\cP = \cP_1 \cup \cP_2$. 
For a given test-rule pair $(\phi,\cR)$ with $\phi:\cup_{j=1}^ \infty \cY^j \to \cP$, we say that $(\phi,\cR)$ solves $\cB_\robust(\{p,q\}, \epsilon, \cT)$ with sample complexity $n$ if
\begin{align}
\label{eq:robust}
 \sup_{\tilde{p} \in \cP_1} \P_{(x_1,\dots,x_n) \sim \tilde{p}^{\otimes n}}( \phi(y_1,\dots,y_n) \ne p) +  \sup_{\tilde{q} \in \cP_2} \P_{(x_1,\dots,x_n) \sim \tilde{q}^{\otimes n}}( \phi(y_1,\dots,y_n) \ne q) \leq 0.1.
\end{align}
We use $\nstar_\robust(\{p,q\},\epsilon,\cT)$ to denote the sample complexity of this task, i.e., the smallest $n$ so that there exists a $(\phi,\cR)$-pair that solves $\cB_\robust(\{p,q\},\epsilon,\cT)$ with sample complexity $n$.
We use $\cB_\robust(\{p,q\},\epsilon)$ and $\nstar_\robust(\{p,q\}, \epsilon)$ to denote, respectively, the robust hypothesis testing problem and the sample complexity of robust testing in the absence of any channel constraints.
\end{definition}
For any two distributions $p$ and $q$ with $\dtv(p,q) = 3 \epsilon$, it is possible to obtain an $\epsilon$-robust test with sample complexity $O(1/\epsilon^2)$ (e.g., using Scheffe's test).
However, as the following example shows, the optimal communication-efficient channel for $\cB(\{p,q\},\cT_D)$ may not be robust to $\epsilon^{1 + \alpha}$-contamination for any $\alpha \in [0,1)$; we refer the reader to \Cref{app:hypothesis_testing} for more details.

\begin{example}[Optimal channel may not be robust]
\label{exm:robust-1}
 Let  $\alpha \in (0,1)$  and let $\epsilon > 0$ be small enough. Let $p \in \Delta_3$ be the distribution $\left(0.5- 3\epsilon- \epsilon^{1 + \alpha}, 0.5 + 3\epsilon, \epsilon^{1 + \alpha}\right)$, and let $q \in \Delta_3$ be the distribution $(0.5,0.5,0)$. 
Then $\dtv(p,q) \geq 3 \epsilon$ and $\nstar_\robust(\{p,q\}, \epsilon) = \Theta(1/\epsilon^{2})$. 
However, the optimal\footnote{The channel corresponds to the function $\I_{\{3\}}(x)$, and it transforms $p$ and $q$ to the distributions $(1 - \epsilon^{1 + \alpha}, \epsilon^{1 + \alpha})$
  and $(1,0)$, respectively.} channel $\bT^*$ for $\cB(\{p,q\},\cT_2)$ is not robust to $\epsilon^{1+\alpha}$-corruption: there exists $\tilde{p}$, satisfying $\dtv(p,\tilde{p}) \leq \epsilon^{1 + \alpha}$, such that $\bT^* \tilde{p} = \bT^* q$. 
\end{example}

As our main result, we show that there is an (efficient) way to choose channels such that the sample complexity increases by at most a logarithmic factor:

\begin{restatable}[Sample complexity of $\cB_\robust(\{p,q\},\cT_D)$]{theorem}{ThmRobustBinaryD}
\label{thm:ub-Simple-D-robust}
There exists a constant $c >0$ such that
for any $p, q \in \Delta_k$ with $\epsilon < \frac{\dtv(p,q)}{2}$ and any $D \geq 2$, we have
\begin{align}
\label{eq:ub-Simple-D-rob}
 \nstar_\robust(\{p,q\} , \epsilon, \cT_D) \leq c \cdot \nstar \cdot \max \left\{ 1, \frac{\min\{k, \log \nstar\}}{D} \right\},
 \end{align}
where $\nstar := \nstar_\robust(\{p,q\}, \epsilon)$.
Furthermore, there is an algorithm which, given $p$, $q$, $\epsilon$, and $D$, finds a channel $\bT^* \in \cTT_{D}$ in $\poly(k,D)$ time that achieves the rate in inequality~\eqref{eq:ub-Simple-D-rob}. 
\end{restatable}
Note that the optimal channel in \Cref{thm:ub-Simple-D-robust} may depend on $\epsilon$.
Our proof critically uses the framework of least favorable distributions (LFDs) for binary hypothesis testing, pioneered by Huber~\cite{Huber65}.
LFDs are pairs of distributions $\tilde{p} \in \cP_1$ and $\tilde{q} \in \cP_2$ that maximize $\inf_{\phi} f (\phi, \tilde{p}, \tilde{q}, n )$, where $f (\phi, \tilde{p}, \tilde{q}, n )$ is the probability of error of a test $\phi$ which distinguishes $\tilde{p}$ and $\tilde{q}$ based on $n$ samples. Remarkably, LFDs do not depend on $n$ when $\cP_1$ and $\cP_2$ are $\epsilon$-balls around $p$ and $q$, respectively, in the total variation distance~\cite{Huber65,HubStr73}.
Moreover, these LFDs can be constructed algorithmically.
Particularly relevant for us is the result of
 Veeravalli et al.\ \cite{VeeBP94},
  who extended these results in the presence of communication constraints.
We achieve \Cref{thm:ub-Simple-D-robust} by applying \Cref{cor:hellQuant} to $\tilde{p}$ and $\tilde{q}$, the LFDs under $\epsilon$-contamination.

\begin{proof}
Consider the setting without any communication constraints.
Let $f(n, \phi, \tilde{p}, \tilde{q})$ be the probability of error by using the test $\phi$ on $n$ i.i.d.\ samples under $\tilde{p}$ and $\tilde{q}$.
Concretely, we define 
	\begin{align*}
	f(n, \phi, \tilde{p}, \tilde{q}) :=   \P_{(x_1,\dots,x_n) \sim \tilde{p}^{\otimes n}}( \phi(y_1,\dots,y_n) \ne p) +   \P_{(x_1,\dots,x_n) \sim \tilde{q}^{\otimes n}}( \phi(y_1,\dots,y_n) \ne q).
	\end{align*}
 Huber~\cite{Huber65} showed that when the corruption is in total variation distance, there exist $p_1\in\cP_1$ and $q_1 \in \cP_2$, called the least favorable distributions (LFDs), which maximize $ \inf_{\phi} f(n, \phi, \tilde{p}, \tilde{q})$ over $\tilde{p}$ and $\tilde{q}$ in $\cP_1$ and $\cP_2$, respectively.
Moreover, the optimal test is a likelihood ratio test, where the likelihoods are computed with respect to $p_1$ and $q_1$ (this corresponds to a clipped likelihood ratio test when likelihoods are computed with respect to $p$ and $q$). 
Thus, the robust sample complexity is $\nstar := \nstar_\robust(\{p,q\},\epsilon) = \Theta(1/\hel^2(p_1,q_1))$.

For the communication-constrained setting, Veeravalli, Basar, and Poor~\cite{VeeBP94}
showed that $p_1$ and $q_1$ above are also LFDs for $\cB_\robust(\{p,q\},\epsilon,\cT_D)$.
By applying \Cref{cor:hellQuant} to $p_1$ and $q_1$, we see that there exists a threshold channel $\bT$ such that $\hel^2(\bT p_1,\bT q_1) \geq \hel^2(p_1,q_1)$ up to a logarithmic factor.
Let $\phi^*$ be the likelihood ratio test between $\bT p_1$ and $\bT q_1$.
Applying Veeravalli, Basar, and Poor~\cite[Theorem 1]{VeeBP94}, we conclude that the probability of error for this test, for any $(\tilde{p}, \tilde{q}) \in \cP_0 \times \cP_1$, is less than the error on $(p_1,q_1)$.
Since the latter is less than $0.1$ when $n \geq n_0 =  \Theta(1/\hel^2(\bT p, \bT q))$ by \Cref{fact:testing}, the sample complexity is at most $n_0$.
As $\bT$ is computed using \Cref{cor:hellQuant} and $\nstar = \Theta(1/\hel^2(p_1,q_1))$, we have  $\nstar_\robust(\{p,q\},\epsilon,\cT_D) \leq \nstar \max\{1, \min\{k, \log \nstar\}/D\}$. 

Finally, we comment on the runtime of the algorithm. The LFDs can be calculated in polynomial time, as outlined in Huber~\cite{Huber65,HubStr73}, and given these LFDs, the optimal channel $\bT$ can again be computed in polynomial time, as mentioned in \Cref{cor:hellQuant}.
\end{proof}

As the following remark shows, the sample complexity of robust testing crucially depends on $\epsilon$:

\begin{remark}[Sample complexity without communication constraints] The sample complexity $\nstar_\robust(\{p,q\} , \epsilon')$ may have phase transitions with respect to $\epsilon'$.
For example, $\nstar_\robust(\{p,q\} , \epsilon')$ in \Cref{exm:robust-1} satisfies $\nstar_\robust(\{p,q\} , \epsilon^{1 + \beta}) = \Theta(1/ \epsilon^{1 + \alpha}) $ for $\beta > \alpha$ (small corruption) and $\Theta(1/ \epsilon^{2})$ for $\beta \in [0 , \alpha)$ (large corruption). See \Cref{exm:robust-2} for another instance. 
\end{remark}

It is instructive to compare the guarantees of \Cref{thm:ub-Simple-D-robust} with the sample complexity of Scheffe's test: \Cref{exm:robust-1,exm:robust-2} show that Scheffe's test may be strictly suboptimal in some regimes. 

Finally, we present the following result showing that  the channels in \Cref{thm:ub-Simple-D} are moderately robust:
\begin{restatable}[Optimal channels are moderately robust]{proposition}{PropRobustnessFree}
\label{prop:robustness-for-free} Let $p$ and $q$ be two distributions over $[k]$. Define $\epsilon_0 := c \dtv^2(p,q) \cdot \min \left\{1, \frac{D}{\log(1/\dtv(p,q))}\right\}$ for a small enough constant $c$.\footnote{This upper bound can be generalized to $\epsilon_0 := c \hel^2(\bT^*p,\bT^*q)$.}
Let $\bT^*$ be a channel that maximizes $\hel^2(\bT p, \bT q)$ over $\bT \in \cT_D$.
Let $\nstar_D$ be the sample complexity of $\bT^*$ for $p$ and $q$ (recall that $\nstar_D = \Theta(\nstar(\{p,q\},\cT_D))$). 
Let $\phi^*$ be the corresponding optimal test, which corresponds to a likelihood ratio test between $\bT^* p$ and $\bT^* q$. Then there exists a test $\phi'$ that uses $\bT^*$ for each user and solves $\cB_\robust(\{p,q\}, \epsilon_0, \cT_D)$ with sample complexity $\Theta(\nstar_D)$.
\end{restatable}

\Cref{prop:robustness-for-free} implies that optimal communication channels are already $\Theta(\epsilon^2)$-robust up to logarithmic factors. However, the result falls short of our desired goal of designing an $\Theta(\epsilon)$-robust test. (Informally, we say a channel is $\epsilon'$-robust if it can be used to perform hypothesis testing with reasonable sample complexity despite $\epsilon'$-corruption.)
 This guarantee is roughly the best possible, as can be seen by taking $\alpha \to 1$ in \cref{exm:robust-1}.

\begin{proof}
Let $\tilde{p}$ and $\tilde{q}$ be arbitrary distributions satisfying $\dtv(p,\tilde{p}) \leq \epsilon_0$ and $\dtv(q, \tilde{q}) \leq \epsilon_0$.
Suppose the following two conditions hold: (i) $n \epsilon_0 \leq 0.01$, and (ii) $n \geq C/\hel^2(\bT^*p, \bT^*q)$ for a large enough constant $C$.
We will first demonstrate the existence of a test that works under these two conditions. 

Sub-additivity of the total variation distance (cf.\ \Cref{fact:div}) implies that $\dtv(p^{\otimes n}, \tilde{p}^{\otimes n}) \leq n \epsilon_0$ and $\dtv(q^{\otimes n}, \tilde{q}^{\otimes n}) \leq n \epsilon_0$.
Under condition (i) above, the probability of each event $\cE$ over the output of $\phi$ is the same under $p^{\otimes n}$ (similarly, $q^{\otimes n}$) and $\tilde{p}^{\otimes n}$ (similarly, $ \tilde{q}^{\otimes n} $), up to an additive error of $0.01$.
Thus, if $(\phi^*,\cR)$ succeeds with probability $0.92$ for $p$ and $q$ with at most $n$ samples, they succeed with probability $0.9$ under $\tilde{p}$ and $\tilde{q}$.
The former condition holds under condition (ii), by \Cref{fact:fail-prob-samp-complexity}.
Thus, when both of these conditions hold simultaneously, the probability of error of $(\phi^*,\cR)$ under $\tilde{p}$ and $\tilde{q}$ is at most $ 0.1$.
These two conditions on $n$ are satisfied if $n = n_0 := C/\hel^2(\bT^*, \bT^* q)$ and $\epsilon_0 \leq 0.01/ n_0$.

We now define $\phi': \cup_{n=1}^ \infty \cY^{n} \to \{p,q\}$, as follows: If $n < n_0$, define $\phi'$ arbitrarily; otherwise, discard $n-n_0$ samples\footnote{A more efficient strategy is to divide the samples into $\lfloor n/n_0 \rfloor$ buckets of size $n_0$ (discarding samples if necessary), apply $\phi^*$ on each of those buckets individually, and then output the median.} and define $\phi'(y_1,\dots,y_{n})$
  to be $\phi^*(y_1,\dots,y_{n_0})$. 
By our previous calculations, it follows that $(\phi',\cR)$ solves $\cB_\robust(\{p,q\}, \epsilon)$ as long as $\epsilon \leq \epsilon_0 := 0.01/ n_0 = \Theta(1/\hel^2(\bT^* p, \bT^* q))$. 
By \Cref{cor:hellQuant}, we have $\hel^2(\bT^* p, \bT^* q) \geq c \hel^2(p,q)/ \log(1/\hel^2(p,q))$.
Applying \Cref{fact:div}, we obtain the desired result.
\end{proof}

\section{Simple $M$-ary hypothesis testing} %
\label{sec:m-ary}

In this section, we study the $M$-wise simple hypothesis testing problem, i.e., $\cP$ is a set of $M \geq 2$ distributions. Our focus in this section will be slightly different from that of \cref{sec:BinHypTest} in the following ways: (i) in addition to the choices of identical or non-identical channels, we will also allow channels to be selected adaptively; and (ii) our primary focus will be on studying the effect of $M$, the number of hypotheses, instead of the pairwise distance, i.e., $\min_{p,q \in \cP: p\neq q } \hel(p,q)$.

\begin{definition}[Sequentially adaptive channels]
Let $\cX$ be the domain, $\cP$ a family of distributions over $\cX$, and $\cT$ a family of channels from $\cX$ to $\cY$.
Let  $(U_1,\dots,U_n)$ denote $n$ (ordered) users. Each user $U_i$ observes a random variable $X_i$ i.i.d.\ from an (unknown) $p' \in \cP$.
The observations are then released sequentially, as follows:
for each time $i \in [n]$,
user $U_i$ first selects a channel $\bT_i \in \cT$ based on $X_i$ (personal sample) and $(Y_1,\dots,Y_{i-1})$ (public knowledge up to now), generates $Y_i = \bT_i(X_i)$, and finally releases $Y_i$ to everyone. 
The central server $U_0$ observes $Y_1,\dots,Y_n$ and constructs an estimate $\widehat{p} = \phi(Y_1,\dots,Y_n)$.
Both the hypothesis testing task and its sample complexity are defined analogously to \Cref{def:simpleHypothesisTestComm,def:sampleComplexityComm}. When $\cT$ is the set of channels that map to $D$ alphabets, i.e., $\cT = \cT_D$, we denote the sample complexity by $\nada(\cP,\cT_D)$.
\end{definition}

\subsection{Upper bounds}
In the following result, we show using a standard argument that we can use a communication-efficient binary test from \cref{thm:ub-Simple-D} as a subroutine to solve the $M$-wise hypothesis testing problem.

\begin{restatable}[Upper bounds using threshold tests]{proposition}{PropUppBoundsMary} 
\label{prop:upper-bounds-M-ary}
Let $\cP$ be a set of $M$ distributions in $\Delta_{k}$ such that $ \rho = \min_{p,q \in \cP: p \neq q} \hel(p,q)$. 
Let $k' = \log(1/\rho)$ and define the blow-up factor $R := \frac{\min\{k,\log(1/\rho)\}}{D} + 1$. Then the sample complexity of the simple $M$-ary hypothesis testing problem satisfies the bounds
\begin{enumerate}
  \item  $\nnid(\cP,\cT_D) \lesssim \frac{M^2 \log M}{\rho^2} \cdot R$,
  \item  $\nada(\cP,\cT_D) \lesssim  \frac{M \log M}{\rho^2} \cdot R $.

  \end{enumerate}
\end{restatable}

The proof of \Cref{prop:upper-bounds-M-ary}, given below, proceeds by analyzing a standard tournament procedure, which we now briefly describe.
We think of each hypothesis as a player, and each hypothesis test between any two distributions (players) as a game. 
 The tournament procedure decides the fixtures of the games (which distributions will play against each other) and the overall winner of the tournament (the hypothesis that will be returned) based on the results of individual games. 
 It is easy to see that the true distribution $p \in \cP$ will never lose a game against any of the competitors, with high probability. 
 Thus, as long as we have a unique player who has not lost a single game, we can confidently choose it to be the winner. 
 An obvious strategy is to organize all pairwise $\Theta(M^2)$ tests and output the player who never loses a game.
 Sans communication constraints, each game (hypothesis test) can be played with the same set of samples, and since the failure probability is exponentially small, it suffices to take $O(\log M)$ samples so that the results of all the games involving player $p$ are correct. 
However, under communication constraints, we observe the samples only
after they have passed through a channel. Furthermore, the channel that would ideally be employed for the game between $p$ and $q$ crucially relies on $p$ and $q$, and it is unclear if the same channel provides useful information for the game between $p$ and another player $q'$.
We can circumvent this obstacle by using a new channel and a fresh set of samples for each game, which guarantees correctness after taking $O(M^2 \log M)$ samples.
Recall that the choice of channels was non-adaptive here---when the channels can be adaptive, we can reduce the number of games (and thus the sample complexity) by organizing a ``knock-out'' style tournament of $M-1$ games, where each losing player is discarded from the tournament.  

\begin{proof}[Proof of \Cref{prop:upper-bounds-M-ary}]
Let $\cP =\left\{p^{(1)},\dots,p^{(M)}\right\}$.
We first prove the results for non-adaptive, non-identical channels.
Denote the set $\cS = \{\{i,j\}: i \neq j, i \in [M], j \in [M]\}$.
For each $\{i,j\} \in \cS$, let $\bT_{\{i,j\}} \in \cTT_{D}$ be the channel achieving the guarantee in \cref{cor:hellQuant}.
Since $\hel^2(p^{(i)},p^{(j)}) \geq \rho^2$, \cref{cor:hellQuant} states that $\hel^2(\bT_{\{i,j\}}p^{(i)}, \bT_{\{i,j\}}p^{(j)}) \geq \rho^2/R$.
Let $m = (CR \log M)/ \rho^2$, for a large enough constant $C>0$ to be decided later.

Fix any ordering $\sigma(\cdot)$ of the set $\cS$.
Consider the strategy where we take a total of $0.5M(M -1)m$ users, such that the $r^{\text{th}}$ user uses the channel $\bT_{\sigma(\lceil r/m \rceil)}$, i.e.,  each channel is repeated $m$ times in a predetermined order.
For any $\{i,j\} \in \cS$, let $\cA_{\{i,j\}}$ denote the set of samples observed by the central server after passing through $\bT_{\{i,j\}}$.

We now describe the strategy at the central server: For any $\{i,j\} \in \cS$, consider the optimal test $\psi_{\{i,j\}}$ between $\left\{\bT_{\{i,j\}}p_i, \bT_{\{i,j\}}p_j\right\}$ that uses the samples $\cA_{\{i,j\}}$ and maps to either $\{i\}$ or $\{j\}$.
We say this is a game between $i$ and $j$, and call $\psi_{\{i,j\}}\left(\cA_{\{i,j\}}\right)$ the winner of the game.
The central server outputs the unique hypothesis that wins all of its games against other hypotheses, i.e., the unique element in the set $\{i: \forall j \neq i, \psi_{\{i,j\}}\left(\cA_{\{i,j\}}\right) = i\}$.

Let $i\in [M]$ be the unknown true hypothesis. It suffices to show that $i$ never loses a game against any other hypothesis.
For any $j \neq i$, we have $\hel^2(\bT_{\{i,j\}} p^{(i)},\bT_{\{i,j\}} p^{(j)}) \geq \rho^2/R$. Thus, we have  $\P(\psi(\cA_{\{i,j\}}) \neq i) \leq 0.01/M^2$ by \Cref{fact:fail-prob-samp-complexity}, since $C$ is large enough. Taking a union bound over all $j \neq i$, we see that the probability of error is less than $0.01/M$. Taking the sum, we see that the sum of the probabilities of errors satisfies condition \eqref{eq:DefErrorProbCommCons}. 
Thus, we obtain $\nnid(\cP,\cT_D) \lesssim M^2 m \lesssim \frac{M^2 \log M}{\rho^2} \cdot R$.

We now turn our attention to the adaptive setting. Consider the following strategy:
\begin{enumerate}
  \item Set $j = 2$ and $\widehat{i}= 1$.
  \item While $j \leq M$:
  \begin{enumerate}
    \item $m$ users choose $\bT_{\{j, \widehat{i}\}}$.
    \item Let $\cA_{\{j, \widehat{i}\}}$ be the set of $m$ observed samples.
    \item Assign $\widehat{i} \gets \psi_{j,\widehat{i}}(\cA_{\{j, \widehat{i}\}})$ and $j \gets j +1$.
  \end{enumerate}
  \item Output $\widehat{i}$.
 \end{enumerate}
First, it is easy to see that the procedure terminates after taking $Mm$ samples.
Turning to the correctness of the algorithm, let $i^*$ be the true unknown probability distribution. It suffices to show that $i^*$ never loses a game against any other $j$. The same arguments as above show that this does not happen.
\end{proof}

\begin{remark}[Dependence on $M$ and $\rho$]
\label{rem:AdapUppBound}
We now comment on the dependence of these bounds on $M$ and $\rho$. In particular, note that:
\begin{enumerate}
\item As shown later in \cref{thm:Adaptive-lb}, the dependence on $M$ is nearly tight (up to logarithmic factors) for the case of adaptive algorithms (for constant $\rho$ and $D)$.
\item For non-adaptive algorithms and $D=2$, \Cref{thm:M-ary-Identical-lb} shows a lower bound of $\Omega(M^2)$ for the case of identical channels.
\item The dependence on $\rho$ is tight for constant $M$ (\Cref{thm:BinHypTestLowerBound}).
\end{enumerate}
\end{remark}

\begin{remark}[Robust $M$-ary hypothesis testing]\label{rem:Robust} One can also consider a robust version of simple $M$-ary hypothesis testing, which is often called hypothesis selection, i.e., the true distribution $p$ satisfies $\min_{q \in \cP}\dtv(p,q) \leq \epsilon$ (analogous to \Cref{def:robustHyp}). 
One can use the robust test for binary hypothesis testing from \Cref{sec:robust_tests_binary} to obtain a similar dependence on $M$ as in \Cref{prop:upper-bounds-M-ary} under this setting.
\end{remark}

As the dependence on $M$ is nearly tight for adaptive channels (for constant $\rho$), we now shift our attention to procedures that use identical channels, which might be desirable in certain practical situations.
We establish the following bound for identical channels:

\begin{theorem}[Upper bounds with identical channels]
\label{thm:M-ary-Ident-Chann-ub}
Let $\cP$  be a set of  $M$ distributions in $\Delta_{k}$ satisfying $\min_{p, q \in \cP: p \neq q} \dtv(p,q)> \epsilon$. Then
\begin{align}
\label{eq:MaryIdentChan-ub1}
\nide(\cP, \cT_D) \lesssim \frac{D \log M}{\epsilon^2} \min \left\{\log(DM^2) M^{6+\frac{4}{D-1}},    kM^{\frac{4}{D-1}} \log(Dk)      \right\}.
\end{align}
In particular, for $D = \Omega (\log(M))$, we have
\begin{align}
\label{eq:MaryIdentChan-ub2}
\nide(\cP, \cT_D) \lesssim \frac{\log^2M}{\epsilon^2} \min\left\{ M^6 , k\log k \right\}
\end{align}
(since $\nide(\cP, \cT_D)$ decreases in $D$). Furthermore, for any $p$ and $q$, the channel achieving the rates in inequalities~\eqref{eq:MaryIdentChan-ub1} and \eqref{eq:MaryIdentChan-ub2} can be found efficiently using a linear program of polynomial size.

\end{theorem}

\begin{remark}
If public randomness is available to all users and the central server, it is possible to simulate the protocol with non-identical channels from Proposition~\ref{prop:upper-bounds-M-ary} using identical channels, with more samples. To do this, each user chooses a channel from the non-identical channel setting uniformly at random (using public randomness). Since the central server knows what channels are being picked, it is able to solve hypothesis testing when all channels from the non-identical setting have been picked. For this to happen, a coupon-collector argument shows that it suffices to have $\asymp \nnid (\cP, \cT_D) \log \nnid(\cP, \cT_D)$ many users. The resulting upper bound will be an improvement over the one presented above in certain regimes using \Cref{prop:upper-bounds-M-ary}.
\end{remark}

\begin{proof}
Let $\cP = \left\{p^{(1)},\dots, p^{(M)}\right\}$.
We prove the bound~\eqref{eq:MaryIdentChan-ub1} by reducing the problem to a (decentralized) testing problem between distributions $\cP'=\left\{ q^{(1)},\dots,q^{(M)} \right\}$,
 where $q^{(i)} \in \Delta_D$ and $q^{(i)} = \bT p^{(i)}$, for some $\bT \in \cT_D$.
Defining $\epsilon' = \min_{i \neq j} \dtv\left(\bT q^{(i)}, \bT q^{(j)}\right)$, \cref{fact:testing} shows the existence of an algorithm with sample complexity $ O\left(\frac{\log M}{\epsilon'^2}\right)$.
Thus, the goal is to find a channel $\bT$ that maximizes $\min_{i \neq j} \dtv\left(\bT p^{(i)}, \bT p^{(j)}\right)$, leading to the linear program
\begin{align}
\label{eq:M-ary-LP}
\max_{\bT \in \cT_D} \min_{i \neq j} \dtv(\bT p^{(i)}, \bT p^{(j)}).
\end{align}
Let $\opt$ be the value of the maximum in expression~\eqref{eq:M-ary-LP}. The overall sample complexity of the algorithm is then $O\left(\frac{\log M}{\opt^2}\right)$.
We now prove each of the two bounds in inequality~\eqref{eq:MaryIdentChan-ub1} by lower-bounding $\opt$ in two different ways.

\paragraph{Bound II} %
\label{par:bound_ii}
We first prove the second bound in inequality~\eqref{eq:MaryIdentChan-ub1} of \Cref{thm:M-ary-Ident-Chann-ub}. 
The following result, proved in \Cref{app:upper_bounds_for_identical_channels}, provides a lower bound on the quantity \eqref{eq:M-ary-LP} by using a Johnson-Lindenstrauss (JL) type of sketch:
\begin{restatable}[JL-sketch]{lemma}{JLSkecth}
\label{prop:JLMainBody} There exists a constant $c > 0$ such that the following holds:
Let $\left\{p^{(1)},\dots,p^{(M)}\right\} \subseteq \Delta_k$ be $M$ distributions such that $\min_{i \neq j} \dtv(p^{(i)},p^{(j)}) > \epsilon$. Then
\begin{align*} 
\max_{\bT \in \cT_D} \min_{i \neq j} \dtv(\bT p^{(i)}, \bT p^{(j)}) \geq c \cdot \frac{\epsilon}{\sqrt{k} M^{\frac{2}{D-1}} \sqrt{D \log(Dk)} }.
\end{align*}
\end{restatable}

\paragraph{Bound I.} %
We note that the first bound in inequality~\eqref{eq:MaryIdentChan-ub1} is better than the second bound when $k \gg M$. 
Our strategy will be to reduce the problem from a domain with $k$ elements to a domain of (potentially) smaller size.
\begin{claim}[Reduction to a domain of size $M^2$]
\label{claim:IniChannelMSq} Let $\cP = \left\{p^{(1)},\dots,p^{(M)}\right\}$ and consider the setting of \cref{thm:M-ary-Ident-Chann-ub}.
There exists a channel $\bT: [k] \to [M^2]$ such that for all $1 \leq i < j \leq M$, 
we have $\dtv\left(\bT p^{(i)}, \bT  p^{(j)}\right) \geq \frac{1}{M^2} \cdot \dtv \left(p^{(i)},  p^{(j)}\right)$.
\end{claim}
\begin{proof} 
Note that the result holds trivially for $k \leq M^2$, so assume that $k > M^2$.
Let $d = {M \choose 2}$.
For two distributions $p$ and $q$, we have $\dtv(p,q) = \frac{1}{2} \|p -q\|_1$. 
Thus, we will show the existence of a column-stochastic matrix $\bT \in \R^{d \times k}$, i.e.,
each entry of $\bT$ is non-negative and the sum of each column is $1$, 
 satisfying the following conclusion when interpreted as an inequality concerning matrices and vectors:
$\left\|\bT (p^{(i)} - p^{(j)})\right\|_1 \geq \frac{1}{M^2}\cdot\left\|p^{(i)} - p^{(j)}\right\|_1$.

We will index rows by $(i,j)$, for $1 \leq i < j \leq M$.
We first define a matrix $\bT' \in \real^{d \times k}$,
such that the $(i,j)^{\text{th}}$ row is the vector $z'_{(i,j)} \in \R^{k}$ with $\ell^{\text{th}}$ entry equal to $\I_{p^{(i)}(\ell) > p^{(j)}(\ell)}$.
It is easy to see that
\begin{align}
\label{eq:ChSmallDomain}
\left\|\bT'p^{(i)} - \bT'p^{(j)}\right\|_1 = \left|\left\langle z'_{(i,j)}, p^{(i)} - p^{(j)}\right \rangle\right| = \frac{1}{2} \left\|p^{(i)} - p^{(j)}\right\|_1.
\end{align}
We now construct $\bT$ by transforming $\bT'$ into a column-stochastic matrix by dividing each column by the sum of its entries.
Let $\left\{z_{(i,j)}\right\}$ denote the rows of $\bT$.
As each entry of $\bT'$ is at most 1 and the number of rows is $d$, each entry of $\bT$ is at least $\frac{1}{d}$ times the corresponding entry of $\bT'$, i.e., $z_{i,j} \geq \frac{z'_{i,j}}{d}$, interpreted as an entrywise inequality.

Thus, for any $1 \leq i < j \leq M$, we have
\begin{align*}
\left\|\bT p^{(i)} - \bT p^{(j)}\right\|_1  \geq \left|\left\langle z_{i,j}, p^{(i)} - p^{(j)}\right\rangle\right | \geq \frac{1}{d} \left|\left\langle z'_{(i,j)}, p^{(i)} - p^{(j)}\right \rangle\right|,
\end{align*}
noting that each entry in the sum $\langle z_{i,j} , p^{(i)} - p^{(j)} \rangle$ is nonnegative by construction. %
Combining with inequality~\eqref{eq:ChSmallDomain}, we obtain
\begin{align*}
\left\|\bT p^{(i)} - \bT p^{(j)}\right\|_1 \geq \frac{1}{2d} \left\|p^{(i)} - p^{(j)}\right\|_1 \geq \frac{1}{M^2} \left\|p^{(i)} - p^{(j)}\right\|_1.
\end{align*}
This completes the proof.\qedhere 
\end{proof}

Returning to the original problem setting, let $\bT_1$ be a channel from \cref{claim:IniChannelMSq} that transforms $p^{(i)} \in \Delta_k$ to $q^{(i)} \in \Delta_{M^2}$, such that $\min_{i \neq j}\dtv(q^{(i)},q^{(j)}) \geq \frac{\epsilon}{M^2}$.
Define $\epsilon' = \frac{\epsilon}{M^2}$ and $k' = M^2$.
Applying \cref{prop:JLMainBody}, there exists a channel $\bT_2: [k'] \to [D]$ such that for all $i \neq j$, we have
\begin{align*}
\dtv\left(\bT_2 q^{(i)}, \bT_2 q^{(j)}\right) \gtrsim  \frac{\epsilon' }{\sqrt{k'}} \frac{1}{M^{\frac{2}{D-1}}\sqrt{D \log(Dk')}} =  \frac{\epsilon}{M^{3 + \frac{2}{D-1}} \sqrt{D \log(DM^2)}}.
 \end{align*}
We define the final channel $\bT: [k] \to [D]$ to be the concatenation of the two channels $\bT_1$ and $\bT_2$. In matrix notation, this corresponds to $\bT:= \bT_2 \times \bT_1$.
Then $\opt \gtrsim \frac{\epsilon}{M^{3 + \frac{2}{D-1}} \sqrt{D \log(DM^2)}}$. 
\end{proof}

\subsubsection{Proof of \Cref{prop:JLMainBody}}
\label{app:upper_bounds_for_m-ary}

\label{app:upper_bounds_for_identical_channels}
This section contains the proof of \cref{prop:JLMainBody} that was omitted above.
\JLSkecth*

\begin{proof}
Let $D' := D - 1$.
Consider a matrix $\bH \in \R^{D' \times k}$ that satisfies the following constraints: $H_{i,j} \geq 0$ for all $(i,j)$, and $\sum_{i=1}^{D'} H_{i,j} \leq 1$ for all $j \in [k]$.
Let $\cH_{D'}$ be the set of all such matrices.
It is easy to see that given any matrix $\bH \in \cH_{D'}$, it is possible to generate a unique matrix $\bT \in \cT_D$ by adding an extra row to make the column sums 1. Consider such pairs $(\bH,\bT)$ in $\cH_{D'} \times \cT_D$, and note that $ \left\|\bH p^{(i)} - \bH p^{(j)}\right\|_1 \leq  \left\|\bT p^{(i)} - \bT p^{(j)}\right\|_1 = 0.5\dtv\left(\bT p^{(i)}, \bT p^{(j)}\right)$.
We will generate $\bH$ randomly such that, with positive probability, it belongs to $\cH_{D'}$ and $\min_{i \neq j} \left\|\bH p^{(i)} - \bH p^{(i)}\right\|_1$ is large.

We will show the following result:
\begin{lemma}
\label{lem:JlSkeChannel}
There exists a constant $c > 0$ such that, for any $A = \{ a_1,\dots,a_N \} \subseteq \R^{k}$ such that the sum of the components of each $a_i$ is equal to 0, there is a linear map $\bH \in \cH_{D'}$ such that the following holds:
\begin{align*}
\|\bH a\|_2 \geq c \|a\|_2 \frac{1}{ \sqrt{D'\log(D'k)}} N^{-\frac{1}{D'}}, \qquad \forall  a \in A.
\end{align*}
\end{lemma}

Before giving the proof of \cref{lem:JlSkeChannel}, we show how to use it to complete the proof of \cref{prop:JLMainBody}.
Let $A = \left\{ p^{(i)} - p^{(j)}: 1 \leq i < j \leq M\right\}$, and observe that $A$ satisfies the conditions of \cref{lem:JlSkeChannel} with $N \leq M^2$.
Using the $\bH$ in \cref{lem:JlSkeChannel} and the fact that for $x \in \R^p$, we have $ \|x\|_1 \geq \|x\|_2 \geq \|x\|_1/\sqrt{p}$, we obtain
\begin{align*}
\left\|\bH\left(p^{(i)}- p^{(j)}\right)\right\|_1 & \geq \left\|\bH \left(p^{(i)} - p^{(j)}\right)\right\|_2 \geq 2c \cdot \left\|\left(p^{(i)} - p^{(i)}\right)\right\|_2 \cdot \frac{1}{ \sqrt{D\log(Dk)}} M^{-\frac{2}{D'}}  \\
&\geq 2c \cdot \left\|\left(p^{(i)} - p^{(i)}\right) \right\|_1 \cdot \frac{1}{\sqrt{k}} \frac{1}{ \sqrt{D\log(Dk)}} M^{-\frac{2}{D'}} \\
&\geq c \cdot  \frac{\epsilon}{\sqrt{k}} \frac{1}{ \sqrt{D\log(Dk)}} M^{-\frac{2}{D'}},
\end{align*}
where we use the fact that $\dtv\left(p^{(i)},p^{(j)}\right) = \frac{1}{2} \left\|p^{(i)} - p^{(j)}\right\|_1$.
This completes the proof of \cref{prop:JLMainBody}.
\end{proof}

We now provide the proof of \cref{lem:JlSkeChannel} that was omitted above.

\begin{proof}[Proof of \cref{lem:JlSkeChannel}]
Let $Q_1, Q_2>0$ be numbers to be determined later. Let $\bJ \in \R^{D' \times k}  $ be the matrix of all ones, and let $\bG\in \R^{D' \times k}$ be a matrix with i.i.d.\ $\cN(0,1)$ entries $\{G_{i,j}\}$.  
We will choose $\bH$ to be of the following form:
\begin{align*}
\bH := \frac{1}{Q_1} \left( \bJ + \frac{\bG}{Q_2} \right).
\end{align*}
The following claim shows that with probability at least $0.9$, we have $\bH \in \cH_{D'}$ (the proof is given later).
\begin{claim}
\label{lem:TValChannel}
If  $Q_2  \geq 10\sqrt{\log(kD')}$ and $Q_1 \geq D' + 10 \sqrt{D' \log k}/Q_2$, then with probability at least $9/10$, we have $\bH \in \cH_{D'}$.
\end{claim}
We will now show that with high probability, $\bH$ preserves the Euclidean norm of each $a \in A$:
\begin{claim}
\label{lem:ProbTL2Dist}
There exists a constant $c ' > 0$ such that with probability at least $9/10$, we have
$$\|\bH a\|_2 \geq c' 
\frac{\|a\|_2}{Q_1Q_2} \sqrt{D'} N^{-\frac{1}{D'}}, \qquad \forall a \in A.$$
\end{claim}
Given \cref{lem:TValChannel,lem:ProbTL2Dist}, if we choose $Q_2 = 10\sqrt{\log(k D')}$ and $Q_1 = 11D'$, then with probability at least $0.8$, we have $\bH \in \cH_{D'}$, and for all $a \in A$, we have
\begin{align*}
\|\bH a\|_2 \geq c'\|a\|_2 \frac{1}{ \sqrt{\log(Dk)}} \sqrt{D} N^{-\frac{1}{D'}}.
\end{align*}
This completes the proof of \cref{lem:JlSkeChannel}.
\end{proof}

We now provide the proofs of intermediate results that we have used.
\begin{proof}[Proof of \cref{lem:TValChannel}]
We need to ensure that the following two events hold simultaneously:
\begin{align*}
\cE_1 &:= \left\{ \forall (i,j) \in [D']\times [k]: \,\, H_{i,j} \geq 0 \right\}, \\
\cE_2 &:= \left\{ \forall j \in [k]: \,\,  \sum_{i=1}^{D'} H_{i,j} \leq 1 \right\}.
\end{align*}

The event $\cE_1$ holds when for all $(i,j)$, we have $1 + G_{i,j}/Q_2 \geq 0$ if and only if $G_{i,j} \geq - Q_2$.
Thus, it suffices to take $Q_2 > \sup_{i \in [D'], j \in [k]} |G_{i,j}|$.
Taking $Q_2 = 10\sqrt{\log(kD')}$, standard results on maxima of Gaussian random variables \cite{Wainwright19} imply that $\cE_1$ holds with probability at least $0.95$.

We now focus on $\cE_2$. For $j \in [k]$, letting $Z_j := \sum_{i=1}^{D'} H_{i,j} = \sum_{i=1}^{D'} \frac{1}{Q_1}(1 + G_{i,j}/Q_2)$,
we have $Z_j - \frac{D'}{Q_1}\sim \cN\left( 0 , \frac{D'}{Q_1^2Q_2^2}  \right)$.
A union bound then implies that with probability $0.95$, for all $Z_j \in [k]$, we have
\begin{align*}
Z_j \leq \frac{D'}{Q_1} + 10 \sqrt{\frac{D'}{Q_1^2Q_2^2}} \sqrt{\log(k)} = \frac{1}{Q_1}\left( D' +  10\frac{\sqrt{D' \log k}}{Q_2} \right).
\end{align*}
For this to be at most $1$, we need $Q_1 \geq D' +  10\frac{\sqrt{D' \log k}}{Q_2}$. By a union bound, the events $\cE_1$ and $\cE_2$ hold simultaneously with probability at least $0.9$.
This completes the proof.
\end{proof} 

\begin{proof}[Proof of \cref{lem:ProbTL2Dist}]
Without loss of generality, we will assume that $\|a\|_2 = 1$.
It suffices to show that for all $a \in A$, with probability at least $1 - \frac{1}{10 N}$, we have $\|\bH a\|_2 \geq c'\frac{1}{Q_1 Q_2} \sqrt{D'} N^{-\frac{1}{D'}}$.
Equivalently, we will show that $\P\left( \|\bH a\|_2 < c' \frac{1}{Q_1 Q} \sqrt{D'} N^{-\frac{1}{D'}} \right) \leq \frac{1}{10N}$.
For any $a\in A$, we note that $\bJ a$ is a zero vector, so
\begin{align*}
\bH a = \frac{1}{Q_1} \left( \bJ + \frac{\bG}{Q_2} \right)a = \frac{1}{Q_1} \left( \bJ a + \frac{\bG a}{Q_2} \right) = \frac{1}{Q_1 Q_2} \bG a. 
\end{align*}
Letting $G_1,\dots,G_{D'}$ be the rows of $\bG$, and letting $\chi^2_{D'}$ be a chi-square random variable with $D'$ degrees of freedom, we have
\begin{align*}
\|\bH a\|_2^2 = \left( \frac{1}{Q_1 Q_2} \right)^2 \sum_{i=1}^{D'} (G_i^\top a)^2 \sim  \left( \frac{1}{Q_1 Q_2} \right)^2 \chi^2_{D'},
\end{align*}
since $G_i$ is an isotropic multivariate Gaussian and $a$ has unit norm.
Standard approximations for $\chi^2_{D'}$\footnote{This can be obtained by upper-bounding the pdf of the $\chi^2$ random variable and using Stirling's approximation.} imply that
\begin{align*}
\P\left(\chi^2_{D'} \leq t \right) \leq \left( \frac{et}{D'} \right)^{D'/2}.
\end{align*}
Furthermore, for $t_*= \frac{D'}{e} \left( \frac{1}{10 N} \right)^{\frac{2}{D'}}$, the expression on the right-hand side is less than $0.1/N$. Then
\begin{align*}
\P\left(\|\bH a\|_2^2 \leq \left( \frac{1}{Q_1 Q_2} \right)^2 t_*   \right) &= \P\left\{ \chi^2_{D'} \geq t_*   \right\} \leq \frac{1}{10N}.
\end{align*}
Thus, with probability at least $1 - \frac{1}{10 N}$, we have
$$\|\bH a\| \geq \sqrt{\frac{1}{e}} \left( \frac{1}{10} \right)^{ \frac{1}{D'}} \cdot 	\sqrt{D'} N^{-\frac{1}{D'}}  \cdot \frac{1}{Q_1 Q_2},$$
completing the proof of the claim with $c'= \frac{1}{\sqrt{e} 10^{1/D'}} \geq 0.001$.
\end{proof}

\subsection{Lower bounds}
\label{sec:lower_bound_for_m_ary_hypothesis_testing}

In this section, we present lower bounds for the $M$-ary hypothesis testing problem.
We begin by proving a lower bound of $\Omega(M)$ for adaptive algorithms, thus also for non-adaptive algorithms.
Although we state our results for sequentially adaptive algorithms, these lower bounds also hold for the more general blackboard protocol; see the papers \cite{BravGMNW16,SteVW16} for more details.
\begin{theorem} [Adaptive lower bounds]
\label{thm:Adaptive-lb} 
For every $M\geq 2$ and $\epsilon < 0.1$, there exist $k \in \N$ and a set of $M$ distributions $\cP_M \subseteq \Delta_k$ such that the following hold:
\begin{enumerate}
  \item (Lower bound from strong distributed data processing and direct-sum reduction~\cite{BravGMNW16}.) $k = O\left(2^M\right)$, $\nstar(\cP_M) \lesssim \frac{\log M}{\epsilon^2}$, and $\nada(\cP_M, \cT_D) \gtrsim \frac{M}{\epsilon^2 \log D}$.
  \item (Lower bound from SQ dimension~\cite{SteVW16,Feldman17}.)    $k = O(M)$, $\nstar(\cP_M) \lesssim \frac{\log M}{\epsilon^2}$, and $\nada(\cP_M, \cT_D) \gtrsim \Omega\left(\frac{M^{1/3}}{\epsilon^{2/3} D^{2/3} (\log D)^{1/3}}\right)$, as long as $M \gtrsim \frac{\log D}{\epsilon D} $.
\end{enumerate} 
\end{theorem}
The proof of \cref{thm:Adaptive-lb} is given in \cref{app:lower_bounds_for_m-ary}.

\begin{remark}[An elementary proof of $\Omega(\sqrt{M})$] We also provide an elementary proof of an $ \Omega(\sqrt{M})$ lower bound for non-adaptive channels that relies on the impossibility of $\ell_1$-embedding using linear transforms~\cite{LeeMN05,ChaSah02}. See \Cref{app:lower_bounds_from_impossibility_of_l1} for more details.
\end{remark}

\begin{theorem}[Lower bounds for identical channels and $D=2$]
\label{thm:M-ary-Identical-lb}
There exist constants $c_1 , c_2 > 0$ such that the following holds for every $M\geq 2$ and $\cP_M := \left\{ p^{(1)},\dots,p^{(M)} \right\} \subseteq \Delta_k$: Let $\epsilon_1 := \min_{i \neq j} \hel\left(p^{(i)},p^{(j)}\right)$ and $\epsilon_2:= \max_{i \neq j} \hel\left(p^{(i)},p^{(j)}\right)$. Then  $\nstar(\cP_M) \leq \frac{c_1\log M}{\epsilon_1^2}$ and $\nide(\cP_M, \cT_2) \geq \frac{c_2 M^2}{\epsilon_2^2}$. 
\end{theorem}

\begin{remark}\label{rem:Ident-lb-M-sq}
\cref{thm:M-ary-Identical-lb} is a strong lower bound in the sense that it holds for every set of distributions. By a standard volumetric argument, it is possible to construct a set of $M$ distributions $\cP_M \subseteq \Delta_k$ such that (i) $\epsilon_2 \lesssim \epsilon_1 = c$ for a constant $c$  and (ii) $M = 2^{\Omega(k)}$. As an algorithm for distribution estimation in $\dtv$ implies a testing algorithm (see, e.g., \cite{Tsybakov09}), a standard argument implies that any algorithm that uses an identical channel $\bT \in \cT_2$   and learns the underlying distribution $p$ in $\dtv$ up to a small constant requires at least $2^{\Omega(k)}$ samples.  This is in stark contrast to the setting of non-identical channels, where Acharya, Canonne, and Tyagi~\cite{AchCT20-II}
 provide an algorithm with sample complexity $O(k^2)$.
\end{remark}

\begin{proof}[Proof of \cref{thm:M-ary-Identical-lb}]
Observe that the upper bound on $\nstar(\cP_M)$ follows directly from \Cref{fact:testing}. We now focus on the lower bound.
Our goal is to show that there exists a constant $c_2$ such that 
\begin{align}
\label{eq:ideChanD2}
\sup_{\bT \in \cT_2} \min_{i \neq j} \hel(\bT p^{(i)}, \bT p^{(j)}) \leq \frac{c_2 \epsilon_2}{M},
\end{align}
since \cref{fact:testing} then implies a lower bound of $\nstar(\cP_M, \{\bT\}) \gtrsim \frac{M^2}{\epsilon_2^2}$.

Fix a channel $\bT \in \cT_2$.
Let $\cQ_M = \left\{ q^{(1)},\dots,q^{(M)} \right\}$ be the set of binary distributions obtained after transforming $\cP_M$ via the channel, where $q^{(i)} = \bT p_i$.
Since $\cQ_M$ is a set of binary distributions, each distribution $q^{(i)}$ can be represented by a single scalar parameter, which we denote by $q_i$.
Without loss of generality, let $0 \leq q_1 <  q_2  < \cdots < q_k \le 1$.
By the data processing inequality, we have $\max_{i \neq j} \hel\left(q^{(i)},q^{(j)}\right) \leq \hel\left(p^{(i)},p^{(j)}\right) \leq \epsilon_2$.
We will show that in fact, there exists an index $i^*$ such that $\hel\left(q^{(i^*)}, q^{(i^*+1)}\right) \lesssim \frac{\epsilon_2}{M}$.
Taking a supremum over all $\bT \in \cT_2$ would then establish the result in inequality~\eqref{eq:ideChanD2}.

For a $q \in [0,1]$, let $\text{Ber}(q)$ denote the Bernoulli distribution with parameter $q$.
For $0 \leq q \leq q' \leq 1/2$, we have the following (cf.\ \cref{prop:simpleHellinger}):
\begin{align}
\sqrt{q'}-\sqrt{q} \leq \hel\left(\text{Ber}(q),\text{Ber}(q')\right) \leq \sqrt{2} \left(\sqrt{q'}-\sqrt{q}\right).
\label{eq:HellingerLinear}
\end{align}
Let $r$ be the largest integer such that $q_{r} \leq 1/2$. We will assume that $r \geq \frac{M}{2}$ (otherwise, apply
the following argument to $1 - q_i$).
Using inequality~\eqref{eq:HellingerLinear} twice, we obtain
\begin{align*}
\sum_{i=1}^{r-1} \hel\left(q^{(i)},q^{(i+1)}\right) &\leq \sqrt{2} \sum_{i=1}^{r-1} \left(\sqrt{q_{i+1}} - \sqrt{q_i}\right)  = \sqrt{2} \left(\sqrt{q_r} - \sqrt{q_1}\right) \leq \sqrt{2} \hel\left(q^{(1)},q^{(r)}\right) \leq \sqrt{2} \epsilon_2. 
\end{align*}
As $r \geq \frac{M}{2}$, the preceding inequality implies that there exists some $i^*$ such that $$\hel\left(q^{(i^*)},q^{(i^*+1)}\right) \lesssim \frac{\epsilon_2}{r} \lesssim \frac{\epsilon_2}{M}.$$
\end{proof}

\section{Discussion} %
\label{sec:discussion}

We have studied the problem of simple  hypothesis testing under communication constraints. Taking a cue from past work on decentralized detection, we have focused on threshold channels and analyzed the sample complexity for a (nearly optimal) threshold channel. For simple binary hypothesis testing, we showed that this choice leads to an at most logarithmic increase in the sample complexity of the test.
We extended this result to the robust setting, where distributions may be contaminated in total variation distance. Importantly, our algorithms for hypothesis testing in the simple and robust settings were shown to be computationally efficient.
Finally, we studied $M$-ary hypothesis testing by considering settings where the channels are identical, non-identical, or adaptive. We showed that communication constraints may lead to an exponential increase in sample complexity, even for adaptive channels. For identical channels, we developed an efficient algorithm and analyzed its sample complexity. 
At a technical level, our results rely on a reverse data processing inequality for communication-constrained channels, a reverse Markov inequality, and a sketching algorithm akin to the Johnson--Lindenstrauss theorem.

There are several research directions that are worth exploring: Is adaptivity of channels useful in simple binary hypothesis testing? Can one tighten the dependence of the sample complexity on the minimum Hellinger distance between distributions for $M$-ary hypothesis testing? (Our results use total variation as a proxy for Hellinger and are likely to be loose.) It would also be interesting to study simple hypothesis testing under other constraints such as local differential privacy and memory. There has been some recent work along these lines in Pensia, Asadi, Jog, and Loh~\cite{PAJL23}. The technical tools developed in this paper, particularly the reverse data processing inequality, may also have applications in quantization via ``single-shot compression'' in information theory.

\section*{Acknowledgement}

We thank the anonymous reviewers for their feedback that has led to an improved manuscript.

\bibliographystyle{alpha}
\bibliography{ref}

\appendix

\section{Effect of public randomness on sample complexity}
\label{app: public}

We now consider a more general setting than the one in \Cref{def:simpleHypothesisTestComm}.
Suppose we perform a (possibly composite) hypothesis test between $M$ sets of distributions $D_1, \dots, D_M$. When each set is a singleton, we recover the usual $M$-ary hypothesis testing problem from \Cref{def:simpleHypothesisTestComm}. 

We assume that a public random variable $U$ is observed by $\{U_0, \dots, U_n\}$; i.e., by each user and the central server. If $X_i \sim q \in D_i$, we say that an error occurs if the declared hypothesis $\hat D$ is not equal to $D_i$. We say that $\hat D$ successfully solves the hypothesis testing problem if the maximum error is small; i.e., 
$$\max_{i \in M} \max_{q \in D_i} \mathbb P_U(\hat D \neq D_i) \leq r,$$  
for a sufficiently small constant $r>0$. Observe that the probability is also taken over the public randomness $U$.
We chose the $\max$ in the error probability for technical reasons. Moreover, a lower bound on the $\max$ formulation also implies a lower bound on the $\sum$ formulation.

We use $n^*_{\text{public}}(r)$ to refer to the resulting sample complexity above in the presence of public randomness.
When there is only private randomness, we use $n^*_{\text{private-coin, non-identical}}(r)$ for the resulting sample complexity. Our focus will be on the regime where $r$ is small enough.

We will argue that for any distribution testing problem, the following holds: 
\begin{align}
\label{eq:public-vs-private}
n^*_{\text{public}}(r) \geq  n^*_{\text{private-coin, non-identical}}\left(2r\sum_{i=1}^M |D_i|\right) .
\end{align}
In particular, for simple binary hypothesis testing, the two bounds are within constants (and thus equal to deterministic protocols by Tsitsiklis~\cite[Proposition 2.1]{Tsitsiklis93}). However, this is not the case for the simple M-ary hypothesis testing problem. Here, an extra factor of $\log M$ differentiates these two bounds (for $r$ smaller than $1/M$). In particular, for simple $M$-ary hypothesis testing and 
 $r \ll 1/M$,
  we have
\begin{align*}
 n^*_{\text{private-coin, non-identical}}\left(r\right) &\geq n^*_{\text{public}}(r)  \geq n^*_{\text{private-coin, non-identical}}\left(2rM\right) 
\gtrsim \frac{n^*_{\text{private-coin, non-identical}}\left(r\right)}{\log\left(M\right)},
\end{align*}
where the last inequality follows by a standard application of boosting. 

Note that the above analysis is only useful if the sets $D_i$ are finite. When this is not the case, public-coin protocols may be superior to private-coin protocols, as hinted at in Acharya, Canonne, and Tyagi~\cite{AchCT20-II}. 
Also, note that inequality~\eqref{eq:public-vs-private} does not contradict the exponential lower bound of identical private-coin protocols of \Cref{rem:Ident-lb-M-sq} since inequality~\eqref{eq:public-vs-private} does not provide any comparison between identical and non-identical private-coin protocols.

We now provide the proof of inequality~\eqref{eq:public-vs-private}. 
\begin{proof}
Let $t \geq 1$ be a value to be decided later. 
For every $i \in [M]$  and $q \in D_i$, define the set $V_q(rt)\subseteq {\cal U}$ to be the set of $u \in {\cal U}$ such that the probability of error with $n^*_{\text{public}}(r)$ samples, conditioned on $U=u$, is at most  $rt$.
By Markov's inequality, the probability $\mathbb P(U \in V_q(rt))$ is at least  $1 - 1/t$. Define $V(rt)$  to be the intersection of all these $V_q(rt)$'s, i.e., $V(rt) = \bigcap_{i \in M} \bigcap_{q \in D_i} V_q(rt)$. By a union bound, we have 
$$\mathbb P(U \in V(rt)) \geq 1 - \frac{\sum_{i=1}^M |D_i|}{t}.$$ 
This probability is positive for $t= 2 \sum_{i=1}^M |D_i|$, so the set $V(rt) $ is non-empty.
Observe that for any $u \in V(rt)$, we are in the non-interactive, non-identical, private-coin protocol setting for the original $M$-ary testing problem, with the probability of error blown up to be $rt$ instead of   $r$. 
Since the set $V(rt)$ is non-empty, there is a non-interactive, non-identical, private-coin protocol with $n^*_{\text{public}}(r)$ samples with error at most $rt$, by choosing the protocol for $u \in V(rt)$. 
Thus, we obtain inequality~\eqref{eq:public-vs-private}.
\end{proof}

\section{Reverse data processing} %
\label{app:reverse_data_processing}

In this section, we prove our results from \Cref{sec:RevDataProc}.
\Cref{app:RevDatProcProof} contains the proof of the 
reverse data processing inequality (\Cref{thm:quant-scheme}).
 We establish the tightness of the reverse data processing inequality for Hellinger distance (\cref{lem:HellTight})  in \Cref{app:tightRevDatProc}.

Fix the distributions $p$ and $q$ over $[k]$.
For $0 \leq l < u < \infty$, we first define the following sets\footnote{When $q(x) = 0$ for some $x$ and $p(x) \neq 0$, we think of $p(x)/q(x) = \infty$. Without loss of generality, we can assume that for each $x \in [k]$, at least one of $p(x)$ or $q(x)$ is non-zero.}:
\begin{align}
\label{eq:Adefinition}
A_{l,u} &= \left\{ i \in [k]: \frac{p_i}{q_i} \in [l,u) \right\} \,\, \text{and}\,\,\nonumber
\\ A_{l,\infty} &= \left\{ i \in [k]: \frac{p_i}{q_i} \in [l,\infty] \right\}.
\end{align}
We will use the notation from \cref{def:thresh_test}.

\subsection{Reverse data processing: Proof of \Cref{thm:quant-scheme}}
\label{app:RevDatProcProof}
\fDivRevDataProc*

\begin{proof}
Let $\kappa > 0$ be as in \cref{def:well-behaved}.
By definition of the $f$-divergence, we have the following:
\begin{align}
I_f(p,q) &= \sum_{i \in A_{1+\kappa,\infty}} q_i f\left( \frac{p_i}{q_i} \right) + \sum_{i \in A_{1, 1+ \kappa}} q_if\left( \frac{p_i}{q_i} \right) \nonumber \\
&\qquad + \sum_{i \in A_{1/(1+\kappa),1}}q_i f\left( \frac{p_i}{q_i} \right) + \sum_{i \in A_{0,1/(1+\kappa)}}q_if\left( \frac{p_i}{q_i} \right),
\label{eq:DecomHell}
\end{align}
where the sets $A_{l,u}$ are defined as in equation~\eqref{eq:Adefinition}. Note that the sets $A_{1+\kappa,\infty}$ and $A_{0,1/(1+\kappa)}$ contain  the elements that have a large ratio of probabilities under the two distributions. We now consider two cases.

\paragraph{Case $1$: Main contribution by large ratio alphabets} %
\label{par:case_}

We first consider the case when $\sum_{i \in A_{1+\kappa,\infty} \cup A_{0, 1/(1+\kappa)}} q_if\left( \frac{p_i}{q_i} \right) \geq \frac{I_f(p,q)}{2}$.
As we will show later, this is the simple case ($D=2$ already achieves the claim).
By symmetry of the $I_f$-divergence for the well-behaved $f$-divergence (\ref{item:symmetry} in \cref{def:well-behaved}), it suffices to consider the case when $\sum_{i \in A_{1+\kappa,\infty} } q_if\left( \frac{p_i}{q_i} \right)  \geq \frac{I_f(p,q)}{4}$.\footnote{
\label{footnote:asymmetry}
That is, we can apply the following argument to the distributions $\tilde{p}:=q$ and $\tilde{q} := p$ with $\tilde{A}$ defined as in equation~\eqref{eq:Adefinition} (with $\tilde{p}$ and $\tilde{q}$). There is a slight asymmetry because of the elements that have likelihood ratio exactly $1+ \kappa$ or $1/(1 + \kappa)$, but note that if $\sum_{i \in A_{0,1/(1+\kappa)} } q_if\left( \frac{p_i}{q_i} \right)  \geq \frac{I_f(p,q)}{4}$, we have $\sum_{i \in \tilde{A}_{1+\kappa, \infty} } \tilde{q}_if\left( \frac{\tilde{p}_i}{\tilde{q}_i} \right) \geq \sum_{i \in {A}_{0, 1/(1+\kappa)} } {q}_if\left( \frac{{p}_i}{{q}_i} \right) \geq \frac{I_f(p,q)}{4} $, because $A_{0,1/(1+\kappa)} \subseteq \tilde{A}_{1+\kappa,\infty}$, since the interval in $A_{0,l}$ is left-open.}

We will show that there exists $\bT \in \cTT_2$ such that $I_f(\bT p,\bT q) \geq \frac{f(1/(1+\kappa)}{4f(\nu)} I_f(p,q)$. 
Let $\bT$ be the channel corresponding to the threshold $1 + \kappa$, i.e., $\bT$ corresponds to the function $i\mapsto \I_{A_{1+\kappa,\infty}}(i)$. 
Note that $p' := \bT p$ and $q' := \bT q$ are distributions on $\{0,1\}$, with $(\bT p)_1 = \sum_{i \in A_{1+\kappa,\infty}} p_i$ and $(\bT q)_1 = \sum_{i \in A_{1+\kappa,\infty}} q_i$.
Furthermore, $p' \geq (1+\kappa) q'$.
Using convexity and non-negativity of $f$, and the fact that $f(1)= 0$  (see \ref{item:pos}), we have $f(x) \leq f(y)$ for $0 \leq y \leq x \leq 1$.
Using the non-negativity of $f$ (\ref{item:pos}), symmetry of $f$ (\ref{item:symmetry}), and monotonically decreasing property of $f$ on $[0,1]$, we obtain the following:
\begin{align}
 I_f(\bT p, \bT q) &= p' f\left( \frac{q'}{p'} \right) + (1 - p') f\left( \frac{1-q'}{1-p'} \right) \nonumber   \\
 &\nonumber \geq p' f\left( \frac{q'}{p'} \right) \\
    &\geq p' f\left( \frac{1  }{1+\kappa} \right).
\label{eq:FDiv-heavy1}
\end{align}
Moreover, by the assumption that $ \sum_{i \in A_{1+\kappa,\infty} } p_i f\left( \frac{q_i}{p_i} \right) \ge 0.25 I_f(p,q)$ (where we use the symmetry property of $f$ in \ref{item:symmetry}), we have
\begin{align}
0.25I_f(p,q) &\leq  \sum_{i \in A_{ 1+\kappa,\infty} } p_i f\left( \frac{q_i}{p_i} \right) \leq  \sum_{i \in A_{1+\kappa,\infty} } p_i f( \nu) 
= p' f(\nu),
\label{eq:FDiv-heavy2}
\end{align}
where we use the facts that $ q_i/p_i \in [\nu,1] $ and $f$ is decreasing on $[\nu,1]$.
Combining inequalities~\eqref{eq:FDiv-heavy1} and \eqref{eq:FDiv-heavy2}, we obtain
\begin{align}
I_f(\bT p, \bT q) \geq \frac{f(1/(1+\kappa))}{4f(\nu)} I_f(p,q),
\label{eq:Case1OptimalTGuarantee}
\end{align}
which implies that $\frac{I_f({p,q})}{I_f(\bT p, \bT q)} \leq \frac{4 f(\nu)}{f(1/(1+\kappa))}$,
proving the desired result. 

We now comment on the computational complexity of finding a $\bT^*$ that achieves the rate~\eqref{eq:Case1OptimalTGuarantee}. 
Since the channel $\bT^*$ only depends on $\kappa$, the algorithm only needs to check whether the threshold should be $1 + \kappa$ or $1/(1 + \kappa)$, which requires at most $\poly(k)$ operations.

\paragraph{Case $2$: Main contribution by small ratio alphabets} %
We now consider the case when $\sum_{i \in A_{1,1+\kappa} \cup A_{1/(1+\kappa), 1}} q_i f\left(\frac{p_i}{q_i}\right)\geq \frac{I_f(p,q)}{2}$. By symmetry (\ref{item:symmetry}), it suffices to consider the case when $\sum_{i \in A_{1,1+\kappa}} q_i f\left(\frac{p_i}{q_i}\right) \geq \frac{I_f(p,q)}{4}$.\footnote{There is a slight asymmetry here as well, but similar to the previous footnote, it suffices to consider this case.}
This requires us to handle the elements where $p_i$ and $q_i$ are close, and the following arguments form the main technical core of this section.

We first recall the reverse Markov inequality, whose role will become clear later in the proof:
\RevMarkovD*
For any $i \in A_{1,1+\kappa}$, both $q_i$ and $p_i$ are positive. Let $\delta_i = \frac{p_i}{q_i}-1$, which lies in $[0,\kappa)$, by definition. Then $p_i = q_i(1 + \delta_i)$.
Let $X$ be a random variable over $[0, \kappa)$ such that for $i \in A_{1,1+\kappa}$, we define $\P(X = \delta_i) = q_i$ and $\P(X = 0) = 1 - \sum_{i \in A_{1,1+\kappa}} q_i$. 
We define $\Omega$ to be the support of the random variable $X$.

We now apply \cref{lem:revMarkovD} to the random variable $ Y = X^\alpha$. Let $\beta= \kappa^\alpha$ and $R_2 = \min\{k,1 + \log(\kappa^\alpha/\E[X^\alpha])\}$. 
Let $0\leq\nu_1'\leq \cdots \le \nu'_D= \beta$ be thresholds achieving the bound \eqref{eq:revMarkovD}.
Let $\nu_j = (\nu'_j)^{1/\alpha}$ for all $j \in [D]$.
We thus have
\begin{align}
\label{eq:revMarkovReqGuarantee}
\sum_{j=1}^{D-1} \nu_j^\alpha \P(X \in [\nu_j, \nu_{j+1})) \geq \frac{1}{13} \E[X^\alpha] \min\left\{ 1, \frac{D}{R_2} \right\}.
\end{align}

We now define the thresholds $\Gamma=(\gamma_1,\dots,\gamma_{D-1})$ such that $\gamma_j = 1 + \nu_j$ for $i \in [D-1]$. We set $\gamma_0 = 0$ and $\gamma_ \infty = 0$. Recall that by definition, for $j\in[D-1]$, we have  $A_{\gamma_j,\gamma_{j+1}} = \{i : p_i/q_i \in [\gamma_j, \gamma_{j+1})\}$.
Since $1 \leq \gamma_1 \leq \gamma_{D-1} = 1 + \nu_{D-1} \leq 1 + \kappa$,  we have the following for $j \in [D-2]$:
\begin{align*}
A_{\gamma_j,\gamma_{j+1}} = 
\{i : p_i/q_i \in [\gamma_j, \gamma_{j+1})\} = \{i : \delta_i \in [\nu_j, \nu_{j+1})\}.
\end{align*}
Note that for any $j \in [D-2]$ and any function $g$, we have
\begin{align}
\label{eq:XrandVariable}
\sum_{i \in A_{\gamma_j, \gamma_{j+1}}}  g(\delta_i)q_i &= \sum_{i \in A_{\gamma_j, \gamma_{j+1}}} g(\delta_i)\P(X = \delta_i) \nonumber\\&= \sum_{x \in \Omega \cap [\nu_j, \nu_{j+1})} g(x)\P(X = x) = \E\left[ g(X) \I_{X \in [\nu_j, \nu_{j+1})}\right].
\end{align}
Using \ref{item:quad} and the fact that $0\leq \delta_i \leq \kappa$, we further have
\begin{align}
\label{eq:HellExpXsq}
\sum_{i \in A_{1,1+\kappa}} q_i f\left( \frac{p_i}{q_i} \right) &= \sum_{i \in A_{1,1+\kappa}} q_i f(1 + \delta_i) \leq \sum_{i \in A_{1,1+\kappa}} C_2q_i \delta_i^\alpha = C_2\E[X^\alpha],
\end{align}
where the last equality uses the same arguments as in inequality \eqref{eq:XrandVariable}.
Finally, we note that inequality \eqref{eq:HellExpXsq} and the assumption $I_f(p,q) \leq 4 \sum_{i \in A_{1,1+\kappa}} q_i f(p_i/q_i)$ implies that
\begin{align}
I_f(p,q) & \leq 4C_2 \E[X^\alpha], \qquad \text{and} \nonumber\\
R_2 & \leq \min\{k,1 + \log(4C_2\kappa^\alpha/I_f(p,q))\} = R.
\label{eq:FdivAndXsquare}
\end{align}

We use $p'$ and $q'$ to denote the probability measures $\bT p$ and $\bT q$, 
respectively, where $\bT$ corresponds to the thresholds $\Gamma$.
Thus, for $j \in [0:D-1]$, we have $p'(j) = \sum_{i \in A_{\gamma_j, \gamma_{j+1}}} p_i$; we have an analogous expression for $q'(j)$.
We now define the positive measure $p''$, as follows:
\begin{equation*}
p''_j = \begin{cases}
\sum_{i \in A_{\gamma_j, \gamma_{j+1}}}p_i, & \text{for } j \in [0: {D-2}], \\
\sum_{i \in A_{\gamma_{D-1}, 1 + \kappa}}p_i, & \text{for } j = D-1,
\end{cases}
\end{equation*}
and define $q''$ similarly. 
Recall that $\gamma_{D-1} = 1 + \nu_{D-1} \leq 1 + \kappa$.
Equivalently, we have $p''_j:= \sum_{i \in A_{\gamma_j, \min\left\{\gamma_{j+1}, 1 + \kappa\right\}}}p_i$ for each $j \in [0:D-1]$, since $\gamma_D = \infty$.
 Note that $p''$ and $q''$ might not be probability measures, as their sums might be smaller than 1, but they are equal to $p'$ and $q'$, respectively, on all elements except the last.
 { Moreover, we may define the ``$f$-divergence'' between $p''$ and $q''$ by mechanically applying the standard expression for $f$-divergence, but replacing the probability measures by $p''$ and $q''$, instead.} The $f$-divergence between $p''$ and $q''$, thus obtained, is smaller than the $f$-divergence between $p'$ and $q'$, since
\begin{align*}
q'_{D-1}  f\left( \frac{p'_{D-1}}{q'_{D-1}} \right) \geq  q''_{D-1}  f\left( \frac{p''_{D-1}}{q''_{D-1}} \right),
\end{align*}
which follows by noting that $q''_{D-1} \leq q'_{D-1}$, $p'_{D-1}/q'_{D-1} \geq p''_{D-1}/q''_{D-1} \geq 1$, and $f(x) \geq f(y) \geq 0$ for any $x \geq y \geq 1$.\new{\footnote{We briefly outline how $p'_{D-1}/q'_{D-1} \geq p''_{D-1}/q''_{D-1}$: Let $p'_{D-1} = p''_{D-1} + x$ and $q'_{D-1} = q''_{D-1} + y$ for $x,y \in \R$. By construction, we have that $x,y \geq 0$ and $x/y \geq 1 + \kappa \geq p''_{D-1}/q''_{D-1}$. Expanding $p'_{D-1}/q'_{D-1} - p''_{D-1}/q''_{D-1}$, we get the desired conclusion.}}
 We thus obtain the following relation:
\begin{align}
\label{eq:f-DivRestricted}
I_f(p',q')  \geq \sum_{j=0}^{D-1} q''_j f \left( \frac{p''_j}{q''_j}\right).
\end{align}

Fix $j \in [D-1]$. Using the facts that $ 0 \leq \frac{p''_j}{q''_j} - 1 \leq \kappa$ and $f(1 + x) \geq C_1 x^\alpha$ for $x \in [0,\kappa]$ (cf.\ \ref{item:quad}), we have the following for any $j$ such that $q''_j > 0$:
\begin{align}
\nonumber q''_j f \left( \frac{p''_j}{q''_j}\right) &= q''_j f \left(1 +  \frac{p''_j - q''_j}{q''_j}\right)
 \\
 \nonumber &\geq C_1 q_j'' \left( \frac{p''_j - q''_j}{q''_j}  \right)^\alpha\\
\nonumber   &= C_1 q_j'' \left(\frac{\sum_{i \in A_{\gamma_j, \min\left\{\gamma_{j+1}, 1 + \kappa\right\}}} q_i \delta_i}{\sum_{i \in A_{\gamma_j, \min\left\{\gamma_{j+1}, 1 + \kappa\right\}}} q_i} \right)^\alpha \\
\nonumber &\geq C_1 q_j''\nu_j^ \alpha && \left( \text{using $\delta_i \geq \nu_j$ for $i \in A_{\gamma_j, \gamma_{j+1}}$} \right)\\
&\geq C_1\nu_j^\alpha \P(X \in [\nu_j , \nu_{j+1})).  \label{eq:f-DivquadLowerBound}
\end{align}
We note that this inequality is also true if $q''_j = 0$, because $q''_j = \P\left( X \in [\nu_j , \nu_{j+1}) \right)$, and if the former is zero, then the expression in inequality \eqref{eq:f-DivquadLowerBound} is also zero, while $q''_j f \left(\frac{p''_j}{q''_j}\right)$ is nonnegative. 

Overall, we obtain the following series of inequalities:
\begin{align*}
I_f(p',q') &\geq \sum_{j=1}^{D-1} q''_j f \left( \frac{p''_j}{q''_j}\right) && \left( \text{using inequality \eqref{eq:f-DivRestricted} and $f \geq 0$} \right)\\
&\geq C_1 \sum_{j=1}^{D-1} \nu_j^\alpha \P(X \in [\nu_j , \nu_{j+1})) && \left( \text{using inequality \eqref{eq:f-DivquadLowerBound}} \right)\\
&\geq \frac{C_1}{13} \E[X^\alpha] \min\left\{1, \frac{D}{R_2} \right\} && \left( \text{using inequality \eqref{eq:revMarkovReqGuarantee}} \right)
\\&\geq \frac{C_1}{52 C_2} I_f(p,q) \min\left\{1, \frac{D}{R} \right\} && \left( \text{using inequality \eqref{eq:FdivAndXsquare}} \right).
\end{align*}
This shows that there exists a $\bT \in \cTT_D$ such that
\begin{align}
\frac{I_f(p,q)}{I_f(\bT p, \bT q)} \leq \frac{52C_2}{C_1} \max\left\{ 1 , \frac{R}{D} \right\}.
\label{eq:Case2OptimalTGuarantee}
\end{align}

 We now comment on the computational complexity of finding a $\bT^*$ that achieves the rate \eqref{eq:Case2OptimalTGuarantee}. Finding the thresholds $\Gamma$ is equivalent to finding  $(\nu'_1,\dots,\nu'_{D-1})$. As noted in \cref{lem:revMarkovD} and its proof, the guarantee of inequality \eqref{eq:revMarkovD} can be achieved by choosing $\nu'_j$ in one of the following ways: 
 \begin{itemize}
  \item Setting $\nu_j' = \min\{\kappa^\alpha, x 2^j\}$ for all $j$ and optimizing over $x$. As the random variable $Y$ has support of at most $k$, this algorithm runs in $\poly(k,D)$-time.

  \item  Choosing the top $D-1$ elements that maximize $\delta_iq_i$, and defining $\nu'_j$ appropriately. 
\qedhere
 \end{itemize} 
 \end{proof}

\subsection{Tightness of reverse data processing inequality: {Proof of \cref{lem:HellTight}}}
\label{app:tightRevDatProc}
\RevMarkIneqHell*

\begin{proof}

We will design $p$ and $q$ such that $p_i /q_i \in [0.5,1.5]$ for all $i$.
Fix any set of thresholds $\Gamma = \{\gamma_1, \dots, \gamma_{D-1}\}$ which, without loss of generality, lie in $[0.5,1.5]$.
Let $\bT$ be the corresponding channel.
Let $p'$ and $q'$ be the distributions after using the channel $\bT$.

Note that $k$ will depend on $\rho$, which will be decided later.
For now, let $k$ be even and equal to $2m$.
Let $\tilde{q}$ be an arbitrary distribution on $[m]$, to be decided later.
Using this $\tilde{q}$, we define a distribution $q$ on $[k]$, as follows:
\begin{align*}
q_i = \begin{cases} 0.5\tilde{q}_i, & \text{if } i \in [m] \\
0.5\tilde{q}_{i-m}, & \text{if } i \in [k]\setminus [m]. \end{cases}
\end{align*}
Let $\tilde{\delta} \in [0,0.5]^{m}$, also
to be decided later.
Using $\tilde{\delta}$, we define $\delta$, as follows:
\begin{align*}
\delta_i = \begin{cases} \tilde{\delta}_i, & \text{if } i \in [m] \\
-\tilde{\delta}_{i-m}, & \text{if } i \in [k]\setminus [m]. \end{cases}
\end{align*}
We now define $p$ as follows:
For $i \in [k]$, define $p_i = q_i(1 + \delta_i)$. Equivalently,
\begin{align*}
p_i = \begin{cases} 0.5\tilde{q}_i(1 + \tilde{\delta}_i), & \text{if } i \in [m] \\
0.5\tilde{q}_{i-m}(1 - \tilde{\delta}_{i-m}), & \text{if } i \in [k]\setminus [m]. \end{cases}
\end{align*}
Thus, $p$ is a valid distribution if $q$ is a valid distribution. 
 Let  $\tilde{X}$ be the random variable such that $\P\{\tilde{X} = \tilde{\delta}_i\} = \tilde{q}_i$. 
We will need the following results, whose proofs are given at the end of this section:

\begin{claim}
\label{cl:hellExp}
We have the following inequality:
 \begin{align*}
0.02 \E[ \tilde{X}^2] \le \hel^2(p,q) \le \E[ \tilde{X}^2].
 \end{align*}

\end{claim}
\begin{claim}
\label{cl:hellExpGamma}
Let $\bT\in \cTT_D$ be a channel corresponding to a threshold test. Then
\begin{align}
\label{eq:LbHellUpperBd}
\hel^2(\bT p,&\bT q) \leq \sup_{0< \nu'_1< \dots< \nu'_D = 1}  \sum_{j=1}^{D-1} \P\left\{ \tilde{X} \geq \nu'_j] \right\} \left( \E\left[ \tilde{X} | \tilde{X} \geq \nu'_j \right] \right)^2.
\end{align}
\end{claim}

We will now show that there exist $p$ and $q$ (i.e., $\tilde{q} \in \R^{m}$ and $\tilde{\delta} \in \R^{m}$) such that the desired conclusion holds.
Defining $\tilde{q}$ and $\tilde{\delta}$ is equivalent to showing the existence of a random variable $\tilde{X}$ satisfying the desired properties. 
This was precisely stated in \cref{claim:ReverseMarkovTight}, 
 showing that there exists a distribution $\tilde{X}$ such that the following hold: (i) $\E[ \tilde{X}^2] = \Theta(\rho)$; (ii) the expression on the right-hand side of inequality \eqref{eq:LbHellUpperBd}, for any choice of thresholds $\Gamma$, is upper-bounded by a constant multiple of  $\frac{\E [\tilde{X}^2] D }{R'}$; and (iii) $R' = \max\{m, k'\} = \Theta(\log(1/\rho))$.

Using \cref{cl:hellExp,cl:hellExpGamma,claim:ReverseMarkovTight}, we obtain the following for any threshold channel $\bT$:
\begin{align*}
\hel^2( \bT p, \bT q) &\lesssim \E[\tilde{X}^2] \frac{D}{R'} \lesssim \hel^2(p,q) \frac{D}{R'},
\end{align*}
completing the proof.
\end{proof}

The omitted proofs of \cref{cl:hellExp} and \cref{cl:hellExpGamma} are given below.
\begin{proof}[Proof of \cref{cl:hellExp}]
We have the following:
\begin{align*}
\hel^2(p,q) &= \sum_{i \in [m]} \left( \sqrt{q_i \left( 1 + \delta_i \right)} - \sqrt{q_i} \right)^2 + \sum_{i \in [k]\setminus[m]} \left( \sqrt{q_i} - \sqrt{q_i \left( 1 + \delta_i \right)}  \right)^2
 \\
&= 0.5\sum_{i \in [m]} \left( \sqrt{\tilde{q}_i \left( 1 + \tilde{\delta}_i \right)} - \sqrt{ \tilde{q}_i} \right)^2 + 0.5\sum_{i \in [m]} \left( \sqrt{ \tilde{q}_i} - \sqrt{\tilde{q}_i \left( 1 - \tilde{\delta}_i \right)}   \right)^2
 \\
&= 0.5\sum_{i \in [m]} \tilde{q}_i \left( \left( \sqrt{ 1 + \tilde{\delta}_i} - 1 \right)^2 + \left( 1 - \sqrt{  1 - \tilde{\delta}_i }   \right)^2\right).
 \end{align*}
Using the fact that for $x \in [0,1]$, we have
\begin{align*}
\sqrt{1 + x} -1 & \geq 0.1x, \\
1 - \sqrt{1 - x} & \geq 0.1 x, \\
\sqrt{1 + x} & \leq 1 + x, \\
1 - x & \leq \sqrt{1 - x},
\end{align*}
we obtain
\begin{align*}
\E[ \tilde{X}^2]\geq  \hel^2(p,q) \geq  0.02 \E[ \tilde{X}^2].
 \end{align*}
\end{proof}

\begin{proof}[Proof of \cref{cl:hellExpGamma}]
Suppose $\bT$ corresponds to a threshold test with thresholds $\Gamma = \{\gamma_1, \dots , \gamma_{D-1} \}$ such that $\gamma_j < \gamma_{j+1}$. We define $ \gamma_0 = \min_i p_i/q_i$ and $\gamma_D = \max_i p_i/q_i$.
It suffices to consider the case when all $\gamma_j \in [0.5,1.5]$ for $j \in [0:D]$.
Let $p' = \bT p$ and $q' = \bT q$.
Let $j^* \in [D-1]$ be such that $\gamma_{j^*-1} < 1 $ and $\gamma_{j^*} \geq 1$. 

We now define the $\nu_j$'s as follows, for $j \in [0:D-1]$: 
\begin{align*}
\nu_j = \begin{cases} 
    \gamma_j - 1, & \text{ if } j \geq j^*\\
    1 - \gamma_j, & \text{ otherwise.} \end{cases}
\end{align*}
Thus, $\nu_j \in [0,1)$.

\noindent For $j \in [0:D-1]$, define
$$A_j := \left\{i \in [k]: (p_i/q_i) \in [\gamma_{j}, \gamma_{j+1})\right\} = \left\{i: 1 + \delta_i \in [\gamma_{j}, \gamma_{j+1})\right\}.$$
For $j \geq j^*$, we have $A_j = \left\{ i : \delta_i \in [\nu_{j}, \nu_{j+1}) \right\}$.
For $j < j^*$, we have
$$A_j = \left\{ i \in [k] : -\delta_i \in (\nu_{j+1}, \nu_{j}] \right\} = \left\{ i \in [k] : \tilde{\delta}_{i - m} \in (\nu_{j+1}, \nu_{j}] \right\}.$$
For $j \in [0:D-1]$, we have  $p'_j = \sum_{i \in A_j} p_i$ and $q'_j = \sum_{i \in A_j} q_i$.

We have the following decomposition of the squared Hellinger distance between $p'$ and $q'$:
\begin{align}
\label{eq:SimLBTaylor1}
\hel^2(p',q') = \sum_{j < j^*} \left(  \sqrt{p'_j} - \sqrt{q'_j}\right)^2 + \sum_{j \geq j^*}\left(  \sqrt{p'_j} - \sqrt{q'_j}\right)^2
\end{align}

 We analyze these two terms separately:
\paragraph{Case $1$: $j \geq j^*$}
Let $j$ be such that $q_j' > 0$. We have $p'_j \in [q'_j, 1.5 q'_j]$. Using the fact that $\sqrt{1 + x} - 1 \leq  x$ for $x \in [0, 0.5]$, we have
\begin{align}
\label{eq:SimLBTaylor2}\left(  \sqrt{p'_j} - \sqrt{q'_j}\right)^2 
&= q'_j\left( \sqrt{1 + \frac{p'_\gamma - q'_j}{q'_j}}  - 1 \right)^2  \leq \frac{(p'_i - q'_j)^2}{q'_j}.
\end{align}
Since $\gamma_{j}\geq 1$, note that 
\begin{align*}
q'_j &= \sum_{i \in A_j} q_i = \sum_{i \in [m]: \tilde{\delta}_i \in [\nu_j, \nu_{j+1})} q_i = \sum_{i \in [m]: \tilde{\delta}_i \in [\nu_j, \nu_{j+1})} 0.5 \tilde{q}_i= 0.5 \P\{ \tilde{X} \in [\nu_j, \nu_{j+1}  )\}.
\end{align*}
Similarly, we have
\begin{align*}
p'_j - q'_j &= \sum_{i \in A_j} \delta_i q_i = \sum_{i \in [m]: \tilde{\delta}_i \in [\nu_j, \nu_{j+1})} \delta_iq_i = 0.5 \sum_{i \in [m]: \tilde{\delta}_i \in [\nu_j, \nu_{j+1})} \tilde{\delta}_i \tilde{q}_i  = 0.5 \E\left[ \tilde{X} \I_{\tilde{X} \in [\nu_j, \nu_{j+1})} \right].
\end{align*}

Combining the last two displayed equations with inequality \eqref{eq:SimLBTaylor2} and using the definition of conditional expectation, we then obtain
\begin{align}
\sum_{j \geq j^*} \left(  \sqrt{p'_j} - \sqrt{q'_j}\right)^2 &\leq 0.5  \sum_{j \geq j^*} \P\left\{ \tilde{X} \in [\nu_j, \nu_{j+1}) \right\} \left( \E\left[ \tilde{X} | \tilde{X} \in [\nu_j, \nu_{j+1}) \right] \right)^2\nonumber\\
&\leq 0.5  \sum_{j \geq j^*} \P\left\{ \tilde{X} \geq \nu_j \right\} \left( \E\left[ \tilde{X} | \tilde{X} \geq  \nu_j) \right] \right)^2.
\label{eq:SimLBTaylor3}
\end{align}

\textbf{Case $2: j < j^*$:} Let $j < j^*$ be such that $q_j' > 0$. We have $p'_j \in [q'_j/2, q'_j)$. Using the fact that $1 - \sqrt{1 - x} \leq x$ for $x \in [0,1]$, we have 
\begin{align}
\label{eq:SimLBTaylor4}
\left(  \sqrt{q'_j} - \sqrt{p'_j}\right)^2 = q'_j\left(  1 - \sqrt{1 - \frac{q'_j - p'_i}{q'_j}}\right)^2 \leq \frac{(q'_j - p'_i)^2}{q'_j}.
\end{align}
Since $\gamma_{j} < 1$, we have
\begin{align*}
q'_j &= \sum_{i \in A_j} q_i = \sum_{i \in [k]\setminus [m]: \tilde{\delta}_{i-m} \in (\nu_{j+1}, \nu_{j}]} q_i = \sum_{i \in [k]\setminus [m]: \tilde{\delta}_{i-m} \in (\nu_{j+1}, \nu_{j}]} 0.5 \tilde{q}_{i-m} = 0.5 \P\{ \tilde{X} \in (\nu_{j+1}, \nu_{j}  ]\}.
\end{align*}
Similarly, we have
\begin{align*}
q'_j - p'_i &= \sum_{i \in A_j} (-\delta_i q_i) = \sum_{i \in [k]\setminus [m]: \tilde{\delta}_{i-m} \in (\nu_{j+1}, \nu_{j}]} \tilde{\delta}_i(0.5 \tilde{q}_{i-m}) = 0.5 \E\left[ \tilde{X} \I_{\tilde{X} \in (\nu_{j+1}, \nu_{j}]} \right].
\end{align*}
Combining the last two displayed equations with inequality \eqref{eq:SimLBTaylor4} and using the definition of conditional expectation, we then obtain
\begin{align}
\nonumber
\sum_{j < j^*} \left(  \sqrt{p'_j} - \sqrt{q'_j}\right)^2 &\leq  0.5\sum_{j < j^*} \P\left\{ \tilde{X} \in (\nu_{j+1}, \nu_{j}] \right\} \left( \E\left[ \tilde{X} | \tilde{X} \in (\nu_{j+1}, \nu_{j}] \right] \right)^2 \\
&\leq  0.5\sum_{j < j^*} \P\left\{ \tilde{X}  > \nu_j] \right\} \left( \E\left[ \tilde{X} | \tilde{X} > \nu_j] \right] \right)^2.
\label{eq:SimLBTaylor5}
\end{align}
Combining inequalities \eqref{eq:SimLBTaylor3} and \eqref{eq:SimLBTaylor5}, we can complete the proof by noting that $\tilde{X}$ is a discrete random variable, so the distinction between $\tilde{X} \geq \nu_j$ (cf.\ inequality \eqref{eq:SimLBTaylor3}) and $\tilde{X} > \nu_j$ (cf.\ inequality \eqref{eq:SimLBTaylor5}) does not matter when taking the supremum.
\end{proof}

\section{Simple robust binary hypothesis testing} %
\label{app:hypothesis_testing}

In this section, we prove results concerning (robust) binary hypothesis testing that were omitted from \Cref{sec:BinHypTest}.
We begin by providing another explicit example that shows that (i) the robust sample complexity has phase transitions with respect to the amount of corruption, and (ii) the sample complexity of Scheffe's test can be strictly suboptimal.
\begin{example}[Sample complexity of $\cB_\robust(\{p,q\},\cdot)$]
\label{exm:robust-2} Consider $0< \alpha < \beta < \delta < 1$, satisfying $ \delta < 2 \beta - \alpha$. Let $p$ and $q$ be the following two distributions:
\begin{align*}
p & := \left(1/2 - 2\epsilon - \epsilon^{1 + \alpha} + \epsilon^{1 + \beta} - \epsilon^{1 + \delta}, 1/2 + 2\epsilon, \epsilon^{1 + \alpha} - \epsilon^{1 + \beta}, \epsilon^{1 + \delta}\right), \\
q & := \left(1/2,1/2- \epsilon^{1 + \alpha}, \epsilon^{1 + \alpha}, 0\right),
\end{align*}
where $\epsilon \leq 0.01$.
Note that $\dtv(p,q) = \Theta(\epsilon)$.
We have $\nstar(\{p,q\}) = 	\Theta\left(1/\epsilon^{1 + \delta}\right)$. For any fixed $\gamma > 0$, the sample complexity $\nstar_\robust(\{p,q\}, \epsilon^{1 + \gamma})$ satisfies the following growth condition for $\epsilon$ small enough, where we omit constant factors for brevity: 
    \begin{align*}
    \nstar_\robust(\{p,q\}, \epsilon^{1 + \gamma}) = \begin{cases} \frac{1}{\epsilon^{1 + \delta}}, & \text{if }\gamma  > \delta\\
    \frac{1}{\epsilon^{1 + 2 \beta - \alpha}}, & \text{if }\gamma \in (\beta, \delta)\\        
    \frac{1}{\epsilon^{2}}, & \text{if }\gamma  \in (0 ,\beta).
    \end{cases}
    \end{align*}
    By \Cref{thm:ub-Simple-D-robust}, the sample complexity under communication constraints of $D = 2$ messages satisfies the same result up to constants (note that the optimal channel may change with respect to $\gamma$). 
  On the other hand, for all $\gamma > 0$, the sample complexity of Scheffe's test for $\cB_\robust(\{p,q\},\epsilon^{1 + \gamma})$ is $\Theta(1/\epsilon^2)$.  
  \end{example}
  \begin{proof}
Note that
$$\hel^2(p,q) = \Theta\left(\epsilon^2 + \epsilon^{2 + 2\beta - 1 - \alpha} + \epsilon^{1 + \delta}\right) = \Theta\left(\epsilon^{1 + 2\beta - \alpha} + \epsilon^{1 + \delta}\right) = \Theta\left(\epsilon^{1 + \delta}\right),$$
since $\delta < 2 \beta - \alpha$, which leads to the claim on $\nstar(\{p,q\})$ by \Cref{fact:testing}.

The lower bounds on $\nstar_\robust(\{p,q\},\epsilon^{1 + \gamma})$ follow by applying \Cref{fact:testing} on the following choices of $\tilde{p}$  and $\tilde{q}$, which lie within $\epsilon^{1 + \gamma}$ in total variation distance: 
\begin{enumerate}
  \item $\gamma > \delta$: This follows directly by choosing $\tilde{p} = p$ and $\tilde{q} = q$.
  \item $\gamma \in (\beta,\delta)$: This follows by choosing $\tilde{p} = p$ and $\tilde{q} = (1/2 -\epsilon^{1+ \delta}, 1/2 - \epsilon^{1 + \alpha}, \epsilon^{1 + \alpha}, \epsilon^{1+ \delta})$.
  \item $\gamma \in (0,\beta)$: This follows by choosing $\tilde{p} = p$ and
  $$\tilde{q} = (1/2 -\epsilon^{1+ \delta} - \epsilon^{1 + \beta}, 1/2 - \epsilon^{1 + \alpha}, \epsilon^{1 + \alpha} - \epsilon^{1 + \beta}, \epsilon^{1+ \delta}).$$
\end{enumerate}
We now discuss the channels that achieve the upper bound. 
We will choose $\bT$ corresponding to $\I_A(\cdot)$, as follows:
\begin{enumerate}
   \item $\gamma > \delta$: Take $A = \{4\}$. Then $\E_{\tilde{p}} (A) \geq 2\epsilon^{1 + \delta}/3$ and $\E_{\tilde{q}} (A) \leq \epsilon^{1 + \delta}/3$. As mentioned later, this can be tested with $O(1/ \epsilon^{1 + \delta})$ samples.
   \item $\gamma \in (\beta,\delta)$: Take $A = \{3\}$. Then $\E_{\tilde{p}} (A) \leq \epsilon^{1 + \alpha} -  2\epsilon^{1 + \beta}/3$ and $\E_{\tilde{q}} (A) \geq \epsilon^{1 + \alpha} - \epsilon^{1 + \beta}/3$. As mentioned later, this can be tested with $O(1/ \epsilon^{1 + 2 \beta - \alpha})$ samples.
   \item $\gamma \in (0, \beta)$: This follows by taking $A = \{2\}$ and using similar arguments as above.
 \end{enumerate} 

Finally, we prove that the sample complexity of Scheffe's test is $\Theta(1/\epsilon^2)$. Scheffe's test transforms $p$ and $q$ to  Bernoulli distributions, with probabilities of observing $1$ equal to $\left( 1/2 + 2 \epsilon + \epsilon^{1+\delta} \right)$ and $\left( 1/2 - \epsilon^{1 + \alpha} \right)$, respectively.
It is easy to see that the Hellinger distance between these two Bernoulli distributions is $\Theta(\epsilon^2)$, implying that the sample complexity of Scheffe's test is $\Theta(1/\epsilon^2)$. Its robustness to $\epsilon^{1 + \gamma}$-corruption follows from similar arguments as above.

For completeness, we outline the typical concentration argument that is needed to perform the tests above. Let $X$ be a mean of i.i.d.\ indicator random variables, i.e., $X = (\sum_i Y_i)/n$,  where $Y_i \in \{0,1\}$ and $\E[Y_i] = \mu \leq 1/2$. Then $\E[X] = \mu$ and $\Var(X) \leq \mu/n$. Chebyshev's inequality implies that with probability $0.01$, we have $X \in [\mu - 10\sqrt{\mu/n}, \mu + 10 \sqrt{\mu/n}]$.
Thus, if $n \geq 10^4/ \mu $, then with probability $0.01$, we have $X \in [2\mu/3, 4 \mu/ 3]$. By similar logic, if $n \geq 10^4 / \delta^2$, then with probability $0.01$, we have $X \in [\mu - \delta , \mu + \delta]$.
\end{proof}

Finally, we provide additional details regarding \Cref{exm:robust-1}. \\
\textbf{Details regarding \Cref{exm:robust-1}:} %
Consider the set $A = \{2\}$. We have $p(A) = 0.5+3 \epsilon $ and $q(A) = 0.5$. Any valid $\tilde{p}$ and $\tilde{q}$ lying within $\epsilon$ in total variation distance of $p$ and $q$, respectively, satisfy $\tilde{p}(A) \geq 0.5 + 2 \epsilon$ and $\tilde{q}(A) \leq 0.5 + \epsilon$.
Thus, estimating the mean of $\I_A(X)$ up to error $\epsilon/2$ gives a valid test, which takes $O(1/\epsilon^2)$ samples, by the arguments outlined above.
The lower bound follows by applying \Cref{fact:testing} to $\cB(\tilde{p}, \tilde{q})$, where $\tilde{p} = (0.5 - 3 \epsilon, 0.5 + 3 \epsilon, 0)$ and $\tilde{q} = (0.5,0.5,0)$.
It can be seen that for this choice of $\tilde{p}$, we have $\bT^* \tilde{p} = (1, 0)$ and $\bT^* q = (1,0)$.

\section{Lower bounds for $M$-ary hypothesis testing} %
\label{app:lower_bounds_for_m-ary}

In this section, we provide the proof of \cref{thm:Adaptive-lb}.
We prove the two bounds in \cref{thm:Adaptive-lb} separately: the $\Omega(M)$ lower bound from the strong data processing inequality is proved in \Cref{app:strong_data_processing} (cf.\ \Cref{cor:MaryAdapLB-sdpi}), and the $\Omega(M^{1/3})$ lower bound from the SQ lower bound is proved in \Cref{app:sq-lower-bound} (cf.\ \cref{cor:SQMaryHypTesting}). 
Finally, we prove a $\Omega(\sqrt{M})$ lower bound for non-adaptive, non-identical channels from the impossibility of $\ell_1$-embedding in \Cref{app:lower_bounds_from_impossibility_of_l1} (cf.\ \Cref{thm:lowerbound-l1-impossibility}).

In this section, we abuse notation by using $p_1,p_2$, etc., and $P_1,P_2$, etc., to denote different probability distributions.

\subsection{Strong data processing}
\label{app:strong_data_processing}
\paragraph{Preliminaries} %
\label{par:preliminaries}

We will closely follow the terminology of Braverman, Garg, Ma, Nguyen, and Woodruff~\cite{BravGMNW16},
 to which we refer the reader for more details.
Let $\cQ = \{Q_0, Q_1\}$ be two distributions on $\cX$.
For any $i \in [M]$, we define $P_i$ to be the distribution over $(Z_1,\dots,Z_M)$, where the $Z_j$'s are independent, and $Z_j \sim Q_0$ for $j \neq i$ and $Z_i \sim Q_1$.
We use $P_0$ to denote the distribution $Q_0^{\otimes n}$.
We define $\cP_M = \{P_0,P_1,\dots,P_M\}$.
Our goal is to perform statistical estimation using $n$ machines. 
The model generation process is as follows: $V$ is sampled uniformly from $\{0,\dots,M\}$. Conditioned on $V = v$,  for each $j \in [n]$, machine $j$ receives an i.i.d.\ sample $X_j$ from the distribution $P_v$.
When it is clear from context, we will use $X$ as shorthand for $(X_1,\dots,X_n)$.

We will work in the blackboard protocol. {Here, all machines simultaneously write the first iteration of their messages on a ``blackboard,'' and the subsequent iterations of messages are on subsequent blackboards and may depend on the contents of all past blackboards. The combined content (in bits) of all blackboards is called the transcript of the protocol, which is denoted by $\Pi$. The blackboard protocol is also called the ``fully adaptive'' or simply ``adaptive'' protocol. Since it imposes the fewest constraints on permitted actions, lower bounds proved for this protocol are valid for other protocols, as well. A special case of interest is the ``sequentially adaptive'' protocol, where machines communicate in a fixed order, with subsequent messages allowed to depend on  past messages. In the communication-constrained setting considered in this paper, we restrict the size of the transcript $|\Pi|$ to be at most $n\log D$, as each machine is permitted to send at most $D$ messages ($\log D$ bits).} 

The estimator $\widehat{v}$ maps each transcript $\Pi$ to an element of $\cP_M$.
The failure probability is then defined as $$  
 R(\Pi, \hat{v}, \cP_M) := \max_{v \in \{0,1,\dots,M\}} \Pr[\hat{v}(\Pi) \neq v \vert V =  v].$$
We use $T(n,\cP_M)$ to denote the task of hypothesis testing among the distributions in $\cP_M$ with $n$ machines, and we say that $(\Pi, \widehat{v})$ solves $T(n,\cP_M)$ if the protocol works on $n$ machines and $ R(\Pi, \hat{v}, \cP_M) \leq 0.1 $. 
We will use the definitions of $\ic(\Pi)$ (the information cost of $\Pi$) and $\mic(\Pi)$ (the minimum information cost of $\Pi$) from Braverman, Garg, Ma, Nguyen, and Woodruff \cite{BravGMNW16}.

\begin{lemma}[Direct-sum for multiple hypothesis testing~\cite{BravGMNW16}] \label{thm:direct_sum_sparse}
Let $M \ge 1$, and let $\cQ$ and $\cP_M$ be defined as above.
If there exists a protocol estimator pair $(\Pi, \hat{v})$ that solves the detection task $T(m, \cP_M)$ with  information cost $I$, then there exists a protocol estimator pair $(\Pi',\hat{v}')$ that solves the detection task $T(m,\cQ)$ with minimum information cost $I'$ satisfying $I' \lesssim \frac{I}{M}$. 
\end{lemma}

For the set of two distributions, we use the following hardness result:
\begin{lemma}
\label{lem:sdpiBinaryCase}There exists a constant $c \geq 1$ such that for every $\beta \in (0,1)$, there exist two distributions $\cQ = \{Q_0,Q_1\}$ such that any $(\Pi, \widehat{v})$ that solves $T(m,\cQ)$ with failure probability at most $1/4$ for any $m$
 satisfies $\mic(\Pi) \geq \frac{c}{\beta}$. Moreover, $\dtv(Q_0, Q_1) = \Theta(\sqrt{\beta})$.
\end{lemma}
\begin{proof}
For two distributions $p$ and $q$, we use $\beta(p,q)$ to denote the strong data processing inequality (SDPI) constant, as defined in Braverman, Garg, Ma, Nguyen, and Woodruff \cite{BravGMNW16}. Suppose $\Theta$ is uniform over $\{0, 1\}$, and $X|\Theta = 0 \sim p$ and $X|\Theta = 1 \sim q$. Then the SDPI constant $\beta(p, q)$ is defined as 
\begin{align*}
\sup_{P_{Y|X}} \left\{\frac{I(\Theta; Y)}{I(\Theta; X)} : \Theta \to X \to Y \right\}.
\end{align*}
This constant is also called the post strong data processing inequality (post-SDPI) constant~\cite{PolyanskiyWu23}.

We will use the following result, which shows that if $p$ and $q$ have bounded likelihood ratios, the SDPI constant is small:
\begin{lemma}[\cite{DucJWZ14}]
\label{lem:SDPI4BddedRatio}
Let $\beta \in (0,1)$. If two Bernoulli distributions with parameters $p$ and $q$ satisfy
 \begin{align*}
 & e^{-\sqrt \beta} p \leq q \leq e^{\sqrt \beta} p, \\
 & e^{-\sqrt \beta} (1-p) \leq 1 - q \leq e^{\sqrt \beta} (1-p),
 \end{align*}
then $\beta(p,q) \le (2 e^2)\beta$. %
\end{lemma}
Let $Q_0$ and $Q_1$  be binary distributions with probabilities of observing $1$ equal to $q_0$ and $q_1$, respectively.
Set $q_0 = 1/2$ and %
$q_1 = e^{-\sqrt \beta}/2$.
Then
\begin{align*}
\frac{q_0}{q_1} = e^{\sqrt \beta}\,\,\, \text{ and } \,\,\,\, \frac{1 - q_0}{1 - q_1} \le \frac{1}{2- e^{-\sqrt{\beta}}},
\end{align*}
and both ratios lie between %
$e^{-\sqrt \beta}$ and $e^{\sqrt \beta}$.
Thus, we have $\beta(\mu_0, \mu_1) \lesssim \beta $, by \cref{lem:SDPI4BddedRatio}.

Fix a constant $c' \leq 0.1$. Let the protocol be $\Pi$. Since the protocol is successful, using \cref{fact:testing,fact:div}, we have  $\hel^2(\Pi|_{V=0}, \Pi|_{V=1}) \geq c'$, for a constant $c'$.
Applying Braverman, Garg, Ma, Nguyen, and Woodruff \cite[Theorem 1.1]{BravGMNW16}, we obtain
\begin{align*}
\hel^2(\Pi|_{V=0}, \Pi|_{V=1}) \leq c \beta \cdot \mic(\Pi),  
\end{align*}
which yields the desired conclusion. Finally, the bound on total variation distance follows from a direct calculation and the fact that $\beta \in (0,1)$.
\end{proof}

Combining \cref{lem:sdpiBinaryCase,thm:direct_sum_sparse}, we obtain the following result:

\begin{corollary}
\label{cor:MaryAdapLB-sdpi}
 For every $\epsilon \in (0,1)$, there exist $M+1$ distributions $\{P_0,\dots,P_M\}$ such that (i) $\dtv(p,q) \geq \epsilon$ for all $p\neq q$ in $\cP_M$, and (ii) $\nada(\cP_M,\cT_D) \gtrsim \frac{M }{\epsilon^2 \log D}$.
\end{corollary}
\begin{proof} 
Let $\cQ = \{Q_0,Q_1\}$ be the two distributions from \cref{lem:sdpiBinaryCase}  such that $\dtv(Q_0,Q_1) \geq \epsilon$ and every successful protocol $\Pi$ for $T(n, \cQ)$ satisfies $\mic(\Pi) \geq \frac{c}{\epsilon^2}$.
Construct $\cP_M$ as defined above using $\cQ$.
It can be seen that $\dtv(p,q) \geq \epsilon$ for any distinct $p$ and $q$ in $\cP_M$.
Suppose there exists a successful protocol $\widehat{\Pi}$ for $T(\nstar,\cP_M)$ with each machine sending at most $\log D$ bits.
Then we have   
\begin{align*}
 I = \sup_{v \in [M+1]} I_v\left(\widehat{\Pi};X | R_{\text{pub}}\right) \leq \sup_{v \in [M+1]} h (\widehat{\Pi})
 \leq \nstar \log D,  
 \end{align*}
where $h(\widehat{\Pi})$ denotes the entropy of the transcript $\widehat{\Pi}$.
Thus, \cref{thm:direct_sum_sparse} implies that there exists a successful protocol $\widehat{\Pi}'$ for $T\left(\nstar,\cQ\right)$ such that $\mic(\widehat{\Pi}') \lesssim \frac{\nstar \log D}{M}$.
However, we have $\mic(\widehat{\Pi}') \gtrsim 1 /\epsilon^2$. Thus, we obtain $\frac{1}{\epsilon^2} \lesssim \frac{\nstar \log D}{M}$, or equivalently, $\nstar \gtrsim \frac{M }{\epsilon^2 \log D}$.
This completes the proof of \Cref{cor:MaryAdapLB-sdpi} and the proof of the first claim in \Cref{thm:Adaptive-lb}.
\end{proof}

\subsection{SQ lower bounds}
\label{app:sq-lower-bound}

\noindent \textbf{Preliminaries:} %
We will use the standard notations from the statistical query (SQ) complexity literature \cite{FelGRVX17,Feldman17}. 
In particular, we will use the following oracles: $\STAT(\tau)$, $\VSTAT(t)$, and $\MSAMPLE(D)$.

For two square-integrable functions $f, g : \cX \to \R$ and a distribution $P$ on $\cX$, we define $\left\langle f,g \right\rangle_P := \E_P[f(X)g(X)]$.
For a distribution $P$, we will abuse notation by using $P$ to refer to both the distribution and its pmf.
 For two distributions $P_1$ and $P_2$, their pairwise correlation with respect to the base measure $P$ is defined as
$$\chi_P(P_1, P_2) := \left|\left\langle \frac{P_1}{P} - 1, \frac{P_2}{P} - 1 \right\rangle_P\right| = \left|\left\langle \frac{P_1}{P} , \frac{P_2}{P}  \right\rangle_P - 1 \right|.$$
The {\em average correlation} of a set of distributions $\cP'$ relative to a distribution $P$ is denoted by $\rho(\cP',P)$ and defined as
$\rho(\cP',P) := \frac{1}{|\cP'|^2} \sum_{P_1,P_2 \in \cP'} \chi_P(P_1, P_2)$.
\begin{definition}[Decision problem]
 Let $P$ be a fixed distribution and $\cP$ a set of distributions which does not contain $P$. 
 Given access to the input distribution $Q$, which either equals $P$ or belongs to $\cP$, the goal is to identify whether $Q = P$ or $Q \in \cP$.
We refer to this problem as $\cB_{S}(\cP,P)$.
\end{definition}
We will use $\SDA(\cB_S(\cP,P), \bar{\gamma})$ to denote the average statistical dimension with average $\bar \gamma$ of the decision problem $\cB_S(\cP,P)$~\cite[Definition 3.6]{FelGRVX17}.

Although our main focus will be the $\STAT$ oracle and blackboard protocol, we also mention hardness results for $\VSTAT$ and $\MSAMPLE$ oracles, which follow from $\SDA$.  
\begin{theorem}[{\cite[Theorem 3.7]{FelGRVX17}, \cite[Theorem 7.3]{FelPV18}
}]
\label{thm:SqSTATVSTAT}
  Let $P$ be a distribution and $\cP$ be a set of distributions over a domain $X$, such that $\SDA(\cB_S(\cP,P),\bar{\gamma})=d$ for some $\bar{\gamma}$. Any (randomized) SQ algorithm that solves $\cB_S(\cP,P)$ with success probability $9/10$ must satisfy at least one of the following conditions:
\begin{itemize}
\item[(i)] performs $0.8d$ queries,
\item[(ii)] requires a single query to $\VSTAT(1/3 \bar{\gamma})$, or
\item[(iii)] requires a single query to $\STAT( \sqrt{3\bar{\gamma}})$.
\end{itemize}
In particular, for any $L$, any (randomized) SQ algorithm that solves $\cB_S(\cP,P)$ with success probability $9/10$  requires at least $m$ calls to $\MSAMPLE(L)$, where $m = \Omega\left( \frac{1}{L} \min\left\{d, \frac{1}{ \bar{\gamma}} \right\}\right)$.
\end{theorem}

Steinhardt, Valiant, and Wager~\cite{SteVW16}
show that an SQ lower bound also implies a lower bound for blackboard protocols:
\begin{theorem}[Lower bounds for blackboard communication using SQ algorithms {\cite[Proposition 3]{SteVW16}, \cite[Section B.1]{Feldman17}}]
\label{lem:sqAdaptiveLow}
 Let $\cB_S(\cP,P)$ be a decision problem that can be solved with probability $0.95$ by a communication-efficient algorithm that extracts at most $b$ bits from each of $m$ machines. Then $\cB_S(\cP,P)$ can be solved by an SQ algorithm, with probability at least $0.9$, which uses at most $2bm$ queries of $\STAT$ with tolerance $\tau = O\left(\frac{1}{2^b m}\right)$.
In particular, for some $\bar{\gamma}$, let $d = \SDA( \cB_S(\cP,P), \bar{\gamma})$. Then either $ \frac{1}{2^b m}  \lesssim \sqrt{\bar{\gamma}} $ or $ d \leq 2bm$.
\end{theorem}

We now describe a distribution family that has a large statistical dimension, on average.
\begin{lemma}[A decision problem with large SQ dimension]
\label{lem:SQHardDist}
Let $r \in \N$ and fix an $\epsilon \in (0, 1)$. For any $M = 2^r$, there exist distributions $\cP_M:= \left\{P_1,\dots,P_{M}\right\} \subseteq \Delta_{M+1}$ and $P \in \Delta_{M}$ such that
$\chi_P(P_i,P_j) = \I_{i=j}$ for all $(i,j)$. In particular, $\SDA\left( \cB_S(\cP_M, P), \bar{\gamma} \right) \geq \frac{M \bar{\gamma}}{\epsilon^2}$ for any ${\bar \gamma} \leq \epsilon^2$.
Moreover, for any two distinct $p,q \in \cP_M \cup \{P\}$, we have $\dtv(p,q) \geq 0.01 \epsilon$.
\end{lemma}
\begin{proof}
Let $k= M+1$.
Let $V = [v_1,\dots,v_k] \in \R^{k \times k}$ be the Walsh-Hadamard matrix. We have $V = V^\top $ and $\left\langle v_i,v_j \right\rangle = k\I_{i = j}$. 
Furthermore, we have $v_i \in \{-1,1\}^k$, where $v_1$ has all entries $1$, and for $i > 1$, $v_i$ has half positive entries and half negative entries.
Define $u$ to be the uniform distribution in $\Delta_{k}$, and define $e_i$ to be the distribution that places all its mass on the $i^{\text{th}}$ element.
We also write $v_i = \sum_{j=1}^k v_{i,j}e_j$.
Moreover, for $i \neq j$, we have $\|v_i - v_j\|_1 \geq 0.1 k$~\cite{Horadam07}.

Define $P= u$, and for $m \in [M]$, define $P_m = u + \epsilon (v_{m+1}/k)$.
Note that the $P_m$'s are valid distributions.
For notational purposes, we will also use $P_m(i)$ to denote the probability of element $i \in [k]$ under $P_m$, i.e., $P_m(i) = \frac{1}{k}(1  + \epsilon v_{m,i})$.
 Using the lower bound on $\|v_i - v_j\|_1$, we have $\dtv(P, P_i) = \epsilon/2$ and $\dtv(P_i,P_j) = 0.5 \epsilon \|v_i - v_j\|_1/k \geq 0.01 \epsilon$.

We now calculate $\chi_P(P_u,P_v)$: For $a \neq b$, we have
\begin{align*}
\chi_P(P_a,P_b) &= \left| \sum_{i=1}^k kP_a(i)P_b(i) - 1 \right| \\
&= \left| \sum_{i=1}^k k \frac{1}{k}   \left( 1 + \epsilon v_{a,i} \right) \frac{1}{k}\left( 1 + \epsilon v_{b,i} \right) - 1 \right| \\
&= \left| \sum_{i=1}^k \frac{1}{k}   \left( 1 + \epsilon v_{a,i} + \epsilon v_{b,i} +  \epsilon^2 v_{a,i}  v_{b,i}  \right) - 1 \right| \\
&= 0,
\end{align*}
where use the facts that $\sum_{i}v_{m,i} = 0$ for all $m$, and $\sum_i v_{a,i}v_{b,i}= 0$ for all $a \neq b$. We now consider the setting where $P_a = P_b$:
\begin{align*}
\chi_P(P_a,P_a) &= \sum_{i=1}^k kP_a^2(i) - 1 
= \sum_{i=1}^k k \frac{1}{k^2}   \left( 1 + \epsilon v_{a,i} \right)^2  - 1
= \sum_{i=1}^k \frac{1}{k}   \left( 1 + 2\epsilon v_{a,i} +  \epsilon^2 v_{a,i}^2  \right) - 1  
= \epsilon^2,
\end{align*}
where use the facts that $\sum_{i}v_{a,i} = 0$ and $|v_{a,j}| = 1$ for all $a$ and $j$.
Overall, we obtain the following bound on the average correlation for any subset $\cP' \subseteq \cP_M$:
\begin{align*}
\rho(\cP',P) = \frac{1}{|\cP'|^2}\sum_{P_1,P_2 \in \cP'} \chi_P(P_1, P_2) = \frac{\epsilon^2}{|\cP'|}.
\end{align*}
Thus, we have $\SDA(\cP_M,P, \gamma) \geq M \bar{\gamma}/ \epsilon^2$ for any $\bar{\gamma} \leq \epsilon^2$.
\end{proof}

\begin{corollary} 
\label{cor:SQMaryHypTesting}
Consider any $\epsilon \in (0,1)$, $D \in \N$, and $M \in \N$ such that $M\gtrsim  \frac{\log D}{\epsilon D}$. Let $\cP_M$ and $P$ be as defined in \cref{lem:SQHardDist}. Then the following hold:
\begin{enumerate}
   \item $\nstar(\cP_M \cup \{P\} ) \lesssim \frac{\log M}{\epsilon^2}$.

\item (Blackboard communication model.) Consider the blackboard communication model with $m$ machines, each with an i.i.d.\ sample from $Q$ (belonging to $\cP_M$ or  $P$) and $\log D$ bits. Any (randomized) algorithm that solves $\cB(\cP_M,P)$ with success probability $9/10$  requires  $$m \gtrsim \frac{M^{1/3}}{\epsilon^{2/3} D^{2/3} (\log D)^{1/3}}.$$
\end{enumerate}
\end{corollary}

\begin{proof}
The bound on $\nstar(\cP_M \cup P)$ follows from \Cref{fact:testing} and the fact that the distributions are separated in total variation distance.

We now turn our attention to the lower bound. Fix any $\bar \gamma$ such that $\bar \gamma \leq \epsilon^2$. 
\Cref{lem:SQHardDist} implies that the SDA of this decision problem, denoted by $d$, is at least $\frac{M \bar{\gamma}}{\epsilon^2}$.
Thus, \cref{lem:sqAdaptiveLow} states that $m \gtrsim \min \left\{\frac{1}{D \sqrt{\bar{\gamma}  }} , \frac{d}{\log D} \right\} \gtrsim \min \left\{\frac{1}{D \sqrt{\bar{\gamma}  }} , \frac{M\bar{\gamma}}{\epsilon^2\log D} \right\}$.
Taking $\bar{\gamma} =  \left( \frac{\epsilon^2 \log D}{D M} \right)^{2/3}  $, which satisfies ${\bar \gamma} \leq \epsilon^2$,
we have $m \gtrsim \frac{M^{1/3}}{\epsilon^{2/3} D^{2/3} (\log D)^{1/3}}$.
This completes the proof of \cref{cor:SQMaryHypTesting} and the second claim in \cref{thm:Adaptive-lb}.
\end{proof}

\begin{remark}\label{rem:AdapLowBounds} Note that \Cref{lem:SQHardDist} also implies a lower bound of $\Omega\left(\frac{\sqrt{M}}{\epsilon D}\right)$ for the special case of sequentially-adaptive algorithms, by using the lower bound for the $\MSAMPLE(D)$ oracle in \Cref{thm:SqSTATVSTAT}.
\end{remark}

\subsection{Lower bounds from impossibility of $\ell_1$-embedding} %
\label{app:lower_bounds_from_impossibility_of_l1}

The main result of this section is \cref{thm:lowerbound-l1-impossibility}.
Before that, we first state the following technical lemma, adapted from Lee, Mendel, and Naor~\cite[Lemma 3.1]{LeeMN05}
   (also see Charikar and Sahai~\cite{ChaSah02}):
\begin{lemma} 
\label{lem:L1Embedding}
Let $r \in \N$. For any $M = 2^{r}$ and $\epsilon \in (0, 1)$, there exists a set of distributions $\cP = \{P, P_1,\dots, P_M\} \subseteq \Delta_M$ such that for any $D \in \N$ and $\bT \in \cT_D$, we have
\begin{align*}
\frac{1}{M}\sum_{i=1}^M \dtv(\bT P_i,\bT P) \leq \frac{\epsilon \sqrt{D}}{\sqrt{M}},
\end{align*}
and for any distinct $p,q \in \cP$, we have $\dtv(p,q) \gtrsim \epsilon$.
\end{lemma}
\begin{proof}
Let $\{P,P_1,\dots,P_M\}$ be the distributions from \Cref{lem:SQHardDist}.
We follow the proof strategy in Lee, Mendel, and Naor~\cite[Lemma 3.1]{LeeMN05}. 
We begin by writing
\begin{align*}
 \sum_{i=1}^M \|\bT(P_i - P)\|_2^2 &= \epsilon^2 \sum_{i=1}^M \left\|\bT \frac{v_i}{M}\right\|_2^2  &&{\text{(by definition of the $P_i$'s)}}\\
 &= \frac{\epsilon^2}{M^2}\sum_{i=1}^M \left\| \sum_{j=1}^M v_{i,j} \bT e_j   \right\|_2^2 \\
 & = \frac{\epsilon^2}{M^2}\sum_{i=1}^M\sum_{j=1}^M\sum_{l=1}^M \left\langle  v_{i,j} \bT e_j, v_{i,l} \bT e_l\right\rangle  \\
 &= \frac{\epsilon^2}{M^2}\sum_{j=1}^M\sum_{l=1}^M \left\langle   \bT e_j,  \bT e_l\right\rangle \left( \sum_{i=1}^M v_{i,j} v_{i,l} \right) \\
 &= \frac{\epsilon^2}{M^2}\sum_{j=1}^M\sum_{l=1}^M \left\langle   \bT e_j,  \bT e_l\right\rangle  \left\langle v_j,v_l \right\rangle && \text{(using symmetry of $V$)}\\
 &= \frac{\epsilon^2}{M^2}\sum_{j=1}^M \|   \bT e_j\|_2^2 \|v_j\|^2_2 \\
 & = \frac{\epsilon^2}{M}\sum_{j=1}^M \|   \bT e_j\|_2^2 && \text{(using $V^\top V= kI$ and $k=M$)}\\
 &\leq \frac{\epsilon^2}{M}\sum_{j=1}^M \|   \bT e_j\|_1^2  && \text{(using $\|x\|_2 \leq \|x\|_1$)}\\
 &= \epsilon^2,
 \end{align*}
 where the last equality uses the fact that $\bT e_j \in \Delta_{D}$.
Applying Cauchy-Schwarz and using the fact that $\|x\|_1 \leq \sqrt{D} \|x\|_2$ for $x \in \Delta_D$, we obtain
\begin{align*}
\frac{1}{M} \sum_{i=1}^M \dtv(\bT P_i, \bT P) &\leq \sqrt{\frac{1}{M} \sum_{i=1}^M (\dtv(\bT P_i, \bT P))^2}\\
&\leq \sqrt{\frac{1}{M} \sum_{i=1}^M \|\bT(P_i - P)\|_1^2} \\
&\leq \sqrt{\frac{D}{M} \sum_{i=1}^M \|\bT(P_i - P)\|_2^2} \\
&\leq  \frac{\epsilon \sqrt{D}}{\sqrt{M}}.
\end{align*}
\end{proof}

We are now ready to prove the $\Omega(\sqrt{M})$ lower bound for non-adaptive, non-identical channels:
\begin{theorem}
\label{thm:lowerbound-l1-impossibility}
There exists a set $\cP = \{P,P_1,\dots,P_M\} \subseteq \Delta_M$ such that the following hold:
\begin{enumerate}
	\item $\nstar(\cP) \lesssim \frac{\log M}{\epsilon^2}$, and
	\item $\nnid(\cP,\cT_D) \gtrsim \frac{\sqrt{M/D}}{\epsilon}$.
\end{enumerate}
\end{theorem}
\begin{proof}
We will assume that $M = 2^r$ for some $r \in \N$.
Let $\cP = \{P,P_1,\dots,P_M\}$ be the set of distributions from \Cref{lem:L1Embedding}.
The upper bound on $\nstar(\cP)$ follows by the lower bound on the pairwise total variation distance and \Cref{fact:testing}.  We now turn our attention to the lower bound.

Fix any arbitrary choice of different $\{\bT_1,\dots,\bT_N\}$, where $\bT_i$ is the channel used by the $i^{\text{th}}$ user.
We use the following series of inequalities to upper-bound the minimum separation in total variation distance, for any choice of $\bT_1,\dots,\bT_N$:
\begin{align*}
\max_{\bT_l : l \in [N]} \min_{p\neq q  \in \cP} &\dtv\left(  \prod_{l=1}^N \bT_l p, \prod_{l=1}^N \bT_l q \right) \\
&\leq
 \max_{\bT_l : l \in [N]} \min_{i \in [M]} \dtv\left(  \prod_{l=1}^N \bT_l P_i, \prod_{l=1}^N \bT_l P\right) \\
&\leq \max_{\bT_l : l \in [N]} \sum_{i=1}^M \frac{1}{M} \dtv\left(  \prod_{l=1}^N \bT_l P_i, \prod_{l=1}^N \bT_l P\right)&& \text{(minimum is less than the average)}\\
&\leq \max_{\bT_l : l \in [N]} \sum_{i=1}^M \frac{1}{M} \sum_{r=1}^N \dtv\left(   \bT_l P_i,  \bT_l P\right)&& \text{(subadditivity of $\dtv$ (\Cref{fact:div}))}\\
&= \max_{\bT_l : l \in [N]}\sum_{r=1}^N  \sum_{i=1}^M \frac{1}{M}  \dtv\left(   \bT_l P_i,  \bT_l P\right) && \text{(exchanging the sum)}\\
&= N \max_{\bT \in \cT_D }  \sum_{i=2}^M \frac{1}{M}  \dtv\left(   \bT P_i,  \bT P\right) && \text{($N$ independent optimization problems)}\\
&\leq \frac{N \epsilon \sqrt{D}}{\sqrt{M}} && \text{(using \cref{lem:L1Embedding}).}
\end{align*}
Thus, if $N \lesssim \frac{1}{\epsilon}\sqrt{\frac{M}{D}}$, for any choice of $N$ channels, there exist $P', P \in \cP$ such that the total variation distance between the resulting product distributions is at most $0.001$. Consequently, \Cref{fact:testing} implies that there exists no test with probability of success in distinguishing between $P$ and $P'$ more than $0.95$ (say). Hence, one must have $N \gtrsim \frac{1}{\epsilon}\sqrt{\frac{M}{D}}$ for any successful test. Since this holds for an arbitrary choice of channels, we have the desired lower bound on $\nnid(\cP,\cT_D)$.
\end{proof}

\section{Auxiliary details} %
\label{app:technical_details}

We first mention that the class of well-behaved $f$-divergences includes various well-known $f$-divergences:
\begin{claim}[Examples of well-behaved $f$-divergences]
\label{claim:wellBehExamples}
 The following are examples of well-behaved $f$-divergences (cf.\ \Cref{def:well-behaved}):
\begin{enumerate}
  \item(Hellinger distance) $f(x) = (\sqrt{x} - 1)^2$ with $\kappa = 1$, $C_1 = 2^{-3.5}$, $C_2 = 1 $, and $\alpha = 2$. 
  \item(Total variation distance) $f(x) = 0.5 |x - 1|$ with $\kappa> 0$, $C_1 = 0.5$, $C_2 = 0.5$, and $\alpha = 1$.
  \item(Symmetrized KL-divergence) $f(x) = x \log x - \log x$ with $\kappa = 1$, $C_1 = 0.5$, $C_2 = 1$, and $\alpha= 2$. 
  \item(Triangular discrimination) $f(x) = \frac{(x-1)^2}{1+x}$ with $\kappa = 1$, $C_1 = 1/3$, $C_2 = 1/2$, and $\alpha = 2$.
  \item(Symmetrized $\chi^s$-divergence) For $s \geq 1$, $f(x) = |x-1|^s + x^{1-s}|x-1|^s$ with $\kappa = 1$, $C_1 = 1$, $C_2 = 3$, and $\alpha = s$.\footnote{The usual $\chi^s$-divergence corresponds to $f(x) = |x-1|^s$, for $s \geq 1$~\cite{Sason18}. We consider the symmeterized version with $\tilde{f}(x) = f(x) + xf(1/x)$.} 
\end{enumerate}
\end{claim}
\begin{proof}
It is easy to see that these functions are non-negative, convex, and satisfy the symmetry property of \cref{def:well-behaved}.
In the remainder of the proof, we outline how they satisfy the property \ref{item:quad}.
\begin{enumerate}
  \item
We will show that we can take $\kappa = 1$, $C_1 = 2^{-3.5}$, $C_2 = 1$, and $\alpha= 2$.
 The upper bound $f(1+x) = (\sqrt{1 +x } - 1)^2 \leq x^2$
follows by noting that $\sqrt{1 + x} \leq 1 + x$ for any $x \geq 0$. For the lower bound, we define $g(x) := f(1+x) - C_1 x^2$. Note that $g(0) = 0$, $g'(x) = 1 - (1 + x)^{-0.5} - 2C_1x$, and $g''(x) = 0.5 (1+x)^{-1.5} -2C_1$.
We note that $g'(0) = 0$ and $g''(x) \geq g''(1) = 2^{-2.5} - 2C_1$ for all $x \in [0,1]$. Thus, $g''(x) \geq 0 $ for $x \in [0,1]$, so $g(x)$ is also nonnegative on $x \in [0,1]$.
\item The result follows by noting that for $x \geq 0$, we have $f(1+x) = x$.
\item We have $f(1+x) = x \log (1+ x)$. 
We use the fact that $\frac{x}{1 +x} \leq \log(1 + x) \leq x$ for $x \geq 0$.
This directly gives us $f(1 + x) = x \log(1 + x) \leq x^2$.
The lower bound follows by noting that $\log(1 + x) \geq \frac{x}{2}$ for $x \in [0,1]$, so $f(1+x) \geq x^2/2$.
\item We have $f(1+x) = x^2/(2+x)$, which lies between $x^2/3$ and $x^2/2$, for $x \in [0,1]$.
\item We have $f(1+x) = |x|^s(1 + (1+x)^{1-s})$, which is larger than $x^s$ and less than $3x^s$ for $x \in [0,1]$.
 \end{enumerate} 
\end{proof}
Finally, we mention the following approximation for the Hellinger distance between two Bernoulli distributions that was used earlier:
\begin{claim}[Approximation for Hellinger distance]
\label{prop:simpleHellinger}
For $0 \leq p \leq p' \leq 1/2$, we have the following:
\begin{align*}
\sqrt{p'}-\sqrt{p} \leq \sqrt{ \left(\sqrt{p} - \sqrt{p'}\right)^2 + \left(\sqrt{1-p} - \sqrt{1- p'}\right)^2} \leq \sqrt{2}\left(\sqrt{p'}-\sqrt{p}\right).
\end{align*}
\end{claim}
\begin{proof}
The first inequality follows by the non-negativity of the term $\left(\sqrt{1-p} - \sqrt{1- p'}\right)^2$.
To prove the second inequality, it suffices to show that $\left(\sqrt{1-p} - \sqrt{1- p'}\right)^2 \leq \left(\sqrt{p'} - \sqrt{p}\right)^2$, which is equivalent to showing that $\sqrt{1-p} - \sqrt{1-p'} \leq \sqrt{p'} - \sqrt{p}$.
For $z \in [0,0.5]$, define $f(z) := \sqrt{z} + \sqrt{1 -z}$.
The desired inequality is then equivalent to showing that $f(p) \leq f(p')$, which follows if $ f'(z) \geq  0$ for $z \in [0,0.5]$. Calculating the derivative, we obtain
\begin{align*}
f'(z) = \frac{1}{2 \sqrt{z}} - \frac{1}{2 \sqrt{1 -z} } = \frac{1}{2} \frac{\sqrt{1 - z} - \sqrt{z}} {\sqrt{z(1-z)}} \geq 0,
\end{align*}
since $z \in [0,0.5]$.
\end{proof}

\end{document}